\documentclass{article}

\usepackage[margin=1in]{geometry}
\usepackage{microtype}
\usepackage{placeins}
\usepackage{graphicx}
\usepackage{booktabs}
\usepackage{amsmath,amssymb,amsfonts,mathtools,amsthm}
\usepackage{tabularx,array,makecell,multirow}
\usepackage{algorithm}
\usepackage{algorithmic}
\usepackage{url}
\usepackage{xurl}
\usepackage{xcolor}
\usepackage[colorlinks=true,
            citecolor=teal,
            linkcolor=blue!60!black,
            urlcolor=blue!60!black]{hyperref}
\usepackage[capitalize,noabbrev]{cleveref}
\usepackage{natbib}
\usepackage[]{authblk}
\renewcommand{\Affilfont}{\normalsize}

\allowdisplaybreaks
\setcitestyle{authoryear,open={(},close={)}}

\theoremstyle{plain}
\newtheorem{proposition}{Proposition}[section]

\newtheorem{assumption}[proposition]{Assumption}

\theoremstyle{remark}

\crefname{assumption}{Assumption}{Assumptions}
\Crefname{assumption}{Assumption}{Assumptions}
\crefname{corollary}{Corollary}{Corollaries}
\Crefname{corollary}{Corollary}{Corollaries}
\theoremstyle{definition}

\crefname{problem}{Problem}{Problems}
\Crefname{problem}{Problem}{Problems}

\newcommand{\Mset}{\mathcal{M}}
\newcommand{\Iset}{\mathcal{I}}
\newcommand{\Bset}{\mathcal{B}}

\newcommand{\Kset}{\mathcal{K}}
\newcommand{\Xset}{\mathcal{X}}
\newcommand{\Pset}{\mathcal{P}}
\newcommand{\Tset}{\mathcal{T}}
\newcommand{\Lset}{\mathcal{L}}
\DeclareMathOperator*{\argmin}{arg\,min}
\newcommand{\edge}{e}
\newcommand{\Eset}{\mathcal{E}}

\newcommand{\gap}{s}
\newcommand{\dual}{\lambda}
\newcommand{\rc}{\mathrm{RC}}

\newcommand{\Qval}{Q}
\newcommand{\kk}{k}
\newcommand{\inj}{z}
\newcommand{\own}{\kappa}
\newcommand{\Fval}{F}
\newcommand{\scaled}{w}
\newcommand{\brief}{\mathsf{brief}}

\newcommand{\cmark}{\ensuremath{\checkmark}}

\begin{document}

\title{\LARGE Integrating Learning and Optimization for\\
Open Transport Capacity Pooling under Disruption}

\date{}
\setlength{\affilsep}{2em}
\author[*1,2]{Yujia Xu}
\author[3,1]{Walid Klibi}
\author[1,2]{Benoit Montreuil}
\renewcommand{\baselinestretch}{1.2}
\affil[1]{Physical Internet Center, Supply Chain \& Logistics Institute, \protect\\ Georgia Institute of Technology, Atlanta GA, USA}
\affil[2]{\Affilfont School of Industrial \& Systems Engineering, Georgia Institute of Technology, Atlanta GA, USA}
\affil[3]{\Affilfont Centre of Excellence for Supply Chain Innovation \& Transportation (CESIT) \protect\\
Kedge Business School, Bordeaux, France}
\affil[*]{\textit{Corresponding author: yujiaxu.a@gmail.com}}

\renewcommand{\baselinestretch}{1.2}
\maketitle
\renewcommand{\baselinestretch}{1}

\begin{abstract}
Disruptions can leave logistics networks short of capacity, yet outside
capacity is hard to use: incumbent models are private, provider records
are incomplete, and offers stay hidden until a costly request for
quotation (RFQ). We introduce open capacity pooling as a budgeted
provider-discovery problem that links which capacity need to pursue,
whom to contact, and which verified offer to admit. The alternating
direction method of multipliers (ADMM) coordinates incumbents without
sharing models, and the resulting capacity gaps direct discovery. A
machine-learning layer fitted before search ranks providers by predicted
capability and a fleet-size proxy from directory records and large language
model (LLM) readings of self-descriptions. Verification reveals offers,
reduced-cost screening admits them, and re-optimization updates the gaps
and prices. Under fixed offers and exact coordination, a full-revelation
reference reaches the finite-universe optimum, while budgeted discovery
remains heuristic. We evaluate the mechanism on 4,509 real firms in
three US regions excluded from predictor training, using scenarios
calibrated to seven historical disruptions and simulated availability,
offered quantities, and prices. With 80 RFQs, about 1\% of exhaustive
search, the mechanism recovers 19--43\% of the full-information value of
outside transport, against 2--6\% for random search, and saves
US\$73,000--104,000 per event before search and contracting costs. At
tight RFQ budgets, gap-guided need selection and learned ranking are
complements in these testbeds. LLM reading helps rank capabilities that
records omit, such as refrigerated service, but its operating value
varies by region. The results support combining record-based learning
with coordination-derived gaps under limited search budgets.
\end{abstract}


\section{Introduction}\label{sec:intro}

Disruptions rarely remove logistics capacity uniformly
\citep{snyder2016ormsmodels,ivanov2017literature}. A hurricane,
winter storm, or fuel shortage
\citep[e.g.,][]{beven2025beryl,ferc2021coldweather,ajc2021colonial}
creates shortfalls at particular
facilities, on particular lanes, in particular time windows, and for
particular services, such as refrigerated carriage into a metropolitan
area that has lost power. Aggregate logistics capacity may remain
available elsewhere in the region \citep{kulpa2026synthetic},
although its suitability and contractibility for a particular network
are unknown \citep{scott2017service,acocella2020elephants}.
Temporary access to outside storage and transport can absorb such local
imbalances \citep{guo2018capacity,ceschia2023ondemand}, but some of the
potential providers have no established relationship with the disrupted
network. They may be enumerated from public directories and registries
\citep{sba2026sbs,fmcsa2026census}, yet their capacity remains
operationally inaccessible until the network identifies a relevant firm,
requests an offer, verifies what the firm can do, and fits the offer
into its operating plan. We use \emph{open capacity pooling} to denote the
expansion of an incumbent network's usable capacity through the search
for, verification, and admission of offers from previously unengaged
providers.

Existing research provides building blocks for this process.
Capacity-pooling and Physical Internet models allocate shared and
interoperable resources among members known before operations begin
\citep{jordan1995chaining,montreuil2011physical,xu2026dynamic}, and
distributed optimization and platform-based collaboration coordinate
privately held models or information among a fixed set of participants
\citep{boyd2011admm,maggiar2026consensus,los2020value}. The platform and
crowdsourcing studies considered here admit outside capacity from
registered providers with known capabilities, focusing on uncertainty in
their availability or acceptance
\citep{arslan2019crowdsourced,beirigo2022learning,he2024column,akbarpour2026choice}.
Learning can improve pricing, search, or the parameters passed to a
logistics model \citep{lu2025boosting,yan2023extended,wang2026spovcs}.
These streams leave open how to jointly choose which shortfall to
pursue, which previously unengaged provider to contact, and which
revealed offer to admit when contractible capacity must be discovered
under a limited contact budget.

The information structure is the central difficulty. The candidate
directory is observable, but contractible offers are latent. Provider
records provide uneven detail on capabilities, while event-time
availability and binding quantity and price terms are revealed through a
request for quotation (RFQ)
\citep{caplice2007electronic,acocella2022setit}. An RFQ is
therefore both a sourcing action and a costly information-acquisition
action: it determines which offer enters the observed optimization
instance. Search theory gives optimal inspection rules when each
alternative has a known and independent value distribution
\citep{weitzman1979optimal}. Here the value of revealing an offer is
endogenous. It depends on the gaps that incumbent firms cannot close
among themselves, and an admission can change which gaps remain. The
learning and optimization layers are thus coupled in operation: a
predictor trained before search ranks providers, and
re-optimization can change the need for which they are ranked. The
predictor is not trained end to end on the decision loss, as in
decision-focused learning \citep{elmachtoub2022smart}. We ask how a
network can coordinate need selection, provider contact, and offer
admission under a limited RFQ budget, and which pre-contact information
most affects the operating value recovered.

Incumbent firms first coordinate through exchange-ADMM, the exchange
form of the alternating direction method of multipliers
\citep{boyd2011admm}, without sharing their operating models.
Scale-normalized capacity gaps identify the need to pursue, and a
matching layer ranks providers for that need by predicted capability and
a fleet-size proxy. The matching layer combines directory fields with a
large language model (LLM) reading of the free-text descriptions that
some firms publish \citep{hegselmann2023tabllm,kim2026generative}. The
LLM reads records but makes no sourcing or admission decision. An RFQ
and verification then reveal the provider's offer terms, which are
valued at the current coordination-price estimates to determine
admission. This step resembles reduced-cost pricing in column generation
\citep{lubbecke2005column}, but a provider must be contacted before the
quote's reduced cost can be evaluated. Quotes obtained from the same
provider share its offered capacity limits, so repeated contacts cannot
multiply the capacity available from that provider. Re-coordination of
the expanded pool updates the prices that screen offers and the
gaps that direct the next request. Prediction errors can waste the RFQ
budget and reduce the final operating value, but they cannot bypass
verification or admit an offer that fails the price screen. The
empirical evaluation focuses on transport capacity and compares
directory-based discovery with registry-enriched discovery. The latter
uses linked capability and fleet data when available. Both still
require RFQs to reveal event-time availability and binding offer terms.
The testbeds combine real provider records and historically calibrated
shocks with simulated availability, offered quantities, and prices.

This study makes three contributions.
\begin{itemize}
\item We introduce open capacity pooling as a sequential decision
problem in which the set of operational alternatives is itself
discovered under an RFQ budget. Learning ranks previously unengaged
providers from incomplete records, while optimization
identifies the capacity needs that guide search and values the
offers revealed through contact.

\item We develop this closed-loop, coordination-guided discovery
mechanism and establish its guarantees. Its matching layer predicts
nested capability probabilities and a fleet-size proxy from directory fields and
LLM-read self-descriptions and is frozen before search. It is trained on
7,914 linked firms and tested on 8,335 linked firms in three held-out regions.
Exact coordination prices provide supporting values for capacity, and
admitting a verified offer, which may go unused, never raises the exact
operating cost when no fixed admission charge is imposed. With fixed
offers, sound verification, exact coordination, and full revelation of a
finite provider universe, the reference protocol reaches the optimum over all eligible
offer blocks. Under a binding RFQ budget, discovery is a heuristic
without an optimality claim.

\item We evaluate the mechanism in three data-calibrated regional
networks facing scenarios calibrated to seven historical disruptions.
Their provider pools hold 4,509 of these test-region firms, and
availability, offered quantities, and prices are simulated. With 80
RFQs, about 1\% of exhaustive search, it recovers 19--43\% of the
full-information value of outside transport, against 2--6\% under random
need and provider choice. Two supply chain insights follow. First,
coordination-derived scarcity information and external provider
intelligence are complements under tight budgets, and at the
tightest budgets better provider information makes gap guidance more
valuable. Second, provider data
and outreach are partial substitutes: 80 RFQs with registry data recover
about as much as 480 with learned predictions, although outreach still
matters where outside supply is thin.
\end{itemize}

\Cref{sec:related} reviews the literature. \Cref{sec:setting} formulates
fixed-network capacity pooling and its distributed coordination,
\cref{sec:method} develops learning-guided discovery and admission, and
\cref{sec:experiments} reports the experiments. \Cref{sec:conclusion}
discusses implications and limitations and concludes. The appendices
contain illustrative private firm models, the proofs and implementation
details of coordination and discovery, the data and testbed protocol,
numerical validation, and supplementary learning results.

\section{Related Literature and Research Gap}\label{sec:related}

The five groups of studies compared in \cref{tab:positioning}, one per
subsection below, supply the setting of the mechanism (disruption
sourcing, capacity pooling, and the Physical Internet, or PI), its
coordination and admission steps, and its search and learning steps.
Each group provides part of the mechanism. None of the works reviewed
chooses jointly which shortfall to pursue, which previously unengaged
provider to contact, and which revealed offer to admit, when offers must
first be discovered through costly contacts and the incumbent firms keep
their models private.

\begin{table}[!htbp]
\centering
\caption{Positioning relative to adjacent literature.}
\label{tab:positioning}
\small
\setlength{\tabcolsep}{4pt}
\renewcommand{\arraystretch}{1.25}
\begin{tabularx}{\textwidth}{@{}
>{\raggedright\arraybackslash}X
>{\centering\arraybackslash}p{1.9cm}
>{\centering\arraybackslash}p{1.35cm}
>{\centering\arraybackslash}p{1.75cm}
>{\centering\arraybackslash}p{1.55cm}
>{\centering\arraybackslash}p{1.6cm}@{}}
\toprule
\textbf{Literature group and representative studies} &
\textbf{Open membership} &
\textbf{Local models} &
\textbf{Pre-contact prediction} &
\textbf{Gap-directed search} &
\textbf{Verified admission} \\
\midrule
Disruption sourcing, capacity pooling, and PI\newline
{\footnotesize\citep{tomlin2006value,hu2024machine,jordan1995chaining,sohrabi2011open,xu2026dynamic}}
& (\cmark) & (\cmark) & \textemdash{} & \textemdash{} & (\cmark) \\
\addlinespace[3pt]
Distributed coordination\newline
{\footnotesize\citep{nedic2009distributed,boyd2011admm,aybat2019distributed,maggiar2026consensus}}
& \textemdash{} & \cmark & \textemdash{} & (\cmark) & \textemdash{} \\
\addlinespace[3pt]
Column generation and learning-guided pricing\newline
{\footnotesize\citep{lubbecke2005column,morabit2021machine,tahir2021improved,lu2025boosting}}
& \textemdash{} & \textemdash{} & (\cmark) & (\cmark) & (\cmark) \\
\addlinespace[3pt]
External capacity and costly search\newline
{\footnotesize\citep{scott2017service,arslan2019crowdsourced,beirigo2022learning,ausseil2024online,akbarpour2026choice,weitzman1979optimal}}
& (\cmark) & \textemdash{} & (\cmark) & (\cmark) & \textemdash{} \\
\addlinespace[3pt]
Prediction for optimization and provider records\newline
{\footnotesize\citep{elmachtoub2022smart,yan2023extended,wang2026spovcs,coscrato2024supplier,li2025supplier,kim2026generative}}
& (\cmark) & \textemdash{} & \cmark & \textemdash{} & \textemdash{} \\
\midrule
\textbf{This paper}
& \cmark & \cmark & \cmark & \cmark & \cmark \\
\bottomrule
\end{tabularx}
\par\smallskip
{\footnotesize\raggedright
\cmark{} = provided; (\cmark) = partial or analogous; \textemdash{} = not
addressed. Marks describe each group as a whole, and different marks in a
row may come from different studies. Open membership: non-incumbent
providers can join at run time.
Local models: firm models are not centrally assembled (not cryptographic
privacy). Pre-contact prediction: a model scores candidates from their
records before any engagement. Gap-directed search: coordination gaps
select the need queried. Verified admission: admission follows
verification, an RFQ, and reduced-cost screening.\par}
\end{table}

\subsection{From disruption sourcing and fixed pools to open supply webs}
\label{subsec:rw-pooling}

Disruption sourcing is an established problem in supply chain and
logistics research \citep{snyder2016ormsmodels,ivanov2017literature}.
Supplier selection and order allocation models choose among candidate
suppliers whose capacities, costs, and disruption probabilities are
modeled in advance \citep{sawik2013integrated}, and contingency studies
value backup suppliers, dual sourcing, spot purchasing, and flexible
capacity as responses to supply failure
\citep{tomlin2006value,namdar2018supply}. Learning also enters this
stream, for example in the selection of relief suppliers under option
contracts \citep{hu2024machine}. Our setting shares the objective, but
each provider's eligibility for a specific need and its binding offer
are latent until engagement. Directory records can therefore only rank
providers before a budgeted contact, and an offer is
valued by the coordination prices of a multi-firm pool rather than by a
single buyer's cost.

Capacity-pooling research explains why sharing resources is valuable,
but it usually takes the membership of the pool as given. A limited set
of well-chosen flexibility links can capture much of the value of full
flexibility \citep{jordan1995chaining}. The PI and the Open Supply Web
replace closed private networks with interoperable webs of independently
operated warehouses, hubs, and transport services
\citep{montreuil2011physical,sohrabi2011open}, and interconnected PI
services can mitigate supply chain disruptions \citep{yang2017mitigating}.
Operational PI studies bring these decisions closer to execution, for
example by deploying relocatable storage modules
\citep{faugere2022dynamic}, relocating mobile production containers
\citep{liu2024containerized}, planning hub throughput capacity
\citep{liu2025dynamic}, or pooling heterogeneous human--robot resources
\citep{xu2026dynamic}. In all of these studies the resources and the
participating network are known before allocation begins.

\subsection{Distributed coordination and economically meaningful duals}
\label{subsec:rw-distributed}

Distributed optimization decomposes a system objective into private local
subproblems connected by message passing
\citep{nedic2009distributed,boyd2011admm}. The exchange form of ADMM
exposes the multiplier of a resource-clearing constraint as a price
\citep{boyd2011admm}, which is the basis of our capacity exchange.
\citet{aybat2019distributed} establish convergence to a consensual dual
price over time-varying networks, and \citet{maggiar2026consensus}
coordinate heterogeneous planning agents around shared decisions.
Supply-chain research pursues the same goal through collaborative
planning without full disclosure of cost structures
\citep{stadtler2009framework}. On collaborative-routing platforms,
\citet{los2020value} show that the value of each kind of shared carrier
information trades off against confidentiality, a trade-off that our
firms face as well.

All of these methods coordinate a fixed set of participants, and our
contribution is not a new ADMM variant. We reuse two coordination outputs
for distinct tasks. The scale-normalized recourse gap identifies where the
incumbent pool still depends on fallback before any contact, and the
shadow price values an offer block only after verification and an RFQ
reveal its terms.

\subsection{Column generation and reduced-cost pricing}
\label{subsec:rw-cg}

Once a new provider's offer block is represented as a generalized supply
column, its admission is a classical reduced-cost decision. Dantzig--Wolfe
decomposition, column generation, and branch-and-price use the duals of a
restricted master to price omitted columns and add improving ones
\citep{dantzig1960decomposition,barnhart1998branchprice,lubbecke2005column},
so we claim no novelty for the pricing arithmetic. A growing literature
uses learning to speed up this loop. It selects which generated columns
enter the master \citep{morabit2021machine}, restricts pricing to
connections likely to appear in a near-optimal solution
\citep{tahir2021improved}, prunes the pricing graph with a graph neural
network \citep{lu2025boosting}, or guides rapid re-optimization by dual
values \citep{xue2026rapid}.

In these works the columns are generated algorithmically. A pricing
problem builds each column from the model data, so its cost and
coefficients are known when it is created, and learning predicts which
generated column will be useful. Provider offer blocks are not generated
in this way. The provider records are enumerable, but a block exists only
after an RFQ and verification reveal the provider's capabilities and
terms, so its reduced cost can be evaluated only after contact. A learned
capability probability can still rank an ineligible provider first, and
the RFQ spent on it reveals no block.

\subsection{External capacity and costly search}
\label{subsec:rw-external}

Freight research documents that access to outside capacity is uncertain
even under contract. Truckload contracts set prices without obliging
carriers to accept tendered loads, and rejected loads move to a spot
market whose rates average well above contract rates
\citep{scott2017service}. After major hurricane landfalls,
state-level long-haul dry-van spot prices and volumes showed no
significant change relative to synthetic controls
\citep{kulpa2026synthetic}. This is consistent with capacity existing in
aggregate during a regional disruption. Locating the part of it that can
serve a specific network need is the problem this paper addresses.

Platform studies model outside capacity whose participation is
uncertain. They match delivery tasks to ad hoc drivers backed by a fleet
\citep{arslan2019crowdsourced}, schedule courier shifts alongside ad hoc
couriers with a prescriptive machine-learning method
\citep{behrendt2023prescriptive}, offer drivers personalized request menus
\citep{horner2021optimizing}, approximate acceptance probabilities online
\citep{ausseil2024online}, embed a choice model of courier acceptance in
facility location and routing \citep{akbarpour2026choice}, and schedule
privately owned autonomous vehicles in their idle time slots
\citep{he2024column}. Closest to our valuation step,
\citet{beirigo2022learning} hire idle third-party vehicles on demand and
value them through the duals of the assignment problem. In these studies
the outside providers are platform participants whose capabilities are
known. In
our setting the providers have never interacted with the network, their
capability must be predicted from directory records, and each RFQ is a
metered action.

Search theory studies when to pay to inspect alternatives.
\citet{weitzman1979optimal} shows that an index policy is optimal when
each alternative has a known prize distribution and inspection cost and
prizes are independent. Our problem departs from these conditions. One
inquiry can reveal offers for several needs, the value of an offer depends
on the blocks already admitted, and the contact budget is fixed. We
therefore use a priority rule with explicit stopping conditions and
establish optimality only under full revelation of a finite provider
universe, rather than claim an optimal search policy.

\subsection{Prediction for optimization and learning from provider records}
\label{subsec:rw-pto}\label{subsec:rw-llm}

The standard paradigm estimates uncertain parameters and then optimizes
with the estimates. Decision-focused learning trains the predictor on the
loss of the decision it induces \citep{elmachtoub2022smart}.
Transportation research applies these ideas widely. Port-state-control
studies predict ship deficiencies and select the few ships to inspect
under an inspector budget, training the predictor with decision-aware
losses \citep{yan2020semi,yan2023extended}. This is the closest structural
analog to choosing whom to contact. Other studies feed probabilistic
berth-availability forecasts into risk-aware speed planning
\citep{sheng2026predict}, turn quantile forecasts into scenarios for
online inventory routing \citep{jia2025scenario}, or train relocation
decisions end to end through an ADMM-based optimization layer
\citep{wang2026spovcs}. In all of these studies the prediction concerns a
parameter of a known instance, and the candidates of the instance are
given. Here the prediction decides which candidates enter the instance at
all, the truth is revealed only for contacted providers, and the value of
a revealed offer comes from the coordination prices. We therefore train
the matching layer on registry labels and leave valuation to the
optimization layer. Training it on the decision loss would require
differentiating through a budgeted sequential search, which we do not
attempt.

Firm records also inform supplier choice. Machine learning on supplier
attributes recovers how buyers select suppliers on an online marketplace
for production capacity \citep{hu2025supplier}. The closest published pipeline
we found is that of \citet{kim2026generative}. Generative AI standardizes
free-text cargo descriptions, gradient-boosted models predict container
dwell time from the standardized records, and the predictions reduce yard
relocations. LLMs can also classify serialized tabular records zero-shot
and are competitive with gradient-boosted trees when only a few labels
exist \citep{hegselmann2023tabllm}. This motivates the two LLM roles of our matching layer: a zero-shot estimator when labels are scarce, and a
reader whose outputs feed a supervised classifier when labels exist.

LLM work at the interface with operations research often asks the model
instead to formulate or solve an optimization problem
\citep{huang2025orlm}, and transportation studies have used LLMs as
forecasters, for example of ship trajectories \citep{jiao2026llm4stp}.
Closer to our use, a supplier recommender ranks road-freight suppliers
for invitation to bid in online procurement \citep{coscrato2024supplier},
and an LLM with a knowledge graph built from supplier capability data
supports the discovery of manufacturing suppliers \citep{li2025supplier}.
These systems stop at a ranked or retrieved list. In our mechanism the
ranking serves an optimization that selects the need, values the
revealed offer, and decides its admission. The LLM neither formulates
nor solves the optimization. It reads provider records before search, and its outputs
only order contacts.

\section{Fixed-Network Capacity Pooling and Distributed Coordination}
\label{sec:setting}

This section answers a self-contained question: given a fixed set of
participating firms, what capacity-pooling problem is being solved, how
does each firm participate without disclosing its private model, and how
are the commodity-level shadow prices computed? \Cref{sec:method}
then embeds this inner coordination loop in an outer loop that expands the
pool. Three indices are kept separate throughout the main text:
$\kk$ indexes capacity commodities, $r$ indexes outer network-expansion
rounds, and $t$ indexes inner ADMM iterations within one round.

\subsection{Capacity commodities and firm participation}
\label{subsec:commodities}

We consider a regional logistics network of independently operated
firms that control warehouses, fulfillment hubs, cross-docks, and carrier
fleets. The setting follows the open supply webs of the Physical
Internet, in which such firms share standardized capacity across
organizational boundaries \citep{montreuil2011physical,sohrabi2011open}.
The participating firms form a fixed set $\Mset^0$ of incumbents. The
outer expansion mechanism keeps this set and adds verified
external capacity as offer-block agents $R^r$, giving the exchange-agent
set $\Iset^r=\{0\}\cup\Mset^0\cup R^r$, in which agent $0$ supplies
emergency recourse (\cref{subsec:benchmark}). We reserve $m$ for a firm
and $\edge$ for a provider-offer block, a generalized column in the
column-generation sense. \Cref{sec:method} defines admission to $R^r$.

Shareable capacity is standardized into a finite set of capacity
commodities $\Kset$. The ontology is fixed and system defined: openness
concerns network membership, not the discovery of new capacity types. Each commodity $\kk\in\Kset$ carries a public
attribute profile
\begin{equation}
\label{eq:commodity}
\alpha_\kk=\big(
\underbrace{\text{type}_\kk}_{\text{resource family}},\;
\underbrace{\ell_\kk}_{\text{service location}},\;
\underbrace{\omega_\kk}_{\text{time window}},\;
\underbrace{\chi_\kk}_{\text{quality class}}
\big).
\end{equation}
We call $\alpha_\kk$ the key of commodity $\kk$: two units of capacity are
interchangeable only if their keys coincide. For example, refrigerated
truck capacity on the lane from Atlanta to Tampa on a given day is one
commodity, and dry capacity on the same lane and day is another.

The resource family covers storage, handling, processing, dock
operations, and transport on directed lanes. The quality class $\chi_\kk$
records hard qualifications. The service location $\ell_\kk$ is not a provider's
physical address but where capacity must be operationally deliverable,
either a service region for facility capacity or an
origin--destination lane for transportation.
Transportation, repositioning, and distance costs required for delivery
remain in the firm's private model or an external offer's quoted price.

A firm's public coordination interface is a signed net injection
$\inj_m=(\inj_{m,\kk})_{\kk\in\Kset}$: positive entries supply
commodity $\kk$, negative entries receive it, and zero means no net
transaction. Its verified participation set
$\Kset_m=\{\kk\in\Kset:\inj_{m,\kk}\text{ may be nonzero}\}$ enforces
$\inj_{m,\kk}=0$ outside that set. Because $\alpha_\kk$ is already a
public system definition, an incumbent exposes only $\Kset_m$, assumed
verified at onboarding, and its injection responses.

\subsection{Private firm models}
\label{subsec:firms}

Each firm $m\in\Mset^0$ holds structural parameters $\Theta_m$, a
current operating state $d_m$, internal operating decisions $y_m$, a cost
function $C_m$, and a feasible set $\Xset_m(\Theta_m,d_m)$. The state
$d_m$ records the firm's demand and the capacity that remains available
after a disruption. These
objects may differ across firm types and remain behind the local
optimization interface. The induced cost-of-injection function is
\begin{equation}
\label{eq:valuefn}
\Fval_m(\inj_m;d_m)
=\min_{y_m}\Big\{C_m(y_m,\inj_m;\Theta_m):(y_m,\inj_m)
\in\Xset_m(\Theta_m,d_m)\Big\},
\end{equation}
with value $+\infty$ when no feasible operation supports
$\inj_m$. We assume the minimum is attained on the effective domain.

\Cref{app:private-models} gives two examples, a warehouse and a
cross-dock or lane carrier, which are also the member types of the
testbeds in \cref{subsec:testbeds}. In both, a common resource constraint
maps commodity injections onto the firm's physical resource pools, so
several service-location or time-window commodities can share one pool.
When a disruption raises demand or cuts capacity beyond what owned and
flexible capacity can cover, borrowing or emergency recourse becomes
cost-minimizing. \Cref{ass:regularity} below requires each induced function
$\Fval_m(\cdot;d_m)$ to be closed, proper, and convex, and
\cref{app:scope} verifies this for both examples.

\subsection{Centralized fixed-network benchmark}
\label{subsec:benchmark}

For each commodity $\kk$,
$\gap_\kk\ge0$ is emergency capacity procured from a public fallback
source and injected into the pool at cost $\Phi(\gap)$. It covers the
residual aggregate deficit after firm-to-firm pooling. It is neither
unserved demand nor disposal. We assume
$\Phi:\mathbb R_+^{|\Kset|}\to\mathbb R$ is finite, closed, convex,
and componentwise nondecreasing. The system-level benchmark jointly
chooses firm injections and emergency recourse:
\begin{equation}
\label{eq:pooling}
\Qval(\Mset^0)
=\min_{\{\inj_m\},\,\gap\ge0}
\left\{\sum_{m\in\Mset^0}\Fval_m(\inj_m;d_m)+\Phi(\gap):
\sum_{m\in\Mset^0}\inj_m+\gap=0\right\}.
\end{equation}

Because positive injections supply capacity and negative injections
receive it, the clearing equality sets
$\gap=-\sum_m\inj_m$. Hence nonnegative recourse can offset an
aggregate deficit, $\sum_m\inj_m\le0$, but cannot absorb aggregate
surplus. This commoditywise equality is the only cross-firm coupling.
Problem \eqref{eq:pooling} is the allocation and value target for
distributed coordination, without requiring a coordinator to assemble
the private objectives or feasible sets.

A linear fallback cost is $\Phi(\gap)=P^\top\gap$ with $P_\kk>0$. A
two-tier alternative supplies up to $\bar\gap_\kk$ units at
$P^1_\kk$ per unit and additional emergency capacity at the higher
rate $P^2_\kk$:
\begin{equation}
\label{eq:twotier}
\Phi(\gap)=\sum_{\kk\in\Kset}\left[
P^1_\kk\min\{\gap_\kk,\bar\gap_\kk\}
+P^2_\kk(\gap_\kk-\bar\gap_\kk)_+\right],
\qquad
0<P^1_\kk<P^2_\kk,\quad \bar\gap_\kk>0.
\end{equation}

Both tiers price the same commodity. They are not classes of providers. The two tiers have different meanings. The cheap tier stands for the
limited backup capacity that the network reaches through its usual
channels. The expensive tier is not a market price. It prices
last-resort coverage of a deficit that no known source can serve in
time, through expedited alternatives whose cost also reflects the delay,
and plays the role that the deprivation cost of late supplies plays in
relief operations \citep{holguinveras2013objective,perezrodriguez2016deprivation}.
Because capacity found during a disruption is usually bought near
prevailing rates (\cref{subsec:rw-external}), the expensive tier measures
the cost of not finding it.

\subsection{Exchange-ADMM coordination}
\label{subsec:exchange}

Problem \eqref{eq:pooling} and its offer-expanded round form are optimal
exchange problems \citep[\S7.3]{boyd2011admm}. The recourse agent $0$
has $q_0=\gap$ and
$\tilde{\Fval}_0(q_0)=\Phi(q_0)+I_{\mathbb R_+^{|\Kset|}}(q_0)$,
where $I$ is the convex indicator. Each incumbent firm
$m\in\Mset^0$ has $q_m=\inj_m$ and
$\tilde{\Fval}_m(q_m)=\Fval_m(q_m;d_m)$. Each admitted provider-offer
block $\edge\in R^r$ has a finite set of quoted keys
$\mathcal J_\edge\subseteq\Kset$, one injection direction
$a_{\edge j}=\theta_{\edge j}\mathbf 1_j$ per quoted key, where
$\mathbf 1_j$ is the indicator vector of key $j$, the matrix
$A_\edge=[a_{\edge j}]_{j\in\mathcal J_\edge}$, a price vector
$c_\edge=(c_{\edge j})_{j\in\mathcal J_\edge}$, and the shared-capacity
uptake set
\begin{equation}
\label{eq:uptake-set}
\mathcal X_\edge
=\Big\{x_\edge\in\mathbb R_+^{|\mathcal J_\edge|}:
\sum_{j\in\mathcal J_\edge(\omega)}x_{\edge j}\le\bar x_{\edge\omega}
\text{ for every }\omega,\;\;
x_{\edge j}\le\bar x_{\edge j}\text{ for every }j\in\mathcal J_\edge\Big\},
\end{equation}
where $\mathcal J_\edge(\omega)$ collects the quoted keys with time
window $\omega$. The window limit $\bar x_{\edge\omega}>0$ is the
capacity that the provider offers in that window. All quoted keys of the
window share it, so the same fleet or facility capacity is not counted
once for every key. The key limit $\bar x_{\edge j}\in(0,\infty]$ caps a
single key, for example storage space that the disruption reduces, and
is infinite when no such cap applies. The coordinates of $x_\edge$ use a
capacity unit common to all providers of a resource family, such as
truck-hours for carriers, and $\theta_{\edge j}>0$ converts that unit
into units of commodity $j$.
The block's injection is $q_\edge=A_\edge x_\edge$ and its value function is
\begin{equation}
\label{eq:bid-valuefn}
\tilde{\Fval}_\edge(q_\edge)
=\inf_{x_\edge\in\mathcal X_\edge}
\big\{c_\edge^\top x_\edge:q_\edge=A_\edge x_\edge\big\}.
\end{equation}
This function is closed, proper, convex, and polyhedral. With the
$N_r=|\Iset^r|$ agents of \cref{subsec:commodities}, the exchange form is
\begin{equation}
\label{eq:exchange}
\min_{\{q_i\}}\;
\sum_{i\in\Iset^r}\tilde{\Fval}_i(q_i)
\qquad\text{s.t.}\qquad
\sum_{i\in\Iset^r}q_i=0.
\end{equation}
We denote its optimal value by $\Qval(R^r)$. When $R^r=\emptyset$,
this is the fixed-incumbent benchmark $\Qval(\Mset^0)$ in
\eqref{eq:pooling}.

\begin{assumption}[Fixed-network regularity]
\label{ass:regularity}
For a fixed exchange-agent set $\Iset^r$, every
$\tilde{\Fval}_i$ is closed, proper, and convex. Problem
\eqref{eq:exchange} is feasible with a finite optimal value, and its
unaugmented Lagrangian admits a saddle point.
\end{assumption}

To accommodate heterogeneous commodity scales, fix declared
characteristic sizes $\sigma_\kk>0$, let
$S=\operatorname{diag}(\sigma_\kk)$, and define the positive-definite
penalty metric
$D_\rho=\rho S^{-2}$ for a scalar tuning parameter $\rho>0$. Write
$\lVert v\rVert_{D_\rho}^2=v^\top D_\rho v$ and
$\bar q^{\,t}=N_r^{-1}\sum_iq_i^t$. The implementation keeps the
coordination state $\scaled^t$ in the original commodity units and
applies
\begin{subequations}
\label{eq:admm}
\begin{align}
q_i^{t+1}&\;\in\;
\argmin_{q_i}\;\Big\{\tilde{\Fval}_i(q_i)
+\tfrac12\big\lVert q_i-q_i^t+\bar q^{\,t}+\scaled^t
\big\rVert_{D_\rho}^2\Big\},
\qquad i\in\Iset^r,
\label{eq:admm-x}\\
\scaled^{t+1}&\;=\;\scaled^t+\bar q^{\,t+1}.
\label{eq:admm-w}
\end{align}
\end{subequations}
The metric expresses each imbalance relative to its declared size
(\cref{app:coordination}). Each firm solves its update on its private operational model, and
the recourse agent and each admitted block solve theirs as prescribed
optimization steps (\cref{app:coordination}).

During coordination, firm $m$ communicates only its verified
participation set $\Kset_m$ and its current injection $\inj_m^{t+1}$.
The coordinator forms $\bar q^{\,t+1}$ and broadcasts
$(\bar q^{\,t+1},\scaled^{t+1})$. The firm's demand, endowment,
objective, constraints, and internal decisions stay within its local
problem. These are protocol-level information boundaries, not a
cryptographic privacy guarantee. We assume that firms implement these updates on
their true models and report their injections truthfully
(\cref{ass:compliance}), and we model no incentive scheme.

Each call stops when the scaled primal and dual residuals fall below
declared tolerances or an iteration cap is reached, and a later call can
start from the terminal state of the previous one
(\cref{app:coordination}). For a fixed agent set and exact local
solves, the iteration converges to the round optimum, and
$\mu^t=D_\rho\scaled^t$ converges to an optimal multiplier of
\eqref{eq:exchange} (\cref{prop:convergence}).

\subsection{Coordination outputs}
\label{subsec:outputs}\label{subsec:dual}

Discovery uses two outputs of coordination, a shadow price and a
capacity gap. The exchange formulation uses one sign convention throughout: the shadow
price is the economically signed multiplier of the capacity-balance
constraint, so a positive value indicates scarcity. Let
\[
\Lambda_r^\star
=\{-\mu:\mu\text{ is an optimal multiplier of \eqref{eq:exchange}}\}
\]
be the set of shadow prices in round $r$. Under the metric
in \eqref{eq:admm}, the running exchange multiplier is
$\mu^t=D_\rho\scaled^t$, so
\begin{equation}
\label{eq:price-recovery}
\dual^t=-\mu^t=-D_\rho\scaled^t,
\qquad
\dual^t\xrightarrow[t\to\infty]{}
\dual^{r,\star}\in\Lambda_r^\star
\end{equation}
by \cref{prop:convergence}. At a finite terminal iteration $T_r$,
the numerical price is $\hat\dual^r=-D_\rho\scaled^{r,T_r}$ and need
not belong to $\Lambda_r^\star$. Every exact shadow price supports the value
of a specified injection: it gives an affine lower bound on the round
optimum perturbed by that injection, and the shadow prices are exactly
the negative subgradients of this perturbation value
(\cref{cor:sensitivity}).

The gap carries the other signal. The recourse agent's optimality
conditions tie each price to the recourse it buys
(\cref{app:coordination}). Positive recourse away
from the kink of the two-tier cost \eqref{eq:twotier} pins the shadow
price to the local recourse slope. At the kink, recourse optimality
restricts the price to $[P^1_\kk,P^2_\kk]$, while the other agents may
narrow this interval or pin down a unique price. At zero recourse, the
price only has to stay at or below $P^1_\kk$, and the other agents
determine it. Thus the gap locates a commodity on which the current
network depends on fallback capacity, and the price supplies a
supporting value for a specified injection.

A finite exchange-ADMM call returns the terminal record
\eqref{eq:outputs-hat}: the reported gaps $\hat\gap^r$, the price
estimate $\hat\dual^r$, the terminal residuals, and a status $\zeta_r$
that records whether both tolerances were met or the iteration cap was
reached. Without an additional error bound, the residuals do not
certify the distance to an exact optimizer, objective value, or shadow
price.

The coordination layer therefore serves two roles. In operation, the
mechanism's decisions consume only $(\hat\gap^r,\hat\dual^r)$: the gap
directs need selection before any offer is known, and the price screens a
revealed offer block for admission (\cref{sec:method}). A decision taken
after the iteration cap is a finite-iteration heuristic. In analysis, the
exact round solution, with value $\Qval(R^r)$ and a shadow price
$\dual^{r,\star}\in\Lambda_r^\star$, defines the formal results. The
value is well defined even when the allocation or the price is
nonunique. The exact solution can also replace every coordination call
in an experiment, which separates the effect of the discovery rules from
finite-precision coordination (\cref{subsec:policies}).

\section{Learning-Guided Discovery and Admission of Open Capacity}
\label{sec:method}

This section addresses the problem of expanding the network of
\cref{sec:setting} during a disruption. The incumbents $\Mset^0$ and the
recourse agent are given, and so is a finite set $\Eset$ of records of
external providers that the network has never engaged, such as the
listings of a public business directory. Before contact, the platform
observes only these records, which hold directory fields and, for some
firms, a free-text card (\cref{subsec:match}). Each provider's
capabilities, availability during the event, and offer terms are latent
but fixed (\cref{ass:fixed-offer}). The platform may send at most $K$
requests for quotation (RFQs), each asking one provider about one
inquiry class \eqref{eq:inquiry-class}, and may admit any verified offer
block \eqref{eq:edge} that an RFQ reveals. Admission adds the block, not
the provider's private operating model, to the exchange agents. Each
decision may depend only on the records, the past RFQs and their
outcomes, and the coordination outputs (\cref{subsec:outputs}). The aim
is to minimize the expected terminal operating cost
$\mathbb E[\Qval(R)]$, where $R$ is the set of blocks admitted after at
most $K$ RFQs and the expectation is over the latent provider outcomes.

Unlike classical column generation \citep{lubbecke2005column}, no
pricing problem can generate a candidate block. Each RFQ is a metered
action that reveals at most one latent offer, which only then can be
priced, and each admission changes the value of the next request.
Because a revealed block is part of its provider's full-menu block and
an optional block never raises the exact cost, no policy can do better
than the full-information optimum, which the reference protocol of
\cref{subsec:stop} attains under exact coordination. Under a binding budget, the
mechanism we propose is a heuristic, not an optimal policy.
\Cref{fig:method} shows how it combines learned matching with sequential
discovery and re-optimization.

\begin{figure}[!htbp]
\centering
\includegraphics[width=\textwidth]{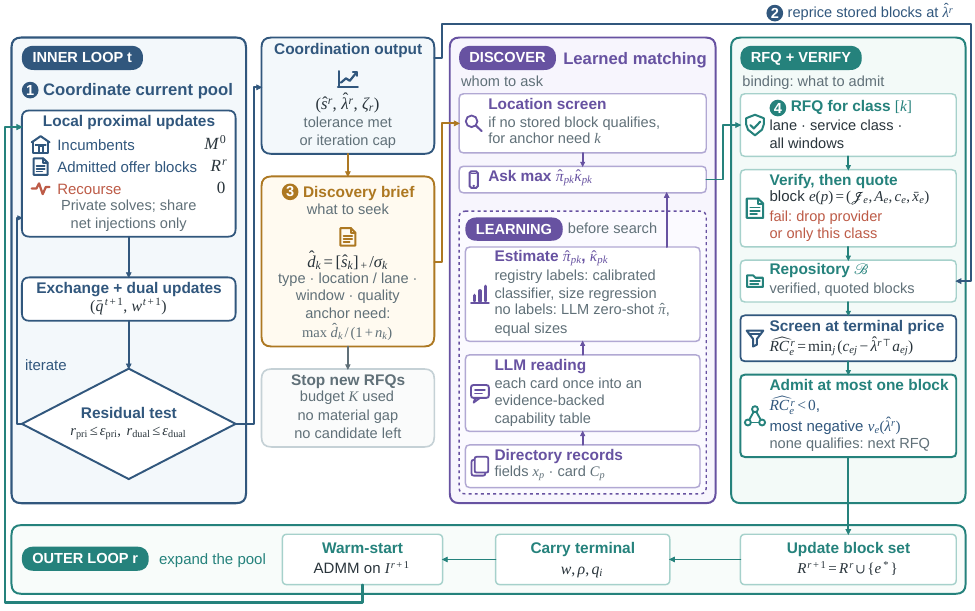}
\caption{Nested coordination and pool expansion.}
\label{fig:method}
\end{figure}

The left part of \cref{fig:method} is the inner coordination loop of
\cref{sec:setting}, embedded in an outer loop that expands the pool.
Each outer round follows the four numbered steps. First, warm-started
exchange-ADMM coordinates the incumbents, the admitted offer blocks, and
the recourse agent, and returns gap and price estimates
$(\hat\gap^r,\hat\dual^r)$ with a convergence status. Second, the blocks
already in the repository, which need no further request, are repriced
at $\hat\dual^r$. If any passes shadow-price screening, the one with the
largest estimated saving is admitted (\cref{subsec:pricing}). The
expanded pool is then re-coordinated, warm-started from the previous
call (bottom of the figure), before any new RFQ is sent. Third, if none
qualifies, the scale-normalized gaps select the need to pursue
(\cref{subsec:brief}), and the matching layer selects the provider to
ask (\cref{subsec:match}). The matching layer is estimated before search
from directory records and an LLM reading of the free-text cards that
some firms publish. Fourth, the provider receives an RFQ for that need
over the planning horizon, and verification and the quotation reveal an
offer block or extend the provider's earlier one (\cref{subsec:rfq}).
The loop then returns to the second step, or to the first if an
admitted block was extended. In this loop, coordination decides what to
seek, learning decides whom to ask, and verified quotes, screened at the
coordination prices, decide what to admit. \Cref{subsec:stop} gives the stopping rules, the complete loop in
\cref{alg:loop}, and what the loop guarantees.

At round $r$, the admitted blocks $R^r$, the incumbents, and the
recourse agent form the exchange-agent set $\Iset^r$ of
\cref{subsec:commodities}, with operating cost $\Qval(R^r)$. The loop
also keeps the repository $\Bset$, the inquiry history, the providers
dropped after a failed verification, and the need counters $n_\kk$.

\subsection{Coordination-gap priorities and discovery briefs}
\label{subsec:brief}

Before an external offer is known, the mechanism uses the reported
recourse vector to identify where the current network still relies on
fallback capacity. Search covers a declared set
$\Kset^{\mathrm{search}}\subseteq\Kset$ of commodities, those of the
resource families for which the platform holds provider records and a
matching layer, such as the lane commodities when the records describe
carriers. Because commodities have different physical units,
raw gaps are not comparable across coordinates. Using the declared
characteristic sizes $\sigma_\kk$ from \cref{subsec:exchange}, define
the dimensionless coordination gap and the active need set by
\begin{equation}
\label{eq:coord-gap}
\hat d^r_\kk
=\frac{[\hat\gap^r_\kk]_+}{\sigma_\kk},
\qquad
\Kset_+^r
=\{\kk\in\Kset^{\mathrm{search}}:\hat d^r_\kk>\varepsilon_{\mathrm{brief}}\},
\end{equation}
where $\varepsilon_{\mathrm{brief}}\ge0$ is a dimensionless numerical
dust threshold. We call an active commodity a need. The scales
$\sigma_\kk$ are fixed before the first coordination call from
information available at that time, such as a recourse tier, incumbent capacity, or
a historical gap.

This signal targets emergency-recourse replacement. An empty
$\Kset_+^r$ means that no material fallback need remains among the
searchable commodities, not that
external capacity has no value, and $\hat d^r_\kk$ measures relative
fallback dependence rather than the saving from a contact.

The contact rule interleaves active needs rather than exhausting the
largest one. Let $n_\kk$ be the number of RFQs sent so far whose inquiry
class (\cref{subsec:rfq}) contains need $\kk$, including failed
verifications and non-improving quotations. At the next RFQ the score of
an active need $\kk\in\Kset_+^r$ is
\begin{equation}
\label{eq:need-interleave}
\hat\psi^r_\kk
=\frac{\hat d^r_\kk}{1+n_\kk},
\end{equation}
and the mechanism selects the highest-scoring need that still has a
candidate provider. This selected need is the RFQ's anchor. Exact ties
are resolved by a predeclared pseudo-random order that is the same for
all policies.

Each active need produces a structured discovery brief
\begin{equation}
\label{eq:brief}
\brief(\kk)=\big(\text{type}_\kk,\;\ell_\kk,\;\omega_\kk,\;\chi_\kk;\;
\hat d^r_\kk\big).
\end{equation}
The first four fields are the commodity's resource family, location or
lane, time window, and quality class. They define the delivery
requirements, the public location screen, and the need type scored by
the matching layer (\cref{subsec:match}). The normalized gap governs attention across
heterogeneous needs. Every
field is a deterministic function of the commodity definition, declared
scale, and reported gap. No shadow price enters
\eqref{eq:coord-gap}--\eqref{eq:need-interleave}. Prices are used only
after an offer block is revealed (\cref{subsec:pricing}).

\subsection{Learned matching of provider records to needs}
\label{subsec:match}

Once a need is selected (\cref{subsec:brief}), the matching layer ranks
the providers that could serve it, using only what their records show
before contact. Each record $p\in\Eset$ holds structured directory
fields, such as a
declared location, industry codes, years in business, and
certifications, and for some providers a free-text card $C_p$ in which
the provider describes its services (\cref{app:data-example} shows one
record). Neither contains verified
capability, prices, or binding capacity limits, which become known only
through an RFQ, its verification, and the quotation. A deterministic
location screen retains provider $p$ for need
$\kk$ only if the provider's declared location is consistent with the
commodity's service location or lane, for example a carrier based at
either end of a lane. The screen is a public rule. It does not establish
that the provider can deliver. The declared location fields are not
inputs of the predictions below, and beyond the screen they enter only
through the public outage adjustment of the size signal. The remaining directory
fields, denoted $x_p$, and the card inform the predictions.

For a retained pair, verification combines two latent conditions. Let
$z_{p\kk}\in\{0,1\}$ indicate that provider $p$ has the capability that
need $\kk$ requires, such as active for-hire operation over the required
distance class, and let $u_p\in\{0,1\}$ indicate that the provider is
available to respond during the current disruption. Verification for
need $\kk$ succeeds exactly when both hold,
\begin{equation}
\label{eq:verify-factor}
v_{p\kk}=z_{p\kk}\,u_p .
\end{equation}
The guarantees below need only its soundness part: a verified key implies
both conditions (\cref{ass:verify}). Capability is a stable attribute that authoritative records can label
for comparable firms. Availability is an event-time condition that no
pre-contact record reveals. The matching layer therefore estimates
capability,
\begin{equation}
\label{eq:match}
\hat\pi_{p\kk}
=\widehat{\Pr}\big(z_{p\kk}=1\mid x_p,C_p\big),
\end{equation}
and a size signal $\hat\kappa_{p\kk}$ for the capacity the provider could
offer if capable: a fleet or facility size proxy fitted to capable
firms, converted into an offered quantity by a public rule and scaled by
the share that public information, such as an outage map, expects to
remain at the provider's base after the event. Within the need selected
by \eqref{eq:need-interleave}, the mechanism asks the candidate with the
largest capability-weighted size score,
\begin{equation}
\label{eq:contact-score}
\hat\pi_{p\kk}\,\hat\kappa_{p\kk}.
\end{equation}
The product targets the capacity that a request is expected to reveal.
It matches that target up to a common factor when availability is
independent of the provider's capability, size, and record and
$\hat\kappa_{p\kk}$ tracks offered capacity. Because the size proxy is
coarse, the score serves as a ranking heuristic rather than a capacity
forecast. The
conversion can assign large fleets a smaller offered share, because they
commit most of their capacity under contract, as in the testbeds
(\cref{app:testbed-offers}). The score only orders requests. Verification
decides eligibility.

Many commodities differ only in where and when capacity is required, so
capability is defined for a need type $g(\kk)$, a coarser class such as
refrigerated long-haul transport, and $\hat\pi$ is shared across the
commodities of that type. A more demanding type is estimated as the
product of the conditional probabilities of the requirements it adds, so
it never receives a higher probability than the type it refines. \Cref{subsec:labels} gives the chain of the transport application. Refrigerated types add a refrigerated-given-carrier factor, which treats
refrigeration and distance as independent given carriage. This
structure also determines what an extension requires. A new location or
time window of an existing need type reuses its estimates. A new need
type requires a new capability estimate, which the LLM can supply before
any label exists, while the size estimate of its resource family carries
over. A new resource family, such as storage, requires its own capability
and size measures and its own provider records, while coordination,
screening, and admission are unchanged.

The matching layer can be estimated in two regimes that differ in label
availability. When authoritative records label capability and size for
comparable firms, $\hat\pi$ is a supervised classifier trained on those
labels, and $\hat\kappa$ comes from a regression of the recorded size on
the same inputs. The input is the directory fields, alone or together
with a representation of the card text, and the learner can be any
classifier for tabular data whose scores a calibration map turns into
probabilities. \Cref{subsec:prediction-results} crosses several
text representations with several standard learners. No estimate is
needed when the provider's own authoritative record can be linked
reliably. Its recorded capability and size then replace both estimates,
which is the registry reference of \cref{subsec:search}.

A large language model (LLM) can enter in either regime. As a reader, it
turns every card once into an evidence-backed capability table: whether
the firm runs its own capacity or arranges capacity from others, its
equipment, temperature control, service range, and stated size, each with
a verbatim span from the card. A deterministic parser checks every span
against the card text and records an unsupported field as not stated. In
a separate call, the model states a probability for each capability
question from the card. The table and these probabilities can serve as
text representations for the supervised models. When labels are scarce,
as on a newly opened platform or for a capacity type without an
authoritative registry, the LLM estimates capability directly. It reads
the provider's record, meaning the directory fields and the card when
one exists, so the estimate also covers providers without a card. Its
zero-shot probability for each capability question is used as $\hat\pi$
without training or calibration. Once a small labeled sample exists, a
calibration map fitted to it can correct these probabilities before
enough labels accumulate for the supervised regime
(\cref{app:learning-diagnostics}). Without size labels, every provider
receives the same size, so the score reduces to $\hat\pi$ times the public outage adjustment. In both regimes the LLM only reads and judges records. The screen, the score, verification,
pricing, and admission remain deterministic rules. No brief, offer, or
registry record is part of any LLM input.

Because the score \eqref{eq:contact-score} multiplies the probability by
a predicted size, the probabilities themselves matter, not only their
order. They also enter the chain of conditional probabilities that
defines a need type and the joint ranking of providers with and without
a card. All three uses treat the scores as calibrated probabilities,
which a zero-shot estimate does not guarantee, so \cref{subsec:protocol}
reports log loss alongside the area under the ROC curve.

The matching layer is estimated before search and stays fixed within
each disruption episode, and no RFQ outcome updates it. Under verification soundness (\cref{ass:verify}), prediction
errors cannot admit an incompatible provider, because admission depends
only on verified and quoted terms (\cref{subsec:pricing}), so
optional-block non-worsening holds for any matching layer, including an
uninformative one. Errors still affect operating performance. A false
positive spends an RFQ that yields no usable key, and a false negative
delays or misses an eligible provider. Re-contact limits part of this loss, because a
provider that cannot serve the requested class can still be asked about
another one (\cref{subsec:rfq}). \Cref{subsec:decision-value} measures
what perfect capability and size information would add in operating
value.

\subsection{Provider requests, verification, and offer-block revelation}
\label{subsec:rfq}

After the mechanism selects provider record $p$ through anchor need
$\kk$, it sends $p$ one request for quotation (RFQ) for the inquiry
class of $\kk$,
\begin{equation}
\label{eq:inquiry-class}
[\kk]=\{j\in\Kset:\text{type}_j=\text{type}_\kk,\;\ell_j=\ell_\kk,\;
\chi_j=\chi_\kk\},
\end{equation}
the keys that differ from $\kk$ only in their time window. In the
transport application, an inquiry class is one lane and one service class
over the planning horizon. One RFQ consumes one unit of the budget
$K$, whatever its outcome. Verification returns $v_{pj}\in\{0,1\}$ for
every $j\in[\kk]$. If the provider is unavailable or lacks the basic
capability of its resource family, verification fails for every key, and
the provider is not asked again. A provider that is available and holds
the basic capability but cannot serve this class, for example a carrier
without refrigerated service asked about a refrigerated lane, remains a
candidate for other classes. No provider is asked twice about the same
class, so the candidates of a need are the location-screened providers
that have not been dropped and have not yet been asked about its class.

A provider can therefore receive several RFQs, and its quotations
accumulate in one block. Let $\mathcal C_p$ be the inquiry classes for
which $p$ has returned at least one verified key, including the current
one. The quoted-key set is
\[
\mathcal J_\edge
=\{j\in[\kk']:[\kk']\in\mathcal C_p,\;v_{pj}=1\},
\]
and the RFQ reveals the provider-offer block $\edge(p)$ on
$\mathcal J_\edge$. The new block replaces the provider's earlier block.
It keeps the earlier quotes and adds the new class under the same window
limits, so the provider's capacity is never counted twice. If the current
class has no verified key, the RFQ yields no offer, and the earlier
block, if any, is unchanged. Each provider--resource-family pair is a
separate record, so a firm that serves several resource families has
separate blocks with family-specific limits.

Denote the block by
\begin{equation}
\label{eq:edge}
\edge=\edge(p)
=(\mathcal J_\edge,A_\edge,c_\edge,\bar x_\edge).
\end{equation}
Its injection columns, quoted prices, and window and key limits define
the offer-block agent of \cref{subsec:exchange} on the verified keys.
The numerical fields, dimensions, and commodity keys must be finite and
valid, except that key limits may be infinite.

A revealed block enters the persistent repository $\Bset$ unless the
provider's earlier block has already been admitted. In that case the new
block takes its place among the admitted blocks, and the network is
re-coordinated before the next decision. The earlier
uptake remains feasible for the new block, so the replacement cannot
raise the round optimum.

The mechanism uses per-request RFQs because they follow how truckload
capacity is procured. Shippers request prices lane by lane in contract
bids, tender loads one at a time, and find spot capacity load by load
\citep{caplice2007electronic,acocella2022setit}, so a carrier quotes the
lane and service it is asked about.

\begin{assumption}[Verification soundness]
\label{ass:verify}
If $j\in\mathcal J_\edge$, then $z_{pj}=u_p=1$, column $a_{\edge j}$
satisfies the declared hard compatibility rules of key $j$, and the block
in \eqref{eq:edge} satisfies the validity conditions above.
Soundness rules out a quoted key that the provider cannot serve, but it
does not assert that every eligible provider is found or admitted.
\end{assumption}

\begin{assumption}[Fixed offer and deterministic delivery]
\label{ass:fixed-offer}
Each provider has at most one latent tariff schedule and one set of
window limits. These terms, its capabilities $z_{p\kk}$, and its
availability $u_p$ are fixed before the first response and do not depend
on when, how often, or in what order providers are asked. An RFQ only
restricts the schedule to the verified keys of the classes asked so far.
The provider does not set terms after observing a shadow price. After a
successful verification, the RFQ returns this fixed block, and the
provider does not later withdraw it or change its terms. If admitted, the block supplies
$A_\edge x_\edge$ at cost $c_\edge^\top x_\edge$ for an uptake
$x_\edge\in\mathcal X_\edge$ selected by the exchange problem. Delivery
failure and post-admission transaction costs are outside the base
operating model.
\end{assumption}

\subsection{Shadow-price screening and re-coordination}
\label{subsec:pricing}

A revealed block $\edge\notin R^r$ defines the would-be optional agent in
\eqref{eq:bid-valuefn}. It does not enter the round exchange problem until
admission. For any selected exact shadow price
$\dual\in\Lambda_r^\star$, define its component and block reduced costs by
\begin{equation}
\label{eq:erc}
\rc^r_{\edge j}(\dual)
=c_{\edge j}-\dual^\top a_{\edge j},
\qquad
\rc^r_\edge(\dual)
=\min_{j\in\mathcal J_\edge}\rc^r_{\edge j}(\dual).
\end{equation}
At the terminal price estimate, the operational counterparts are
\begin{equation}
\label{eq:erc-hat}
\widehat{\rc}^{\,r}_{\edge j}
=c_{\edge j}-\hat\dual^{r\top}a_{\edge j},
\qquad
\widehat{\rc}^{\,r}_\edge
=\min_{j\in\mathcal J_\edge}\widehat{\rc}^{\,r}_{\edge j}.
\end{equation}
The block's pricing value at price $\dual$ is
\begin{equation}
\label{eq:block-pricing}
v_\edge(\dual)
=\min_{x_\edge\in\mathcal X_\edge}
(c_\edge-A_\edge^\top\dual)^\top x_\edge\;\le\;0 .
\end{equation}
The window limits of \eqref{eq:uptake-set} separate, so
\eqref{eq:block-pricing} is a sum over time windows of fractional
knapsacks: each window's limit goes to its quoted keys in increasing
order of reduced cost, each key up to its own limit, as long as the
reduced cost is negative. With a single window and no key limit, the
value reduces to $-\bar x_\edge[-\rc^r_\edge(\dual)]_+$. Because every
limit is positive, $v_\edge(\dual)<0$ exactly when
$\rc^r_\edge(\dual)<0$. The block reduced cost therefore decides whether
a block can improve at $\dual$, and the pricing value measures how much
it saves at that price. After admission and re-coordination, the
exchange optimum may split a window's capacity across several keys. Only
verified and quoted terms enter these expressions. Pre-contact evidence
affects which block becomes known, not its value once revealed.

After each coordination call, all blocks in $\Bset$ are repriced before
new discovery begins. The operational screen is a negative estimated
block reduced cost, so the qualifying stored blocks are
\begin{equation}
\label{eq:admit}
A^r=\{\edge\in\Bset:\widehat{\rc}^{\,r}_\edge<0\}.
\end{equation}
In computation, the equivalent test $v_\edge(\hat\dual^r)<0$ is applied
with a numerical tolerance (\cref{tab:numerical-settings}). If $A^r$ is
nonempty, the mechanism admits the block $\edge^\star\in A^r$
with the most negative estimated pricing value
$\hat v^{\,r}_\edge=v_\edge(\hat\dual^r)$, the largest saving at the
current price estimate. Exact ties are broken by a fixed order. The
admitted block leaves the repository,
$R^{r+1}=R^r\cup\{\edge^\star\}$, and all other quoted blocks remain in
$\Bset$. Because an admission can change every remaining block's reduced
cost, the expanded agent set is re-coordinated, warm-started by
\eqref{eq:warm-state} when enabled, before another block is admitted.
Thus at most one provider-offer block is admitted between coordination
calls, even when a qualifying block was already known and no new RFQ is
required.

Because $0\in\mathcal X_\edge$ and no fixed admission charge enters
$\Qval$, admitting an optional block never raises the exact round
optimum $\Qval(R^r)$. Under the support condition of
\cref{prop:monotone} in \cref{app:discovery}, it lowers the optimum
exactly when the block reduced cost \eqref{eq:erc} has a negative
supremum over all exact shadow prices. With a unique exact price, this
condition is the exact analogue of the screen above. At a kink of the
recourse schedule, negativity at one exact price is not enough.

\subsection{Search budgets, stopping, and guarantee boundary}
\label{subsec:stop}

Reading records, sending RFQs, and verifying providers also cost the
platform, but these search costs lie outside the operating objective
$\Qval$, so optional-block non-worsening does not imply a lower total
response cost. The stop rules apply only to new RFQs. Stored blocks are screened first,
because their discovery and verification costs are already sunk. The
material-gap stop is
\begin{equation}
\label{eq:stop-sat}
\max_{\kk\in\Kset^{\mathrm{search}}}\hat d^r_\kk\le d_{\min},
\end{equation}
with dimensionless $d_{\min}\ge\varepsilon_{\mathrm{brief}}$. The budget
stop applies when the number of RFQs sent, $N_{\mathrm{rfq}}$, reaches
$K$, and the exhaustion stop applies
when no active need has a candidate provider left. None of these stops is
an optimality certificate. \Cref{alg:loop} states the complete loop.

\begin{algorithm}[!htbp]
\caption{Learning-guided discovery and reduced-cost admission}
\label{alg:loop}
\begin{algorithmic}[1]
\STATE \textbf{input:} incumbents $\Mset^0$, provider records $\Eset$,
searchable commodities $\Kset^{\mathrm{search}}$, a frozen matching
layer, RFQ budget $K$, thresholds
$\varepsilon_{\mathrm{brief}}$ and $d_{\min}$, and a coordination method
with its tolerances and iteration cap
\STATE score every location-screened provider and need type once by
\eqref{eq:match} and $\hat\kappa_{p\kk}$
\STATE $r\leftarrow0$, $R^0\leftarrow\emptyset$,
$\Bset\leftarrow\emptyset$, $N_{\mathrm{rfq}}\leftarrow0$,
$n_\kk\leftarrow0$ for all $\kk$; coordinate $\Iset^0$
\LOOP
  \STATE reprice $\Bset$ at $\hat\dual^r$ by
  \eqref{eq:erc-hat}--\eqref{eq:block-pricing}; form $A^r$ by
  \eqref{eq:admit}
  \IF{$A^r\ne\emptyset$}
    \STATE move the block of $A^r$ with the most negative
    $\hat v^{\,r}_\edge$ into $R^{r+1}$; coordinate $\Iset^{r+1}$;
    $r\leftarrow r+1$; \textbf{continue}
  \ENDIF
  \IF{\eqref{eq:stop-sat} holds, $N_{\mathrm{rfq}}=K$, or no active
  need has a candidate}
    \STATE \textbf{return} $\{0\}\cup\Mset^0\cup R^r$ and the stop cause
  \ENDIF
  \STATE select the anchor need $\kk$ by \eqref{eq:need-interleave} and
  its candidate $p$ by \eqref{eq:contact-score}
  \STATE send $p$ an RFQ for $[\kk]$ and record it;
  $N_{\mathrm{rfq}}\leftarrow N_{\mathrm{rfq}}+1$, $n_j\leftarrow n_j+1$
  for $j\in[\kk]$
  \STATE verify $[\kk]$; drop $p$ if it is unavailable or lacks the basic
  capability
  \IF{some key of $[\kk]$ verifies}
    \STATE reveal $\edge(p)$ on the verified keys of every class $p$ has
    quoted
    \IF{$p$ has a block in $R^r$}
      \STATE replace it by $\edge(p)$ in $R^{r+1}$; coordinate
      $\Iset^{r+1}$; $r\leftarrow r+1$
    \ELSE
      \STATE store $\edge(p)$ in $\Bset$, replacing any earlier block of
      $p$
    \ENDIF
  \ENDIF
\ENDLOOP
\end{algorithmic}
\end{algorithm}

For the exact finite-universe benchmark, we distinguish the provider
records in $\Eset$ from the offer blocks they can reveal. The benchmark contacts every provider once, asking it
about all its location-compatible keys in $\Kset^{\mathrm{search}}$.
Because capabilities, availability, and terms are fixed before the
response (\cref{ass:fixed-offer}), the resulting set
$\mathcal C_{\mathrm{elig}}$ of nonempty verified blocks does not depend
on the contact order, and it contains at most one shared-limit block per
provider. This makes the benchmark a full-information reference for the
provider universe. It bypasses learned matching and the RFQ budget, but
verification, repository, and admission rules are unchanged.

\begin{proposition}[Finite termination and optimality under full revelation]
\label{prop:termination}
Run \cref{alg:loop} with its contact step replaced by this reference
protocol, so that each provider is asked once about all its
location-compatible keys. Suppose that $|\Eset|<\infty$, that the exhaustion stop applies only
once every $p\in\Eset$ has been contacted, so that every block in
$\mathcal C_{\mathrm{elig}}$ is revealed, and that every revealed
unadmitted block persists. Suppose further that \cref{ass:regularity} holds in
every round, that the RFQ budget, the material-gap stop,
the numerical admission tolerance, an iteration cap, and any
maximum-round truncation are disabled, and that every round exchange
problem and one of its shadow prices are solved exactly. Then the loop
makes at most $|\mathcal C_{\mathrm{elig}}|\le|\Eset|$ admissions
and terminates finitely. At termination, one exact shadow price of the
final round gives nonnegative block reduced cost to every excluded
offer block, and the final round optimum equals
$\Qval(\mathcal C_{\mathrm{elig}})$, the optimum with every eligible
block admitted.
\end{proposition}

The proof, a finite generalized-column argument, is given in
\cref{app:discovery}.

The guarantee boundary is narrow. Outside the conditions of
\cref{prop:termination}, suppose the screen uses an exact shadow price
of the current round. A stop with repository-first admission then leaves no
stored block with a negative pricing value at that price, and under the
support condition of \cref{prop:monotone} none of them would lower the
exact cost if admitted. Offers not yet revealed, including further
inquiry classes of providers already asked, can still improve the
network. With finite-iteration prices, even a stored block
that fails the screen can improve at every exact price. Under a binding
RFQ budget or incomplete revelation, no
optimality conclusion follows, and because unrevealed offers are unknown
in an open world, no global-optimality certificate is claimed.

\section{Numerical Experiments}\label{sec:experiments}

We evaluate the mechanism in three regional logistics networks facing
disruption scenarios calibrated to historical events. The matching layer is trained on directory firms
in 41 other jurisdictions and frozen before the three regions are used,
so the regions serve both as geographic holdouts for prediction and as
testbeds for the discovery policies. \Cref{subsec:testbeds} describes the testbeds and their evidence
boundary. \Cref{subsec:open-value} measures how much outside capacity
could save if every offer were known, a benchmark that would take
thousands of RFQs per episode to reach. \Cref{subsec:prediction-results}
asks whether records available before contact can narrow whom to ask,
and \cref{subsec:search} measures how much of the open value a limited
number of RFQs recovers and how learned carrier selection and
gap-guided need selection contribute to it. \Cref{subsec:decision-value}
asks whether provider information or contact depth limits the remaining
recovery, and whether better information makes gap guidance more
valuable.

\subsection{Regional testbeds and evidence boundary}
\label{subsec:testbeds}

\begin{figure}[!htbp]
\centering
\includegraphics[width=\textwidth]{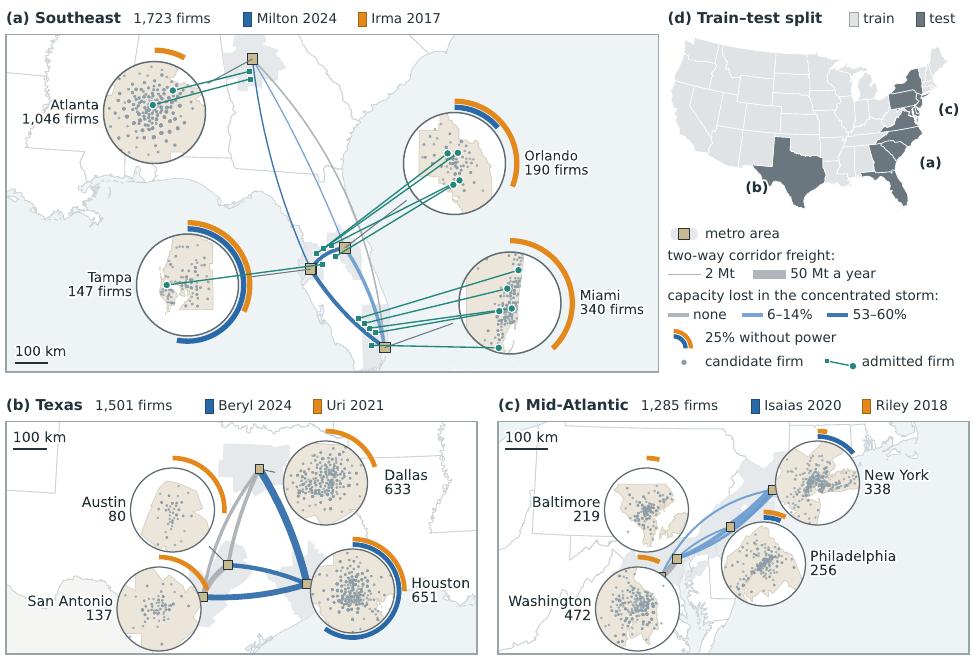}
\caption{Regional testbeds and the train--test split. Panels (a)--(c)
show the four metropolitan areas, candidate firms, and freight
corridors of each testbed. Each lens magnifies the area within 75 km of
a metro's median firm and shows the candidate firms located there.
Corridor width represents annual
two-way truck freight, corridor color represents capacity lost in the
concentrated storm, and the blue and orange arcs show metro outage
shares in the concentrated and widespread storms. In panel (a), teal
marks the 12 firms admitted with 80 RFQs in one Milton episode of the
recommended policy (\cref{subsec:policies}) and links each firm to the
corridor it serves. Panel (d) shows the geographic split: the matching
layer is trained on firms from the light jurisdictions and evaluated in
the three dark holdout regions (\cref{subsec:protocol}).
Map lines delineate study areas and do not necessarily depict accepted
national boundaries.}
\label{fig:testbeds}
\end{figure}

Each testbed combines an incumbent network calibrated to public data
with a pool of observed directory firms, which are treated as previously
unengaged providers (\cref{fig:testbeds}). The incumbent network joins
four metropolitan areas by all twelve directed truck lanes over three
daily periods. Every metro has an
ambient and a refrigerated incumbent warehouse, and every lane has a
dry and a refrigerated incumbent carrier, giving 32 incumbents and 120
capacity keys per network. Warehouse capacity follows County Business
Patterns employment \citep{census2025cbp}, lane capacity follows Freight
Analysis Framework truck tonnage \citep{bts2025faf}, and incumbents use
70 to 80\% of their capacity in the normal state. Shortfalls that
incumbents cannot cover pass first to a limited cheap tier, read as the
network's usual backup sources, and then to an expensive tier that
stands for the cost of not finding capacity
(\cref{subsec:benchmark,app:testbed-settings}).

Each region faces a concentrated storm, a widespread storm, and a demand
surge. The historical anchors are Milton and Irma in the Southeast
\citep{beven2025milton,cangialosi2018irma}, Beryl and Uri in Texas
\citep{beven2025beryl,ferc2021coldweather}, and Isaias and Riley in the
Mid-Atlantic \citep{latto2021isaias,nws2018windstorm}. The first storm
in each pair concentrates its outages on one or two metros, while the
second affects all four. A storm's severity in a metro is
the mean share of its electricity customers without power in the 72
hours after the event. The storm removes this share of the handling,
refrigerated storage, and dock capacity in the metro, and each lane
loses the larger of the shares at its two ends. It also raises relief
freight into the metros it hits. The surge is a stand-alone event
modeled on the days before the concentrated storm. It raises demand in
the metros that the storm threatens but damages no capacity. The
Southeast also faces a fuel shortage on the lanes leaving Atlanta,
modeled on the Colonial Pipeline outage of 2021
\citep{ajc2021colonial,patch2021fuel}. \Cref{app:testbed-protocol} gives the event rules,
conversion constants, and checks run before any policy.

The provider pool of each testbed contains the directory firms located
in its four metros. The firms come from the Small Business
Search of the US Small Business Administration \citep{sba2026sbs}, a
public directory in which small firms describe themselves to government
buyers. We retrieved them under five industry codes for trucking,
freight arrangement, and warehousing. They are linked to the Federal Motor Carrier Safety
Administration Company Census \citep{fmcsa2026census} only under strict
same-state matching rules (\cref{app:data-linkage}). A weaker link could
attach another firm's capability and fleet size to the directory record.
Only linked firms enter the experimental pools because their registry
records provide the capability and fleet-size values needed for
evaluation. The results therefore do not cover unlinked directory firms.
The three pools hold 4,509 firms (\cref{tab:testbeds}). They are not
otherwise filtered and therefore include brokers, warehouses, and
inactive registrants that the mechanism must screen out. Each metropolitan pool lies within a larger
regional holdout, drawn in \cref{fig:testbeds}(d), and the matching
layer is evaluated on every strictly linked firm in that holdout
(\cref{subsec:prediction-results}).

\label{subsec:labels}%
Registry fields supply the capability labels and fleet-size values used
for training and evaluation. The capability labels form the nested chain
of active for-hire carrier, long haul given carrier, and interstate given
long haul used by the contact score (\cref{subsec:match}), together with
refrigerated cargo among carriers (\cref{app:data-labels}). In the
decision experiments, the capability labels determine what verification
reveals, and both capability and fleet size define the
perfect-provider-data reference. Directory-based policies do not see
these fields when choosing whom to contact. They see only the directory
record and, for about a quarter of the firms, a free-text card
(\cref{app:data-example}). An RFQ and its verification then reveal
capability, availability, offered quantity, and price. Because the
evaluated matching layer covers carrier
capability, the searchable set contains lane keys but not storage,
handling, or dock keys (\cref{subsec:brief}).

Provider records and registry attributes are observed, while network
scale, shock magnitudes, and dollar anchors are calibrated to public
data. The records were retrieved in October 2026, and the shocks follow
events from 2017 to 2024. The experiments therefore subject the 2026
provider records to historically calibrated shocks. They are a
contemporary counterfactual evaluation, not a reconstruction of the
provider market at each event date. Availability, offered quantity, and
price are simulated because no data set records the offers that
unaffiliated carriers would make after a disaster. The decision results
are conditional on this simulated offer layer.

A carrier is available with probability 0.7 independently of its
record. This is close to the 68.5\% acceptance rate reported for
contracted primary carriers in a tight market
\citep{acocella2022setit}, but optimistic for an unaffiliated carrier
contacted for the first time. An available carrier offers a share
of its fleet's daily driving hours, scaled by a spare-capacity factor
with mean one (\cref{app:testbed-offers}). It can serve the lanes that
start or end at its home metro, and a carrier based in a hit metro
offers less, in proportion to the outage there. Expected offers
therefore scale with registered fleet size, while the spare-capacity
factor, which no record contains, is revealed only by the quotation.
Quotes are anchored to spot-market premia in tight conditions
\citep{acocella2020elephants}. The mean quote
is about 1.19 times the cheap tier, and about 85\% of quotes exceed it.
Because admitted capacity reduces the gap from its most expensive units
first, outside capacity mostly replaces the expensive tier.
\Cref{subsec:search,subsec:decision-value,app:sensitivity} vary
availability, spare-capacity dispersion, and the other assumptions.

Model costs are translated into dollars separately for transport and
storage, anchored to the industry-average cost of operating a truck
\citep{atri2026costs} and the average third-party pallet-storage fee
\citep{tfa2025warehousing}. The two sides share no
capacity key, so this translation changes neither their decisions nor
their within-side ratios. The reported dollars describe a three-day
response by a platform handling 5\% of the truck freight between its
metros and scale in proportion to the two anchors
(\cref{app:testbed-settings}).

An episode consists of one region, event, and seed. The reported results
use seeds 21 to 40 and cover 200 episodes across the ten region--event
pairs. A seed draws the uncertain parts of a testbed: the demand and
cost parameters of the incumbent network within their stated ranges,
the shock multipliers, and each firm's availability, spare capacity, and
quote. Every policy in an episode faces the same draws, so policy
comparisons are paired, and their differences do not reflect different
luck in shocks, availability, or offers. The incumbent network and each
firm's spare-capacity and quote draws are shared by all events of a
seed, so the 95\% bootstrap percentile intervals use 4,000 resamples of
whole seeds. Seeds 1 to 20 were used for development and testbed
checks.

\subsection{Full-information value of open capacity}
\label{subsec:open-value}

Before evaluating search, we ask how much of the cost caused by capacity
scarcity could be removed if every outside offer were known. For episode
$i$, let $\Qval^{\mathrm{fixed}}_i$ be the cost of the disrupted incumbent
network before any outside capacity is admitted, and let
$\Qval^{\mathrm{full}}_i$ be its cost when every eligible carrier offer
is known (\cref{subsec:stop}). The transport open value
$V_i=\Qval^{\mathrm{fixed}}_i-\Qval^{\mathrm{full}}_i$ is the operating
saving available from the provider pool under full information. Storage
open value is defined analogously using the storage blocks of
\cref{app:testbed-offers}. Because transport and storage costs are in
different model units and the scarcity cost spans both, all shares are
computed in dollars. Both open-value
measures and the search-policy objective include quoted capacity prices
but exclude RFQ and contracting costs
(\cref{ass:fixed-offer,subsec:stop}).

We scale open value by the scarcity cost, the part of disrupted cost
caused by limited capacity rather than by additional freight. We measure
it by re-solving the disrupted network with every incumbent capacity set
to 100 times its normal level. The reduction from
$\Qval^{\mathrm{fixed}}_i$ is the scarcity cost. The remaining cost
reflects freight that must still be moved and handled.

The unrestricted reference protocol asks every firm once about all its
location-compatible keys. Across all 200 episodes, its terminal cost
differs from $\Qval^{\mathrm{full}}_i$ by at most $3.3\times10^{-6}$ in
relative terms, and no excluded block has a negative reduced cost at the
terminal price (\cref{prop:termination,app:validation}). These checks
validate the finite-pool benchmark, not search under an RFQ budget.

\begin{figure}[!htbp]
\centering
\includegraphics[width=\textwidth]{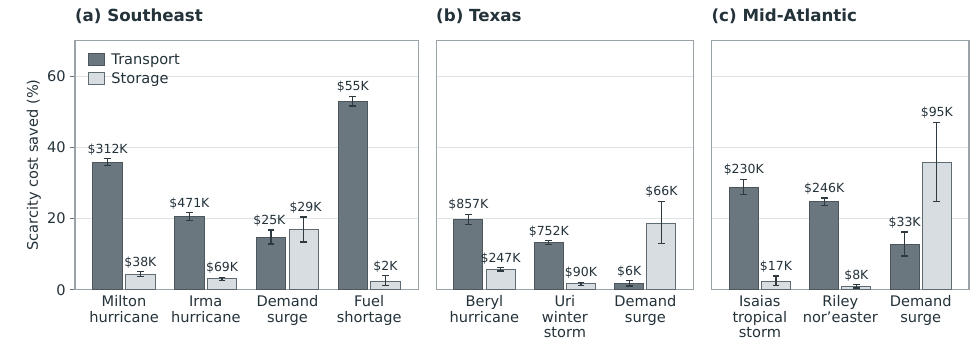}
\caption{Full-information value of outside transport and storage by
event. Each bar reports the open value summed over the 20 test seeds for
that event as a share of the corresponding total scarcity cost. Labels
give mean savings per episode, and
whiskers give the 95\% seed-bootstrap intervals. The storage bars are
synthetic full-information bounds. No sequential storage search was
run.}
\label{fig:prize}
\end{figure}

Pooled across events, the full-information value of outside transport is
25.0, 15.5, and 24.7\% of scarcity cost in the Southeast, Texas, and
the Mid-Atlantic, respectively, equivalent to US\$216,000, 538,000, and
170,000 per episode. The share is highest in the fuel shortage (53\%),
which removes capacity from lanes leaving Atlanta while carriers
elsewhere remain available (\cref{fig:prize}). In every region, outside
transport saves a larger share of scarcity cost under the concentrated
storm than under the widespread storm, although its dollar value can be
larger under the widespread storm because the underlying scarcity cost
is larger. This pattern is consistent with outside carriers being most
useful when unaffected parts of the market can fill a localized deficit,
although the testbeds build in part of it by letting carriers in hit
metros offer less (\cref{subsec:testbeds}).

When events are pooled, transport accounts for 88 to 92\% of scarcity
cost, and full-information storage offers are worth a further 4 to 6\%
of that cost. Demand surges change this mix. They add freight without
removing capacity and place much of the resulting shortage at
warehouses, so storage accounts for 53 to 92\% of combined open value in
these events (\cref{fig:prize}). Storage blocks enter the same
coordination, screening, and admission steps as carrier offers. What is
missing is provider data. No public registry identifies which firms
operate a facility, in which temperature classes, and at what scale, and
directory records and cards rarely state it (\cref{app:data-labels}).
The searches below therefore cover transport only, and the storage value
in \cref{fig:prize} is an opportunity bound rather than an estimate of
value recoverable through search.

\subsection{Learning whom to contact}
\label{subsec:prediction-results}

We first ask how well records available before contact identify which
firms can serve each transport need, and how well they rank those firms
by expected capacity, in regions excluded from training.

\paragraph{Training and model selection.}
\label{subsec:protocol}%
The training sample contains 7,914
directory firms from 41 jurisdictions outside the three test regions,
linked strictly to the motor-carrier registry (\cref{app:data}). The
structured directory fields are used alone or augmented with one of
three representations of the free-text cards available for about a
quarter of the firms: latent semantic analysis (LSA)
\citep{deerwester1990indexing}, Amazon Titan Text Embeddings v2, or the
LLM reading of \cref{subsec:match} by Claude Sonnet 5.5. We cross these
four input sets with six standard tabular learners: L2-regularized
logistic regression (LR), a random forest (RF)
\citep{breiman2001random}, a multilayer perceptron (MLP), XGBoost
\citep{chen2016xgboost}, LightGBM \citep{ke2017lightgbm}, and CatBoost
\citep{prokhorenkova2018catboost}. Each combination also has a regression on the
same inputs, fitted to the logarithm of one plus registered power units.
The directory has no fleet-size field, and fewer than 4\% of the cards
state one, so this fleet-size proxy is only a coarse size signal. It is
used only to order contacts by expected capacity (\cref{app:learning}).

Whole jurisdictions are assigned to five training folds, preventing
within-jurisdiction leakage, while the three test regions remain outside
all folds. The area under the ROC curve (AUC) measures ranking, and log loss measures the probability quality that the contact score requires
(\cref{subsec:match}). In the training states, LightGBM has the best mean
log-loss rank among the five learners compared in
\cref{tab:inputs-learners}. CatBoost came within 0.004 of the best other
boosted tree at several times its run time and was not carried further
(\cref{app:learning-estimation}). Among LightGBM's four input sets, the
LLM reading gives the highest training-state mean AUC and normalized
log-loss improvement for all firms (\cref{tab:oor-predeclared}). The
recommended policy therefore scores capability with LightGBM on the
directory fields and the LLM reading, and takes its fleet-size proxy
from the same inputs and learner. \Cref{tab:prediction-summary} reports
LightGBM for this reason. Before any test label was read, both metrics
were fixed, and all 16 retained input--learner
combinations and their fleet-size proxies were frozen and hashed,
including every trained model used in the decision experiments.

\begin{table}[!htbp]
\centering
\caption{Held-out provider ranking among firms with a card. Entries give
qualified providers per 100 firms selected from each region's top
decile, pooled across the three test regions. Except in the first row,
firms are ranked only within the group named in the row, such as
carriers, so the entries are not yields over all firms.
\Cref{tab:inputs-learners} reports results for the other learners.}
\label{tab:prediction-summary}
\footnotesize
\begin{tabular*}{\textwidth}{@{\extracolsep{\fill}}lrrrrrrr@{}}
\toprule
 & & \multicolumn{2}{c}{Reference rules} & & \multicolumn{3}{c}{LightGBM} \\
\cmidrule(lr){3-4}\cmidrule(l){6-8}
Label & Firms & Random & \makecell[r]{Industry\\code} & \makecell[r]{Zero-shot\\LLM}
 & Directory & \makecell[r]{+ Titan\\embedding} & \makecell[r]{+ LLM\\reading} \\
\midrule
Active for-hire carrier             & 1{,}944 & 52 & 62 & 75 & 88 & 87 & 89 \\
Long haul among carriers            & 1{,}002 & 73 & 78 & 90 & 92 & 90 & 92 \\
Interstate among long-haul carriers &     731 & 84 & 90 & 97 & 97 & 93 & 96 \\
Refrigerated cargo among carriers   & 1{,}001 &  8 & 10 & 30 & 20 & 24 & 29 \\
\bottomrule
\end{tabular*}
\end{table}

\paragraph{Held-out evaluation.}
\label{subsec:grid-results}%
The recommended model finds capable carriers outside its training
states. Among firms with a card, the top decile of its ranking contains
89, 92, and 96 qualified providers per 100 for the active for-hire,
long-haul, and interstate labels. Random selection gives 52, 73, and 84,
and the industry-code rule gives 62, 78, and 90
(\cref{tab:prediction-summary}). Its mean AUC is only 0.005 to 0.034
lower in the held-out regions than in the training states, and every
trained pair has lower log loss there than a constant base-rate
predictor (\cref{tab:oor-predeclared}).

Most of this performance comes from the directory fields. With
directory fields alone, LightGBM reaches 88, 92, and 97, and adding the
Titan embedding or the LLM reading changes these counts by at most four.
Text matters where the directory has no corresponding field. For
refrigerated carriage, the count rises from 20 with directory fields alone to 24 with the Titan
embedding and 29 with the LLM reading, and the same ordering holds in the
training states (\cref{tab:prediction-summary,tab:inputs-learners}).
Even without any labels, the
zero-shot LLM ranks long-haul, interstate, and refrigerated carriers
about as well as the recommended model (\cref{tab:prediction-summary}).
Its scores can therefore order first contacts for a new transport need type in the
label-free regime of \cref{subsec:match}, although its raw probabilities
need a small labeled sample for calibration
(\cref{app:learning-diagnostics}).

\label{subsec:prediction-limits}%
The fleet-size proxy still adds information to the ranking. Ranking the firms of a test
region by the recommended model's carrier probability places 21 to 26\%
of the registered trucks of capable carriers in the top tenth, against
10\% for a random ranking. Multiplying this probability by the
fleet-size proxy, as the contact score does, raises the share to 38 to
41\%. Registry capability and fleet size would place 72 to 76\% there,
so much of the remaining gap reflects information that the registry
holds but the directory lacks. These shares use the carrier
probability rather than the need-specific probability of the score and
count registered trucks rather than offers (\cref{app:learning-diagnostics}). Even complete records would not reveal event-time
availability, offered quantity, or price. The contact score can
therefore direct RFQs toward providers likely to hold capacity, but only
the RFQ and its verification determine how much a provider can deliver
and at what price. \Cref{subsec:search,subsec:decision-value} measure
the operating value of this incomplete information.

\FloatBarrier
\subsection{Limited-budget capacity discovery}
\label{subsec:search}

\label{subsec:policies}%
The discovery loop makes two choices before each RFQ: which need to ask
about and which carrier to ask. Four policies cross these choices, so
the comparisons separate what each choice adds and whether the two
reinforce each other. A random RFQ follows a pre-declared pseudo-random
order over all lane keys and so discards the coordination gaps. A random
carrier is drawn from the candidates for the need and so discards the
learned provider score. The recommended policy keeps both and is
directory-based. It uses the gaps in \eqref{eq:need-interleave} and the
LightGBM score in \eqref{eq:contact-score}, based on the directory
record and the LLM reading. All four policies follow the same
verification, quotation, screening, and admission protocol
(\cref{subsec:rfq,subsec:pricing}). Because the random RFQ choice ignores public outage maps as well as coordination output, the
gap-guided comparison gives an upper bound on what coordination adds
beyond public disruption information. For a policy $h$ whose admitted set after
$K$ RFQs has exact cost $\Qval^h_i(K)$, the recovered share is
\begin{equation}
\label{eq:recovered-share}
\operatorname{RS}_h(K)=
\frac{\sum_i\big(\Qval^{\mathrm{fixed}}_i-\Qval^h_i(K)\big)}
{\sum_iV_i},
\end{equation}
where the sums cover the events and reported seeds of a region. This
ratio weights episodes by their open value. The controlled comparisons
use exact coordination so that their differences come from what is
asked and whom, rather than from finite-precision coordination.

\paragraph{A small RFQ budget already recovers substantial value.}
The pre-declared headline uses the end-to-end mechanism, with
warm-started exchange-ADMM supplying both the gaps used to select RFQs
and the prices used to screen offers. With
80 RFQs, it recovers 35.5, 19.3, and 43.1\% of the transport open value
in the Southeast, Texas, and the Mid-Atlantic, or US\$77,000, 104,000,
and 73,000 per event, net of quoted prices and before RFQ costs. The
recommended policy admits enough providers to enlarge the 32-member
incumbent network by roughly one-third to one-half. In the Milton
episode of \cref{fig:testbeds}(a), every admitted offer serves a
corridor through Tampa, the metro that Milton hit hardest, and 11 of the
12 admitted firms are based in metros that lost less. The headline is
sensitive to the availability rate of 0.7, which is optimistic for firms
contacted for the first time after a disaster (\cref{subsec:testbeds}).
At rates of 0.5 and 0.3, the open value shrinks, and the recommended
policy recovers a somewhat smaller share of it in the Southeast and the
Mid-Atlantic and a slightly larger one in Texas (\cref{app:sensitivity}).

Replacing every ADMM call by the exact central solution changes the
recovered share by at most 0.2 percentage points at 20, 40, and 80 RFQs.
This comparison concerns the providers that each version admits,
evaluated at their exact cost. It does not show that the terminal ADMM
iterate is itself a feasible operating plan, because finite-iteration
injections need not clear exactly (\cref{app:coordination}). The two
versions usually follow different RFQ and admission paths but end with
the same admitted providers in about four of five episodes. Up to 80
RFQs, the exact curve of the recommended policy in \cref{fig:search}
thus also describes the value of the providers that the distributed
mechanism admits. On the same expanded networks, a warm start needs 44
to 54\% fewer ADMM iterations than a cold start (\cref{app:validation}).
An end-to-end episode with 80 RFQs took a median of 5.8, 7.1, and 8.3
minutes under ADMM and about 20 seconds under exact coordination
(\cref{app:testbed-settings}).

\begin{figure}[!htbp]
\centering
\includegraphics[width=\textwidth]{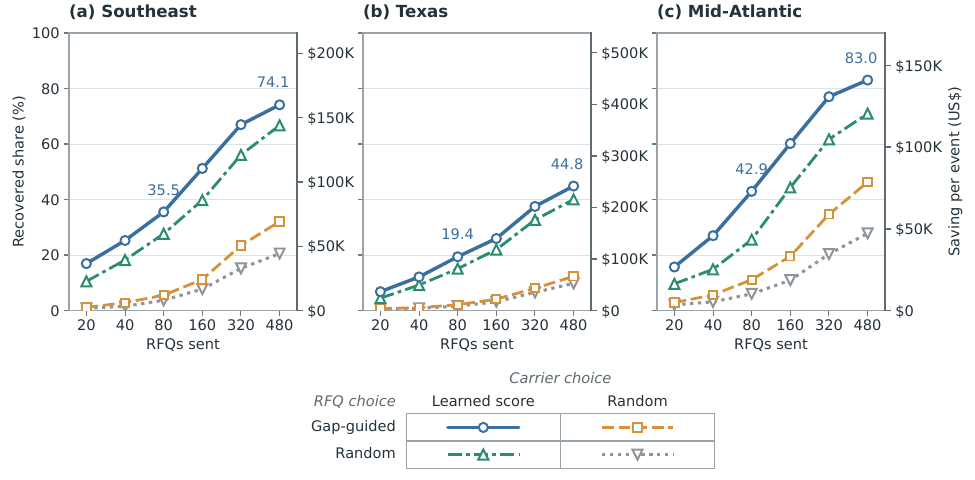}
\caption{Recovered share of the transport open value for the four
combinations of RFQ and carrier selection under exact coordination.
Each step of the horizontal axis doubles the budget, except the last.
Labels mark the exact-coordination version of the recommended policy at 80 and 480 RFQs, and the right
axes translate recovered share into the mean saving per event.
Confidence intervals and the perfect-provider-data reference are
reported in \cref{tab:search-budgets}.}
\label{fig:search}
\end{figure}

\paragraph{Learned ranking supplies most of the gain.}
Choosing both the need and the carrier at random recovers only 2 to 6\% (\cref{fig:search}).
Most of the improvement comes from learned carrier selection, which adds
17 to 32 percentage points to gap-guided search (\cref{fig:search}). About one-third of the recommended policy's RFQs reach an available
carrier that cannot serve the requested class, so prediction errors cost
RFQs mainly through the service classes that records describe least well
(\cref{app:testbed-search-results}). Gap
guidance adds less than learned selection, but on average about four times as much when
carriers are ranked as when they are drawn at random. At 80 RFQs, this
interaction is 4 to 12 percentage points, with paired 95\% intervals
excluding zero in every region, and it stays positive at availability
rates of 0.5 and 0.3 (\cref{tab:search-budgets,app:sensitivity}).
Coordination-derived scarcity information and provider intelligence are
complements when contacts are scarce. Freezing the gaps from the first
coordination call, with all other steps unchanged, lowers the recovered
share at 80 RFQs by only 0.8, 0.3, and 0.6 percentage points, and every
paired 95\% interval includes zero. Most of the value of gap guidance
thus comes from the initial needs assessment, while re-coordination
still supplies the prices that screen offers and decide admission.

Eighty RFQs represent about 1\% of asking every firm about every class
it could serve (\cref{app:testbed-search-results}). The learned ranking remains ahead of random carrier
selection through 480 RFQs in every region (\cref{fig:search}). The interaction between gap
guidance and learned ranking also weakens as more of the market is
explored. It remains positive in every region at 160 RFQs. At 480 RFQs
it is negative in the Southeast and the Mid-Atlantic, where the
learned-score policies have recovered most of the open value, and
positive in Texas, where they recover less than half. All three paired
intervals exclude zero (\cref{tab:search-budgets}).

\subsection{Operating value of provider information}
\label{subsec:decision-value}

\begin{figure}[!htbp]
\centering
\includegraphics[width=\textwidth]{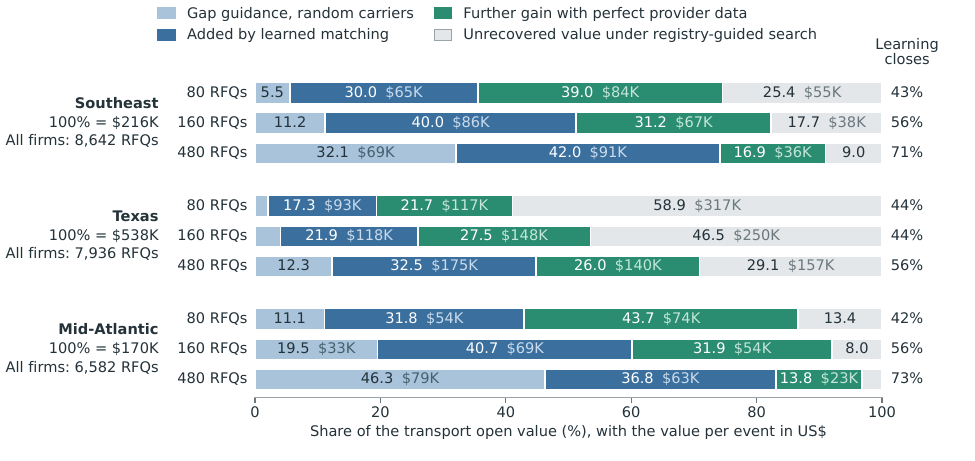}
\caption{Decomposition of the transport open value under gap-guided
RFQs and exact coordination. Perfect provider data replace the
predicted capability and the fleet-size proxy by registry values, and RFQs
still reveal availability and offers. Labels give percentage points and
mean dollars per event. Learning closes gives the share of the distance
from random carriers to perfect provider data that learned matching
closes. Under each region, 100\% is its mean open value per event, and
the RFQ count is what asking every firm about every class it could
serve would take.}
\label{fig:info-gap}
\end{figure}

Learned carrier selection supplies most of the gain in
\cref{subsec:search}. This section asks how far pre-contact information
limits the discovery loop and how it interacts with gap guidance.
Provider records affect only which provider receives the next RFQ, so
the comparisons change what the contact score knows and keep the RFQ
protocol, verification, admission, and coordination rules fixed. The
experiments compare two information regimes. Directory-based discovery,
used by every policy except the registry references, reads only the
directory record and, for about a quarter of the firms, a card.
Registry-enriched discovery replaces the predicted capability and the
fleet-size proxy by registered values. These fields do not change during
an event, so a platform can look them up before contact, but only for
reliably linked firms. Under the strict rules, 54.7\% of the directory
firms in the training jurisdictions link to the registry, and even among
firms whose primary industry is general freight trucking, 23.6\% do not.
Every pool firm links, so withholding the registry from the
directory-based policies lets us score them against known capabilities.
In the testbeds, the registry also defines what verification reveals and
how offers scale, so registry-enriched discovery coincides with the
perfect-provider-data reference and is likely optimistic for a real
lookup.

\paragraph{Provider information and outreach partly substitute for each other.}
At 80 RFQs, learned matching closes 42 to 44\% of the distance between
random carrier selection and perfect provider data
(\cref{fig:info-gap}). In the Southeast and the Mid-Atlantic, this
distance narrows as outreach expands, and by 480 RFQs learned matching
closes 71 to 73\% of it. In all three regions, the learned policy with
480 RFQs recovers about the same share that perfect provider data
recover with 80. Better information finds useful providers earlier, while a
larger contact budget reaches providers that a less accurate ranking
initially misses.

Texas shows the limit of this substitution. Beryl removes about 60\% of
the capacity at Houston, and the three corridors through Houston carry
two-thirds of the region's truck freight (\cref{fig:testbeds}(b)). Even
with perfect provider
data, 58.9\% of its open value remains unrecovered at 80 RFQs and 29.1\%
at 480, and the gain from better information does not shrink as the
budget grows. Its full-information plan admits a median of 472 of the
roughly 490 carriers that would offer, against 229 and 249 in the other
regions, while 480 RFQs with registry values reach about 300 distinct
carriers. Contact depth is therefore the most likely constraint in
Texas, although the gray segment does not isolate it, because it also
contains value that only RFQs reveal and value that a different search
order could reach with the same budget.

\begin{table}[!htbp]
\centering
\caption{Operating value of the contact-score inputs. Recovered share of
the transport open value, defined in \eqref{eq:recovered-share}, after 80
gap-guided RFQs under exact coordination, with 95\% seed-bootstrap
intervals in brackets. Policies are named as in
\cref{tab:prediction-summary}. The industry-code rule uses the
geometric-mean fleet size of its code cell. Under the trained fleet-size
proxy, each LightGBM policy uses the proxy trained on its own inputs, and
the zero-shot rule uses that of the recommended policy. With the same size
for every firm, providers are ranked by capability and the public outage
adjustment. \dag\ Recommended policy.}
\label{tab:contact-inputs}
\footnotesize
\setlength{\tabcolsep}{4pt}
\begin{tabular*}{\textwidth}{@{\extracolsep{\fill}}l r@{\hspace{3pt}}l r@{\hspace{3pt}}l r@{\hspace{3pt}}l@{}}
\toprule
 & \multicolumn{2}{c}{Southeast} & \multicolumn{2}{c}{Texas} & \multicolumn{2}{c}{Mid-Atlantic} \\
\midrule
\multicolumn{7}{@{}l}{Reference rules} \\
\quad Random & 5.5 & {\scriptsize[4.7, 6.6]} & 2.1 & {\scriptsize[1.6, 2.7]} & 11.1 & {\scriptsize[7.9, 15.8]} \\
\quad Industry code & 11.4 & {\scriptsize[10.4, 12.6]} & 4.2 & {\scriptsize[3.1, 5.6]} & 12.4 & {\scriptsize[10.0, 15.4]} \\
\addlinespace[3pt]
\multicolumn{7}{@{}l}{Same size for every firm} \\
\quad Zero-shot LLM capability & 13.3 & {\scriptsize[12.2, 14.4]} & 12.6 & {\scriptsize[10.3, 15.2]} & 22.1 & {\scriptsize[20.1, 24.4]} \\
\quad LightGBM, + LLM reading & 19.5 & {\scriptsize[18.3, 20.8]} & 10.0 & {\scriptsize[8.3, 11.8]} & 31.2 & {\scriptsize[28.7, 33.8]} \\
\addlinespace[3pt]
\multicolumn{7}{@{}l}{Trained fleet-size proxy} \\
\quad Zero-shot LLM capability & 32.6 & {\scriptsize[30.0, 35.3]} & 19.6 & {\scriptsize[17.4, 21.7]} & 48.1 & {\scriptsize[45.0, 51.3]} \\
\quad LightGBM, directory & 35.8 & {\scriptsize[33.4, 38.3]} & 15.2 & {\scriptsize[13.5, 17.0]} & 42.6 & {\scriptsize[39.4, 45.9]} \\
\quad LightGBM, + Titan embedding & 34.0 & {\scriptsize[31.8, 36.0]} & 15.4 & {\scriptsize[13.5, 17.3]} & 43.7 & {\scriptsize[41.0, 46.6]} \\
\quad LightGBM, + LLM reading\,\dag & 35.5 & {\scriptsize[33.0, 38.1]} & 19.4 & {\scriptsize[17.4, 21.3]} & 42.9 & {\scriptsize[40.2, 46.1]} \\
\addlinespace[3pt]
Registry capability and fleet size & 74.6 & {\scriptsize[72.0, 77.0]} & 41.1 & {\scriptsize[38.8, 43.5]} & 86.6 & {\scriptsize[82.3, 90.2]} \\
\bottomrule
\end{tabular*}
\end{table}

\Cref{tab:contact-inputs} varies what the contact score reads. A lookup
table of capability rates and fleet sizes by industry code recovers a
third or less of the recommended policy. Adding the LLM reading or the
Titan embedding to the directory model changes the recovered share by
$-1.8$ to $+4.2$ percentage points, and only the LLM reading in Texas
gives a clear gain (\cref{app:testbed-search-results}). The LLM reading
improves the ranking of refrigerated carriers
(\cref{subsec:prediction-results}), but this adds little operating value
while capacity scale remains uncertain. Removing the coarse size signal
has a larger effect. Assigning every firm the same size lowers the
recovered share of the recommended model from 35.5, 19.4, and 42.9\% to
19.5, 10.0, and 31.2\%, and every paired 95\% interval of the difference
excludes zero. This result relies on the testbed rule that expected
offers scale with registered fleet size
(\cref{subsec:testbeds,app:sensitivity}). Zero-shot LLM capability
scores paired with the trained proxy approach or exceed the operating
value of the recommended policy but recover much less when every firm
has the same size. A new transport need type can therefore start without
capability labels, but its ranking relies on a size signal. Where
authoritative records exist, they can replace predictions field by
field. Replacing either predicted input alone by its registry value
already recovers more than the recommended policy in every region
(\cref{tab:info-value-full}).

\paragraph{Better information strengthens guidance when contacts are scarce.}
At 20 RFQs, gap guidance adds at most 0.8
percentage points when carriers are drawn at random, 2.3 to 6.4 points
with the learned score, and 4.4 to 25.6 points with registry values. At
40 RFQs, the registry effect is again the largest in every region
(\cref{tab:gap-by-info}). A more accurate ranking makes it more likely
that an RFQ for the scarcest capacity reaches a firm that can supply it.
The gain is bounded by the value left to recover. At 80 RFQs, the
registry effect falls below the learned-score effect in the Southeast
and the Mid-Atlantic, where registry values have already recovered about
75 and 87\% of the open value, but stays above it in Texas, where they
have recovered 41\%. Learning and gap guidance thus reinforce each other
while contacts are scarce relative to the capacity still missing.

\section{Conclusions}\label{sec:conclusion}

Open capacity pooling is a budgeted provider-discovery problem that
links which capacity need to pursue, whom to contact, and which verified
offer to admit. Incumbent firms coordinate through exchange-ADMM without
disclosing their operating models. The resulting gaps select the
shortfall to pursue, a matching layer built from provider records ranks
whom to ask, and RFQs reveal verified offers that the coordination
prices screen for admission.
With fixed offers and exact coordination, full revelation of a finite
provider pool reaches the optimum over its eligible offers, while search
under an RFQ budget remains heuristic. In the three regional testbeds,
80 RFQs, about 1\% of exhaustive search, recover 19 to 43\% of the
full-information value of outside transport, against 2 to 6\% when both
the need and the provider are chosen at random. Through 80 RFQs, the
providers admitted under warm-started ADMM, evaluated by exact
re-solving, recover within 0.2 percentage points of the
exact-coordination policy.

Learned carrier ranking supplies most of the gain under tight budgets,
and gap guidance adds about four times as much when carriers are ranked
as when they are drawn at random. At the tightest budgets, gap guidance
adds more as provider information improves. This complementarity weakens
as recovered value approaches the full-information level. Provider
information and outreach are also partial substitutes. In these
testbeds, the learned policy with 480 RFQs and registry-enriched
discovery with 80 recover similar shares. Texas marks the limit of this
substitution. Its full-information plan uses nearly every carrier that
would offer, so contact depth remains the most likely constraint even
with registry data.

Within transport, the fitted layer retains most of its ranking accuracy
among linked firms in the held-out regions. Structured directory fields
do most of the work for common carrier capabilities. LLM reading helps
most where those fields omit a capability, such as refrigerated
carriage, but its operating value varies by region. Zero-shot capability
scores can provide an initial ranking before labels accumulate, although
they need a small labeled sample for calibration and an existing size
signal. Neither directory records nor registry fields reveal event-time
availability, offered quantity, or price, so an RFQ remains necessary.
Storage holds most of the open value in demand surges and fits the same
coordination and admission architecture, but storage discovery requires
its own provider records, capability labels, and size measure.

The evidence remains conditional on the testbeds. Provider availability,
offered quantities, and prices are simulated, RFQ and contracting costs
are not deducted from operating savings, and expected offers scale with
registered fleet size. The prediction and decision results cover
firms that link to the registry under strict rules, not unlinked
directory firms, and storage enters only through a synthetic
full-information bound. All three testbeds are in the United States,
where motor carriers register publicly. In deployment, reliable authoritative records
should replace predictions wherever available, while learned matching
can fill missing fields. Platforms could collect standardized size
information at onboarding, since directory records rarely contain it.
RFQ logs could supply availability, verified capability, and
offered-quantity outcomes. Because those logs
cover selected firms, retraining would require exploration or correction
for selection. A field pilot should replace the simulated offer layer
with observed responses and test whether previously unengaged providers
become usable disruption capacity.

\section*{Data availability}
The provider records combine the Small Business Search directory
\citep{sba2026sbs}, the FMCSA Company Census \citep{fmcsa2026census},
and SAFER company snapshots \citep{fmcsa2026safer}. The networks and
shocks use County Business Patterns \citep{census2025cbp}, the Freight
Analysis Framework \citep{bts2025faf}, and the EAGLE-I outage records
\citep{ornl2025eaglei}, and the dollar anchors follow
\citet{atri2026costs} and \citet{tfa2025warehousing}. The references
give the version and retrieval date of each source. \Cref{app:data}
specifies the strict linkage rules, and names, contact details, and
identifying numbers are masked in every card before it reaches an LLM.
The code, the linked provider records without contact details, the
features derived from the cards, the frozen models with their hashes,
and the episode-level results are available from the corresponding
author upon reasonable request. The free-text cards are written by the
firms and are not redistributed, apart from the abridged excerpt in
\cref{app:data-example}, but the firm identifiers and search
codes needed to retrieve them from the Small Business Search are
included.

\section*{Declaration of competing interest}
The authors declare that they have no known competing financial
interests or personal relationships that could have appeared to
influence the work reported in this paper.

\section*{Funding}
This research did not receive any specific grant from funding agencies
in the public, commercial, or not-for-profit sectors.

\section*{CRediT authorship contribution statement}
\textbf{Yujia Xu:} Conceptualization, Methodology, Investigation,
Writing -- original draft. \textbf{Walid Klibi:} Conceptualization,
Methodology, Supervision, Writing -- review \& editing.
\textbf{Benoit Montreuil:} Conceptualization, Methodology, Supervision,
Writing -- review \& editing. All authors approved the final version and
agree to be accountable for all aspects of the work.

\section*{Declaration of generative AI and AI-assisted technologies in the manuscript preparation process}

Language models are part of the method: Claude Sonnet 5.5 reads
directory records and provider self-descriptions, GPT-6.1 Sol repeats
these readings in robustness checks, and Amazon Titan Text Embeddings v2
provides an alternative text input (\cref{subsec:match,app:learning}).
Model identifiers are given in \cref{app:learning}, and every prompt and
response is cached. Claude Opus 5.5 and GPT 5.6 Sol were used to help
with software development, language improvement, and consistency
checks. The authors have reviewed and edited all content, have checked
the terms of use of each tool, and assume full responsibility for the
accuracy, integrity, and originality of the work.

\bibliography{reference}

\begin{thebibliography}{93}
\providecommand{\natexlab}[1]{#1}
\providecommand{\url}[1]{\texttt{#1}}
\expandafter\ifx\csname urlstyle\endcsname\relax
  \providecommand{\doi}[1]{doi: #1}\else
  \providecommand{\doi}{doi: \begingroup \urlstyle{rm}\Url}\fi

\bibitem[Acocella et~al.(2020)Acocella, Caplice, and
  Sheffi]{acocella2020elephants}
A.~Acocella, C.~Caplice, and Y.~Sheffi.
\newblock Elephants or goldfish?: An empirical analysis of carrier reciprocity
  in dynamic freight markets.
\newblock \emph{Transportation Research Part E: Logistics and Transportation
  Review}, 142:\penalty0 102073, 2020.
\newblock \doi{10.1016/j.tre.2020.102073}.

\bibitem[Acocella et~al.(2022)Acocella, Caplice, and Sheffi]{acocella2022setit}
A.~Acocella, C.~Caplice, and Y.~Sheffi.
\newblock The end of ``set it and forget it'' pricing? {O}pportunities for
  market-based freight contracts.
\newblock arXiv preprint 2202.02367; MIT DSpace
  \url{https://dspace.mit.edu/handle/1721.1/141351}, 2022.

\bibitem[Akbarpour et~al.(2026)Akbarpour, Ropke, Rasouli, Anker~Nielsen, Liao,
  and Jiang]{akbarpour2026choice}
M.~Akbarpour, S.~Ropke, S.~Rasouli, O.~Anker~Nielsen, F.~Liao, and Y.~Jiang.
\newblock A choice-based optimization framework for crowdsourced last-mile
  delivery.
\newblock \emph{Transportation Research Part E: Logistics and Transportation
  Review}, 211:\penalty0 104844, 2026.
\newblock \doi{10.1016/j.tre.2026.104844}.

\bibitem[Akiba et~al.(2019)Akiba, Sano, Yanase, Ohta, and
  Koyama]{akiba2019optuna}
T.~Akiba, S.~Sano, T.~Yanase, T.~Ohta, and M.~Koyama.
\newblock Optuna: A next-generation hyperparameter optimization framework.
\newblock In \emph{Proceedings of the 25th ACM SIGKDD International Conference
  on Knowledge Discovery and Data Mining}, pages 2623--2631, 2019.
\newblock \doi{10.1145/3292500.3330701}.

\bibitem[{American Transportation Research Institute}(2026)]{atri2026costs}
{American Transportation Research Institute}.
\newblock An analysis of the operational costs of trucking: 2026 update.
\newblock Technical report, American Transportation Research Institute, 2026.
\newblock URL
  \url{https://truckingresearch.org/about-atri/atri-research/operational-costs-of-trucking/}.
\newblock Accessed 19 October 2026.

\bibitem[Arslan et~al.(2019)Arslan, Agatz, Kroon, and
  Zuidwijk]{arslan2019crowdsourced}
A.~M. Arslan, N.~Agatz, L.~Kroon, and R.~Zuidwijk.
\newblock Crowdsourced delivery: A dynamic pickup and delivery problem with ad
  hoc drivers.
\newblock \emph{Transportation Science}, 53\penalty0 (1):\penalty0 222--235,
  2019.
\newblock \doi{10.1287/trsc.2017.0803}.

\bibitem[Ausseil et~al.(2024)Ausseil, Ulmer, and Pazour]{ausseil2024online}
R.~Ausseil, M.~W. Ulmer, and J.~A. Pazour.
\newblock Online acceptance probability approximation in peer-to-peer
  transportation.
\newblock \emph{Omega}, 123:\penalty0 102993, 2024.
\newblock \doi{10.1016/j.omega.2023.102993}.

\bibitem[Aybat and Yazdandoost~Hamedani(2019)]{aybat2019distributed}
N.~S. Aybat and E.~Yazdandoost~Hamedani.
\newblock A distributed {ADMM}-like method for resource sharing over
  time-varying networks.
\newblock \emph{SIAM Journal on Optimization}, 29\penalty0 (4):\penalty0
  3036--3068, 2019.
\newblock \doi{10.1137/17M1151973}.

\bibitem[Barnhart et~al.(1998)Barnhart, Johnson, Nemhauser, Savelsbergh, and
  Vance]{barnhart1998branchprice}
C.~Barnhart, E.~L. Johnson, G.~L. Nemhauser, M.~W.~P. Savelsbergh, and P.~H.
  Vance.
\newblock Branch-and-price: Column generation for solving huge integer
  programs.
\newblock \emph{Operations Research}, 46\penalty0 (3):\penalty0 316--329, 1998.
\newblock \doi{10.1287/opre.46.3.316}.

\bibitem[Behrendt et~al.(2023)Behrendt, Savelsbergh, and
  Wang]{behrendt2023prescriptive}
A.~Behrendt, M.~Savelsbergh, and H.~Wang.
\newblock A prescriptive machine learning method for courier scheduling on
  crowdsourced delivery platforms.
\newblock \emph{Transportation Science}, 57\penalty0 (4):\penalty0 889--907,
  2023.
\newblock \doi{10.1287/trsc.2022.1152}.

\bibitem[Beirigo et~al.(2022)Beirigo, Schulte, and
  Negenborn]{beirigo2022learning}
B.~A. Beirigo, F.~Schulte, and R.~R. Negenborn.
\newblock A learning-based optimization approach for autonomous ridesharing
  platforms with service-level contracts and on-demand hiring of idle vehicles.
\newblock \emph{Transportation Science}, 56\penalty0 (3):\penalty0 677--703,
  2022.
\newblock \doi{10.1287/trsc.2021.1069}.

\bibitem[Bergstra et~al.(2011)Bergstra, Bardenet, Bengio, and
  K{\'e}gl]{bergstra2011algorithms}
J.~Bergstra, R.~Bardenet, Y.~Bengio, and B.~K{\'e}gl.
\newblock Algorithms for hyper-parameter optimization.
\newblock In \emph{Advances in Neural Information Processing Systems 24}, pages
  2546--2554, 2011.

\bibitem[Beven~II et~al.(2025{\natexlab{a}})Beven~II, Alaka, and
  Fritz]{beven2025milton}
J.~L. Beven~II, L.~Alaka, and C.~Fritz.
\newblock Hurricane {Milton} ({AL142024}), 5--10 {October} 2024.
\newblock Tropical cyclone report, National Hurricane Center,
  2025{\natexlab{a}}.
\newblock URL \url{https://www.nhc.noaa.gov/data/tcr/AL142024_Milton.pdf}.
\newblock Accessed 15 October 2026.

\bibitem[Beven~II et~al.(2025{\natexlab{b}})Beven~II, Fritz, and
  Alaka]{beven2025beryl}
J.~L. Beven~II, C.~Fritz, and L.~Alaka.
\newblock Hurricane {Beryl} ({AL022024}), 28 {June}--9 {July} 2024.
\newblock Tropical cyclone report, National Hurricane Center,
  2025{\natexlab{b}}.
\newblock URL \url{https://www.nhc.noaa.gov/data/tcr/AL022024_Beryl.pdf}.
\newblock Accessed 16 October 2026.

\bibitem[Boyd and Vandenberghe(2004)]{boyd2004convex}
S.~Boyd and L.~Vandenberghe.
\newblock \emph{Convex Optimization}.
\newblock Cambridge University Press, Cambridge, UK, 2004.
\newblock \doi{10.1017/CBO9780511804441}.

\bibitem[Boyd et~al.(2011)Boyd, Parikh, Chu, Peleato, and
  Eckstein]{boyd2011admm}
S.~Boyd, N.~Parikh, E.~Chu, B.~Peleato, and J.~Eckstein.
\newblock Distributed optimization and statistical learning via the alternating
  direction method of multipliers.
\newblock \emph{Foundations and Trends in Machine Learning}, 3\penalty0
  (1):\penalty0 1--122, 2011.
\newblock \doi{10.1561/2200000016}.

\bibitem[Breiman(2001)]{breiman2001random}
L.~Breiman.
\newblock Random forests.
\newblock \emph{Machine Learning}, 45\penalty0 (1):\penalty0 5--32, 2001.
\newblock \doi{10.1023/A:1010933404324}.

\bibitem[Brelsford et~al.(2024)Brelsford, Tennille, Myers, Chinthavali,
  Tansakul, Denman, Coletti, Grant, Lee, Allen, Johnson, Huihui, Hamaker,
  Newby, Medlen, Maguire, Dunivan~Stahl, Moehl, Redmon, Sanyal, and
  Bhaduri]{brelsford2024eaglei}
C.~Brelsford, S.~Tennille, A.~Myers, S.~Chinthavali, V.~Tansakul, M.~Denman,
  M.~Coletti, J.~Grant, S.~Lee, K.~Allen, E.~Johnson, J.~Huihui, A.~Hamaker,
  S.~Newby, K.~Medlen, D.~Maguire, C.~Dunivan~Stahl, J.~Moehl, D.~Redmon,
  J.~Sanyal, and B.~Bhaduri.
\newblock A dataset of recorded electricity outages by {United States} county
  2014--2022.
\newblock \emph{Scientific Data}, 11:\penalty0 271, 2024.
\newblock \doi{10.1038/s41597-024-03095-5}.

\bibitem[{Bureau of Transportation Statistics and Federal Highway
  Administration}(2025)]{bts2025faf}
{Bureau of Transportation Statistics and Federal Highway Administration}.
\newblock {Freight Analysis Framework}, version 5.7.1 (2018--2024).
\newblock \url{https://www.bts.gov/faf}, 2025.
\newblock Accessed 14 October 2026.

\bibitem[Cangialosi et~al.(2018)Cangialosi, Latto, and
  Berg]{cangialosi2018irma}
J.~P. Cangialosi, A.~S. Latto, and R.~Berg.
\newblock Hurricane {Irma} ({AL112017}), 30 {August}--12 {September} 2017.
\newblock Tropical cyclone report, National Hurricane Center, 2018.
\newblock URL \url{https://www.nhc.noaa.gov/data/tcr/AL112017_Irma.pdf}.
\newblock Accessed 15 October 2026.

\bibitem[Caplice(2007)]{caplice2007electronic}
C.~Caplice.
\newblock Electronic markets for truckload transportation.
\newblock \emph{Production and Operations Management}, 16\penalty0
  (4):\penalty0 423--436, 2007.
\newblock \doi{10.1111/j.1937-5956.2007.tb00270.x}.

\bibitem[Ceschia et~al.(2023)Ceschia, Gansterer, Mancini, and
  Meneghetti]{ceschia2023ondemand}
S.~Ceschia, M.~Gansterer, S.~Mancini, and A.~Meneghetti.
\newblock The on-demand warehousing problem.
\newblock \emph{International Journal of Production Research}, 61\penalty0
  (10):\penalty0 3152--3170, 2023.
\newblock \doi{10.1080/00207543.2022.2078249}.

\bibitem[Chen and Guestrin(2016)]{chen2016xgboost}
T.~Chen and C.~Guestrin.
\newblock {XGBoost}: A scalable tree boosting system.
\newblock In \emph{Proceedings of the 22nd ACM SIGKDD International Conference
  on Knowledge Discovery and Data Mining}, pages 785--794, 2016.
\newblock \doi{10.1145/2939672.2939785}.

\bibitem[Coscrato and Bridge(2024)]{coscrato2024supplier}
V.~Coscrato and D.~Bridge.
\newblock Supplier recommendation in online procurement.
\newblock arXiv preprint 2403.01301, 2024.

\bibitem[{Cushman \& Wakefield}(2026)]{cw2026industrial}
{Cushman \& Wakefield}.
\newblock Industrial {MarketBeat} {Q2} 2026.
\newblock
  \url{https://www.cushmanwakefield.com/en/united-states/insights/us-marketbeats},
  2026.
\newblock Accessed 15--16 October 2026.

\bibitem[Dantzig and Wolfe(1960)]{dantzig1960decomposition}
G.~B. Dantzig and P.~Wolfe.
\newblock Decomposition principle for linear programs.
\newblock \emph{Operations Research}, 8\penalty0 (1):\penalty0 101--111, 1960.
\newblock \doi{10.1287/opre.8.1.101}.

\bibitem[{DAT Freight \& Analytics}(2026)]{dat2026contract}
{DAT Freight \& Analytics}.
\newblock Contract van and reefer rates make record {June}-to-{July} gains.
\newblock
  \url{https://www.dat.com/company/news-events/news-releases/dat-contract-van-and-reefer-rates-make-record-june-to-july-gains},
  2026.
\newblock Accessed 19 October 2026.

\bibitem[Deerwester et~al.(1990)Deerwester, Dumais, Furnas, Landauer, and
  Harshman]{deerwester1990indexing}
S.~Deerwester, S.~T. Dumais, G.~W. Furnas, T.~K. Landauer, and R.~Harshman.
\newblock Indexing by latent semantic analysis.
\newblock \emph{Journal of the American Society for Information Science},
  41\penalty0 (6):\penalty0 391--407, 1990.
\newblock \doi{10.1002/(SICI)1097-4571(199009)41:6<391::AID-ASI1>3.0.CO;2-9}.

\bibitem[Elmachtoub and Grigas(2022)]{elmachtoub2022smart}
A.~N. Elmachtoub and P.~Grigas.
\newblock Smart ``predict, then optimize''.
\newblock \emph{Management Science}, 68\penalty0 (1):\penalty0 9--26, 2022.
\newblock \doi{10.1287/mnsc.2020.3922}.

\bibitem[Faug{\`e}re et~al.(2022)Faug{\`e}re, Klibi, White~III, and
  Montreuil]{faugere2022dynamic}
L.~Faug{\`e}re, W.~Klibi, C.~White~III, and B.~Montreuil.
\newblock Dynamic pooled capacity deployment for urban parcel logistics.
\newblock \emph{European Journal of Operational Research}, 303\penalty0
  (2):\penalty0 650--667, 2022.
\newblock \doi{10.1016/j.ejor.2022.02.051}.

\bibitem[{Federal Energy Regulatory Commission} et~al.(2021){Federal Energy
  Regulatory Commission}, {North American Electric Reliability Corporation},
  and {Regional Entities}]{ferc2021coldweather}
{Federal Energy Regulatory Commission}, {North American Electric Reliability
  Corporation}, and {Regional Entities}.
\newblock The {February} 2021 cold weather outages in {Texas} and the {South
  Central United States}.
\newblock Technical report, FERC, NERC and Regional Entity Staff Report, Nov.
  2021.
\newblock URL
  \url{https://www.ferc.gov/media/february-2021-cold-weather-outages-texas-and-south-central-united-states-ferc-nerc-and}.
\newblock Accessed 16 October 2026.

\bibitem[{Federal Motor Carrier Safety
  Administration}(2026{\natexlab{a}})]{fmcsa2026census}
{Federal Motor Carrier Safety Administration}.
\newblock Company census file.
\newblock
  \url{https://data.transportation.gov/Trucking-and-Motorcoaches/Company-Census-File/az4n-8mr2},
  2026{\natexlab{a}}.
\newblock Retrieved 14 October 2026.

\bibitem[{Federal Motor Carrier Safety
  Administration}(2026{\natexlab{b}})]{fmcsa2026safer}
{Federal Motor Carrier Safety Administration}.
\newblock {SAFER} company snapshot.
\newblock \url{https://safer.fmcsa.dot.gov/CompanySnapshot.aspx},
  2026{\natexlab{b}}.
\newblock Retrieved 14--16 October 2026.

\bibitem[{FreightWaves}(2019)]{freightwaves2019spot}
{FreightWaves}.
\newblock Who's really hauling spot market freight?
\newblock
  \url{https://www.freightwaves.com/news/\%EF\%BB\%BFwhos-really-hauling-spot-market-freight},
  2019.
\newblock Accessed 17 October 2026.

\bibitem[{FreightWaves}(2026)]{freightwaves2026knightswift}
{FreightWaves}.
\newblock {Knight-Swift} {Q2}: Why this carrier's profit beat signals market
  shift.
\newblock
  \url{https://www.freightwaves.com/news/knight-swift-q2-why-this-carriers-profit-beat-signals-market-shift},
  2026.
\newblock Accessed 17 October 2026.

\bibitem[Gorishniy et~al.(2021)Gorishniy, Rubachev, Khrulkov, and
  Babenko]{gorishniy2021revisiting}
Y.~Gorishniy, I.~Rubachev, V.~Khrulkov, and A.~Babenko.
\newblock Revisiting deep learning models for tabular data.
\newblock In \emph{Advances in Neural Information Processing Systems 34}, 2021.

\bibitem[Grinsztajn et~al.(2022)Grinsztajn, Oyallon, and
  Varoquaux]{grinsztajn2022tree}
L.~Grinsztajn, E.~Oyallon, and G.~Varoquaux.
\newblock Why do tree-based models still outperform deep learning on typical
  tabular data?
\newblock In \emph{Advances in Neural Information Processing Systems 35}, pages
  507--520, 2022.
\newblock \doi{10.52202/068431-0037}.

\bibitem[Guo and Wu(2018)]{guo2018capacity}
L.~Guo and X.~Wu.
\newblock Capacity sharing between competitors.
\newblock \emph{Management Science}, 64\penalty0 (8):\penalty0 3554--3573,
  2018.
\newblock \doi{10.1287/mnsc.2017.2796}.

\bibitem[He et~al.(2024)He, Jin, Tr{\'e}panier, and Schulte]{he2024column}
P.~He, J.~G. Jin, M.~Tr{\'e}panier, and F.~Schulte.
\newblock A column-generation matheuristic approach for optimizing first-mile
  ridesharing services with publicly- and privately-owned autonomous vehicles.
\newblock \emph{Transportation Research Part C: Emerging Technologies},
  160:\penalty0 104516, 2024.
\newblock \doi{10.1016/j.trc.2024.104516}.

\bibitem[Hegselmann et~al.(2023)Hegselmann, Buendia, Lang, Agrawal, Jiang, and
  Sontag]{hegselmann2023tabllm}
S.~Hegselmann, A.~Buendia, H.~Lang, M.~Agrawal, X.~Jiang, and D.~Sontag.
\newblock {TabLLM}: Few-shot classification of tabular data with large language
  models.
\newblock In \emph{Proceedings of the 26th International Conference on
  Artificial Intelligence and Statistics}, volume 206 of \emph{Proceedings of
  Machine Learning Research}, pages 5549--5581. PMLR, 2023.

\bibitem[Holgu{\'i}n-Veras et~al.(2013)Holgu{\'i}n-Veras, P{\'e}rez, Jaller,
  Van~Wassenhove, and Aros-Vera]{holguinveras2013objective}
J.~Holgu{\'i}n-Veras, N.~P{\'e}rez, M.~Jaller, L.~N. Van~Wassenhove, and
  F.~Aros-Vera.
\newblock On the appropriate objective function for post-disaster humanitarian
  logistics models.
\newblock \emph{Journal of Operations Management}, 31\penalty0 (5):\penalty0
  262--280, 2013.
\newblock \doi{10.1016/j.jom.2013.06.002}.

\bibitem[Horner et~al.(2021)Horner, Pazour, and Mitchell]{horner2021optimizing}
H.~Horner, J.~Pazour, and J.~E. Mitchell.
\newblock Optimizing driver menus under stochastic selection behavior for
  ridesharing and crowdsourced delivery.
\newblock \emph{Transportation Research Part E: Logistics and Transportation
  Review}, 153:\penalty0 102419, 2021.
\newblock \doi{10.1016/j.tre.2021.102419}.

\bibitem[Hu et~al.(2025)Hu, Kong, and Jia]{hu2025supplier}
K.~Hu, L.~Kong, and Z.~Jia.
\newblock Supplier selection criteria under heterogeneous sourcing needs:
  Evidence from an online marketplace for selling production capacity.
\newblock \emph{Production and Operations Management}, 34\penalty0
  (2):\penalty0 168--186, 2025.
\newblock \doi{10.1177/10591478241279384}.

\bibitem[Hu et~al.(2024)Hu, Dong, and Dai]{hu2024machine}
S.~Hu, Z.~S. Dong, and R.~Dai.
\newblock A machine learning based sample average approximation for supplier
  selection with option contract in humanitarian relief.
\newblock \emph{Transportation Research Part E: Logistics and Transportation
  Review}, 186:\penalty0 103531, 2024.
\newblock \doi{10.1016/j.tre.2024.103531}.

\bibitem[Huang et~al.(2025)Huang, Tang, Hu, Jiang, Zheng, Ge, Wang, and
  Wang]{huang2025orlm}
C.~Huang, Z.~Tang, S.~Hu, R.~Jiang, X.~Zheng, D.~Ge, B.~Wang, and Z.~Wang.
\newblock {ORLM}: A customizable framework in training large models for
  automated optimization modeling.
\newblock \emph{Operations Research}, 73\penalty0 (6):\penalty0 2986--3009,
  2025.
\newblock \doi{10.1287/opre.2024.1233}.

\bibitem[Ivanov et~al.(2017)Ivanov, Dolgui, Sokolov, and
  Ivanova]{ivanov2017literature}
D.~Ivanov, A.~Dolgui, B.~Sokolov, and M.~Ivanova.
\newblock Literature review on disruption recovery in the supply chain.
\newblock \emph{International Journal of Production Research}, 55\penalty0
  (20):\penalty0 6158--6174, 2017.
\newblock \doi{10.1080/00207543.2017.1330572}.

\bibitem[Jia et~al.(2025)Jia, Schrotenboer, and Chen]{jia2025scenario}
M.~Jia, A.~H. Schrotenboer, and F.~Chen.
\newblock Scenario predict-then-optimize for data-driven online inventory
  routing.
\newblock \emph{Transportation Science}, 59\penalty0 (5):\penalty0 1032--1056,
  2025.
\newblock \doi{10.1287/trsc.2024.0613}.

\bibitem[Jiao et~al.(2026)Jiao, Gong, Li, Lam, Shu, Wang, and
  Yang]{jiao2026llm4stp}
H.~Jiao, J.~Gong, H.~Li, J.~S.~L. Lam, Y.~Shu, J.~Wang, and Z.~Yang.
\newblock {LLM4STP}: A large language model-driven multi-feature fusion method
  for ship trajectory prediction.
\newblock \emph{Transportation Research Part E: Logistics and Transportation
  Review}, 207:\penalty0 104599, 2026.
\newblock \doi{10.1016/j.tre.2025.104599}.

\bibitem[Jordan and Graves(1995)]{jordan1995chaining}
W.~C. Jordan and S.~C. Graves.
\newblock Principles on the benefits of manufacturing process flexibility.
\newblock \emph{Management Science}, 41\penalty0 (4):\penalty0 577--594, 1995.
\newblock \doi{10.1287/mnsc.41.4.577}.

\bibitem[Ke et~al.(2017)Ke, Meng, Finley, Wang, Chen, Ma, Ye, and
  Liu]{ke2017lightgbm}
G.~Ke, Q.~Meng, T.~Finley, T.~Wang, W.~Chen, W.~Ma, Q.~Ye, and T.-Y. Liu.
\newblock {LightGBM}: A highly efficient gradient boosting decision tree.
\newblock In \emph{Advances in Neural Information Processing Systems 30}, pages
  3146--3154, 2017.

\bibitem[Kim et~al.(2026)Kim, Kim, Park, Park, and Bae]{kim2026generative}
M.~Kim, T.~Kim, T.~Park, H.~Park, and H.~Bae.
\newblock Generative {AI} and machine learning collaboration for container
  dwell time prediction via data standardization.
\newblock \emph{Transportation Research Part E: Logistics and Transportation
  Review}, 216:\penalty0 105171, 2026.
\newblock \doi{10.1016/j.tre.2026.105171}.

\bibitem[Kulpa et~al.(2026)Kulpa, Balthrop, Miller, Thomas, and
  Waller]{kulpa2026synthetic}
T.~Kulpa, A.~Balthrop, J.~W. Miller, R.~Thomas, and M.~Waller.
\newblock A synthetic control approach to examining the effects of major
  hurricane landfalls on the {US} long-haul dry van truckload spot market.
\newblock \emph{Transportation Research Part E: Logistics and Transportation
  Review}, 209:\penalty0 104765, 2026.
\newblock \doi{10.1016/j.tre.2026.104765}.

\bibitem[Latto et~al.(2021)Latto, Hagen, and Berg]{latto2021isaias}
A.~Latto, A.~Hagen, and R.~Berg.
\newblock Hurricane {Isaias} ({AL092020}), 30 {July}--4 {August} 2020.
\newblock Tropical cyclone report, National Hurricane Center, 2021.
\newblock URL \url{https://www.nhc.noaa.gov/data/tcr/AL092020_Isaias.pdf}.
\newblock Accessed 16 October 2026.

\bibitem[Li et~al.(2025)Li, Ko, and Ameri]{li2025supplier}
Y.~Li, H.~Ko, and F.~Ameri.
\newblock Integrating graph retrieval-augmented generation with large language
  models for supplier discovery.
\newblock \emph{Journal of Computing and Information Science in Engineering},
  25\penalty0 (2):\penalty0 021010, 2025.
\newblock \doi{10.1115/1.4067389}.

\bibitem[Liu et~al.(2024)Liu, Xu, and Montreuil]{liu2024containerized}
X.~Liu, Y.~Xu, and B.~Montreuil.
\newblock Dynamic containerized modular capacity planning and resource
  allocation in hyperconnected supply chain ecosystems.
\newblock In \emph{10th International Physical Internet Conference (IPIC
  2024)}, Savannah, GA, USA, May 2024.
\newblock \doi{10.35090/gatech/11077}.
\newblock URL \url{https://hdl.handle.net/1853/75426}.

\bibitem[Liu et~al.(2025)Liu, Li, Dahan, and Montreuil]{liu2025dynamic}
X.~Liu, J.~Li, M.~Dahan, and B.~Montreuil.
\newblock Dynamic hub capacity planning in hyperconnected relay transportation
  networks under uncertainty.
\newblock \emph{Transportation Research Part E: Logistics and Transportation
  Review}, 194:\penalty0 103940, 2025.
\newblock \doi{10.1016/j.tre.2024.103940}.

\bibitem[Los et~al.(2020)Los, Schulte, Spaan, and Negenborn]{los2020value}
J.~Los, F.~Schulte, M.~T.~J. Spaan, and R.~R. Negenborn.
\newblock The value of information sharing for platform-based collaborative
  vehicle routing.
\newblock \emph{Transportation Research Part E: Logistics and Transportation
  Review}, 141:\penalty0 102011, 2020.
\newblock \doi{10.1016/j.tre.2020.102011}.

\bibitem[Lu et~al.(2025)Lu, Ye, Chen, and Van~Hentenryck]{lu2025boosting}
J.~Lu, T.~Ye, W.~Chen, and P.~Van~Hentenryck.
\newblock Boosting column generation with graph neural networks for joint rider
  trip planning and crew shift scheduling.
\newblock \emph{Transportation Research Part E: Logistics and Transportation
  Review}, 202:\penalty0 104281, 2025.
\newblock \doi{10.1016/j.tre.2025.104281}.

\bibitem[L{\"u}bbecke and Desrosiers(2005)]{lubbecke2005column}
M.~E. L{\"u}bbecke and J.~Desrosiers.
\newblock Selected topics in column generation.
\newblock \emph{Operations Research}, 53\penalty0 (6):\penalty0 1007--1023,
  2005.
\newblock \doi{10.1287/opre.1050.0234}.

\bibitem[Maggiar et~al.(2026)Maggiar, Dicker, and
  Mahoney]{maggiar2026consensus}
A.~Maggiar, L.~Dicker, and M.~W. Mahoney.
\newblock Consensus planning with primal, dual, and proximal agents.
\newblock \emph{INFORMS Journal on Optimization}, 8\penalty0 (2):\penalty0
  95--119, 2026.
\newblock \doi{10.1287/ijoo.2024.0054}.

\bibitem[Montreuil(2011)]{montreuil2011physical}
B.~Montreuil.
\newblock Toward a {Physical Internet}: Meeting the global logistics
  sustainability grand challenge.
\newblock \emph{Logistics Research}, 3\penalty0 (2--3):\penalty0 71--87, 2011.
\newblock \doi{10.1007/s12159-011-0045-x}.

\bibitem[Morabit et~al.(2021)Morabit, Desaulniers, and
  Lodi]{morabit2021machine}
M.~Morabit, G.~Desaulniers, and A.~Lodi.
\newblock Machine-learning-based column selection for column generation.
\newblock \emph{Transportation Science}, 55\penalty0 (4):\penalty0 815--831,
  2021.
\newblock \doi{10.1287/trsc.2021.1045}.

\bibitem[Namdar et~al.(2018)Namdar, Li, Sawhney, and Pradhan]{namdar2018supply}
J.~Namdar, X.~Li, R.~Sawhney, and N.~Pradhan.
\newblock Supply chain resilience for single and multiple sourcing in the
  presence of disruption risks.
\newblock \emph{International Journal of Production Research}, 56\penalty0
  (6):\penalty0 2339--2360, 2018.
\newblock \doi{10.1080/00207543.2017.1370149}.

\bibitem[{National Weather Service Baltimore/Washington Weather Forecast
  Office}(2018)]{nws2018windstorm}
{National Weather Service Baltimore/Washington Weather Forecast Office}.
\newblock Wind storm, {March} 2nd, 2018.
\newblock Event summary (story map),
  \url{https://www.weather.gov/lwx/WindStormMarch2018}, 2018.
\newblock Accessed 16 October 2026.

\bibitem[Nedi{\'c} and Ozdaglar(2009)]{nedic2009distributed}
A.~Nedi{\'c} and A.~Ozdaglar.
\newblock Distributed subgradient methods for multi-agent optimization.
\newblock \emph{IEEE Transactions on Automatic Control}, 54\penalty0
  (1):\penalty0 48--61, 2009.
\newblock \doi{10.1109/TAC.2008.2009515}.

\bibitem[{Oak Ridge National Laboratory}(2025)]{ornl2025eaglei}
{Oak Ridge National Laboratory}.
\newblock The {Environment for Analysis of Geo-Located Energy Information}'s
  recorded electricity outages 2014--2025.
\newblock figshare, article 24237376,
  \url{https://figshare.com/articles/dataset/24237376}, 2025.
\newblock Accessed 15--16 October 2026.

\bibitem[{Occupational Safety and Health Administration}(2026)]{osha2026ita}
{Occupational Safety and Health Administration}.
\newblock {Injury Tracking Application}: {Establishment-Specific Injury and
  Illness Data}, {CY}2023--{CY}2025.
\newblock
  \url{https://www.osha.gov/Establishment-Specific-Injury-and-Illness-Data},
  2026.
\newblock Accessed 14 October 2026.

\bibitem[{Patch Atlanta}(2021)]{patch2021fuel}
{Patch Atlanta}.
\newblock {GA} gas shortage could last next two weeks; over by {Memorial Day}.
\newblock
  \url{https://patch.com/georgia/atlanta/colonial-pipeline-making-substantial-progress-fuel-resupply},
  May 2021.
\newblock Accessed 15 October 2026.

\bibitem[P{\'e}rez-Rodr{\'i}guez and
  Holgu{\'i}n-Veras(2016)]{perezrodriguez2016deprivation}
N.~P{\'e}rez-Rodr{\'i}guez and J.~Holgu{\'i}n-Veras.
\newblock Inventory-allocation distribution models for postdisaster
  humanitarian logistics with explicit consideration of deprivation costs.
\newblock \emph{Transportation Science}, 50\penalty0 (4):\penalty0 1261--1285,
  2016.
\newblock \doi{10.1287/trsc.2014.0565}.

\bibitem[Platt(1999)]{platt1999probabilistic}
J.~C. Platt.
\newblock Probabilistic outputs for support vector machines and comparisons to
  regularized likelihood methods.
\newblock In A.~J. Smola, P.~Bartlett, B.~Sch{\"o}lkopf, and D.~Schuurmans,
  editors, \emph{Advances in Large Margin Classifiers}, pages 61--74. MIT
  Press, Cambridge, MA, 1999.

\bibitem[Prokhorenkova et~al.(2018)Prokhorenkova, Gusev, Vorobev, Dorogush, and
  Gulin]{prokhorenkova2018catboost}
L.~Prokhorenkova, G.~Gusev, A.~Vorobev, A.~V. Dorogush, and A.~Gulin.
\newblock {CatBoost}: Unbiased boosting with categorical features.
\newblock In \emph{Advances in Neural Information Processing Systems 31}, 2018.

\bibitem[Rockafellar(1970)]{rockafellar1970convex}
R.~T. Rockafellar.
\newblock \emph{Convex Analysis}.
\newblock Princeton University Press, Princeton, NJ, 1970.

\bibitem[Sawik(2013)]{sawik2013integrated}
T.~Sawik.
\newblock Integrated selection of suppliers and scheduling of customer orders
  in the presence of supply chain disruption risks.
\newblock \emph{International Journal of Production Research}, 51\penalty0
  (23--24):\penalty0 7006--7022, 2013.
\newblock \doi{10.1080/00207543.2013.852702}.

\bibitem[Scott et~al.(2017)Scott, Parker, and Craighead]{scott2017service}
A.~Scott, C.~Parker, and C.~W. Craighead.
\newblock Service refusals in supply chains: Drivers and deterrents of freight
  rejection.
\newblock \emph{Transportation Science}, 51\penalty0 (4):\penalty0 1086--1101,
  2017.
\newblock \doi{10.1287/trsc.2016.0677}.

\bibitem[Sheng et~al.(2026)Sheng, Weng, Yan, and Yu]{sheng2026predict}
T.~Sheng, J.~Weng, R.~Yan, and Q.~Yu.
\newblock A predict-then-optimize framework for just-in-time arrival of vessels
  with probabilistic berth-availability forecasts.
\newblock \emph{Transportation Research Part E: Logistics and Transportation
  Review}, 214:\penalty0 105033, 2026.
\newblock \doi{10.1016/j.tre.2026.105033}.

\bibitem[Sion(1958)]{sion1958minimax}
M.~Sion.
\newblock On general minimax theorems.
\newblock \emph{Pacific Journal of Mathematics}, 8\penalty0 (1):\penalty0
  171--176, 1958.
\newblock \doi{10.2140/pjm.1958.8.171}.

\bibitem[Snyder et~al.(2016)Snyder, Atan, Peng, Rong, Schmitt, and
  Sinsoysal]{snyder2016ormsmodels}
L.~V. Snyder, Z.~Atan, P.~Peng, Y.~Rong, A.~J. Schmitt, and B.~Sinsoysal.
\newblock {OR/MS} models for supply chain disruptions: A review.
\newblock \emph{IIE Transactions}, 48\penalty0 (2):\penalty0 89--109, 2016.
\newblock \doi{10.1080/0740817X.2015.1067735}.

\bibitem[Sohrabi and Montreuil(2011)]{sohrabi2011open}
H.~Sohrabi and B.~Montreuil.
\newblock From private supply networks and shared supply webs to {Physical
  Internet} enabled open supply webs.
\newblock In \emph{Adaptation and Value Creating Collaborative Networks}, pages
  235--244. Springer, 2011.
\newblock \doi{10.1007/978-3-642-23330-2_26}.

\bibitem[Stadtler(2009)]{stadtler2009framework}
H.~Stadtler.
\newblock A framework for collaborative planning and state-of-the-art.
\newblock \emph{OR Spectrum}, 31\penalty0 (1):\penalty0 5--30, 2009.
\newblock \doi{10.1007/s00291-007-0104-5}.

\bibitem[Tahir et~al.(2021)Tahir, Quesnel, Desaulniers, El~Hallaoui, and
  Yaakoubi]{tahir2021improved}
A.~Tahir, F.~Quesnel, G.~Desaulniers, I.~El~Hallaoui, and Y.~Yaakoubi.
\newblock An improved integral column generation algorithm using machine
  learning for aircrew pairing.
\newblock \emph{Transportation Science}, 55\penalty0 (6):\penalty0 1411--1429,
  2021.
\newblock \doi{10.1287/trsc.2021.1084}.

\bibitem[{The Atlanta Journal-Constitution}(2021)]{ajc2021colonial}
{The Atlanta Journal-Constitution}.
\newblock {Colonial Pipeline} back to ``normal operations'' following
  ransomware attack.
\newblock
  \url{https://www.ajc.com/news/colonial-pipeline-back-to-normal-operations-following-ransomware-attack/KE5GA2VGCZE7PKHFVPBTRJEG5Y/},
  May 2021.
\newblock Accessed 15 October 2026.

\bibitem[{The Fulfillment Advisor}(2025)]{tfa2025warehousing}
{The Fulfillment Advisor}.
\newblock 2025 vs. 2024 warehousing costs and pricing.
\newblock
  \url{https://www.thefulfillmentadvisor.com/2025-warehousing-market-report/},
  2025.
\newblock Accessed 19 October 2026.

\bibitem[Tomlin(2006)]{tomlin2006value}
B.~Tomlin.
\newblock On the value of mitigation and contingency strategies for managing
  supply chain disruption risks.
\newblock \emph{Management Science}, 52\penalty0 (5):\penalty0 639--657, 2006.
\newblock \doi{10.1287/mnsc.1060.0515}.

\bibitem[{U.S. Census Bureau}(2025)]{census2025cbp}
{U.S. Census Bureau}.
\newblock {County Business Patterns}: 2023.
\newblock \url{https://www.census.gov/programs-surveys/cbp.html}, 2025.
\newblock Accessed 14 October 2026.

\bibitem[{U.S. Energy Information Administration}(2022)]{eia2022cbecs}
{U.S. Energy Information Administration}.
\newblock 2018 {Commercial Buildings Energy Consumption Survey} ({CBECS}),
  {Table B1}.
\newblock \url{https://www.eia.gov/consumption/commercial/}, 2022.
\newblock Accessed 14 October 2026.

\bibitem[{U.S. Small Business Administration}(2026)]{sba2026sbs}
{U.S. Small Business Administration}.
\newblock Small business search ({SBS}).
\newblock \url{https://search.certifications.sba.gov/}, 2026.
\newblock Retrieved 14 October 2026.

\bibitem[Wang et~al.(2026)Wang, Peng, and Ma]{wang2026spovcs}
X.~Wang, Y.~Peng, and W.~Ma.
\newblock {SPO-VCS}: An end-to-end smart predict-then-optimize framework with
  alternating differentiation method for relocation problems in large-scale
  vehicle crowd sensing.
\newblock \emph{Transportation Research Part E: Logistics and Transportation
  Review}, 205:\penalty0 104515, 2026.
\newblock \doi{10.1016/j.tre.2025.104515}.

\bibitem[Weitzman(1979)]{weitzman1979optimal}
M.~L. Weitzman.
\newblock Optimal search for the best alternative.
\newblock \emph{Econometrica}, 47\penalty0 (3):\penalty0 641--654, 1979.
\newblock \doi{10.2307/1910412}.

\bibitem[Xu et~al.(2026)Xu, Klibi, and Montreuil]{xu2026dynamic}
Y.~Xu, W.~Klibi, and B.~Montreuil.
\newblock Dynamic deployment of pooled human-robot resources in urban parcel
  logistics.
\newblock \emph{Transportation Research Part B: Methodological}, 205:\penalty0
  103409, 2026.
\newblock \doi{10.1016/j.trb.2026.103409}.

\bibitem[Xue et~al.(2026)Xue, Cui, and Cheng]{xue2026rapid}
N.~Xue, T.~Cui, and S.~Cheng.
\newblock Rapid re-optimization via learning-enhanced column generation for
  vehicle routing with driver break scheduling.
\newblock \emph{Transportation Research Part E: Logistics and Transportation
  Review}, 209:\penalty0 104764, 2026.
\newblock \doi{10.1016/j.tre.2026.104764}.

\bibitem[Yan et~al.(2020)Yan, Wang, and Fagerholt]{yan2020semi}
R.~Yan, S.~Wang, and K.~Fagerholt.
\newblock A semi-``smart predict then optimize'' (semi-{SPO}) method for
  efficient ship inspection.
\newblock \emph{Transportation Research Part B: Methodological}, 142:\penalty0
  100--125, 2020.
\newblock \doi{10.1016/j.trb.2020.09.014}.

\bibitem[Yan et~al.(2023)Yan, Wang, and Zhen]{yan2023extended}
R.~Yan, S.~Wang, and L.~Zhen.
\newblock An extended smart ``predict, and optimize'' ({SPO}) framework based
  on similar sets for ship inspection planning.
\newblock \emph{Transportation Research Part E: Logistics and Transportation
  Review}, 173:\penalty0 103109, 2023.
\newblock \doi{10.1016/j.tre.2023.103109}.

\bibitem[Yang et~al.(2017)Yang, Pan, and Ballot]{yang2017mitigating}
Y.~Yang, S.~Pan, and E.~Ballot.
\newblock Mitigating supply chain disruptions through interconnected logistics
  services in the {Physical Internet}.
\newblock \emph{International Journal of Production Research}, 55\penalty0
  (14):\penalty0 3970--3983, 2017.
\newblock \doi{10.1080/00207543.2016.1223379}.

\end{thebibliography}

\appendix
\counterwithin{equation}{section}
\counterwithin{table}{section}
\counterwithin{figure}{section}

\section{Illustrative Private Firm Models}
\label{app:private-models}

This section gives two concrete instances of the private value
function \eqref{eq:valuefn}: a warehouse or distribution center, and a
cross-dock or lane carrier. They are the two member types of the
testbeds in \cref{subsec:testbeds}. The examples show how heterogeneous
internal operations map to the same public net-injection interface
without attempting to be full application models. Throughout, $t$
indexes planning periods in a firm's private model (not ADMM
iterations), and $\kk$ indexes capacity commodities as in
\cref{subsec:commodities}.

\subsection{Common private-model interface}
\label{app:interface}

For each firm, let $\mathcal J_m$ index physical resource pools,
$\own_m^{\mathrm{phys}}$ their endowments, $u_m(y_m)$ the resources
consumed by private operations, and
$A_m\in\mathbb R_+^{|\mathcal J_m|\times|\Kset|}$ the mapping from public
commodity injections to those resources. Both examples below specialize
the common constraint
\begin{equation}
\label{eq:app-interface}
u_m(y_m)+A_m\inj_m\le\own_m^{\mathrm{phys}},
\end{equation}
which allows several service-location, period, or quality commodities to
draw on one physical pool. Partial minimization over each firm's private
variables and constraints gives the value function in \eqref{eq:valuefn},
and only $\inj_m$ is exposed to coordination. The public commodity
profile remains the one defined in \eqref{eq:commodity}. A single signed
injection per commodity suffices here. If bilateral terms, asymmetric
transaction costs, or simultaneous gross lending and borrowing matter,
separate directional variables are required.

\subsection{Warehouse or distribution-center model}
\label{app:warehouse}

Let $\Pset_m$ be the products handled by firm $m$,
$\Tset=\{1,\dots,T\}$ the planning periods, and let
$\Kset_m^{\mathrm{stor}}(t)$ and $\Kset_m^{\mathrm{hand}}(t)$ collect
the storage and handling commodities of period $t$ in which the firm
participates: one commodity per cell within its service radius and per
quality class it can serve, so that several commodities draw on one
physical resource in each period. Private variables: inventory
$I_{mpt}\ge 0$, receiving quantity $a_{mpt}\ge 0$, fulfilled quantity
$b_{mpt}\ge 0$, and backlog $\beta_{mpt}\ge 0$ for each product
$p\in\Pset_m$ and period $t\in\Tset$. Inventory flow, demand, and
end-of-horizon stock follow
\begin{equation}
\label{eq:app-wh-flow}
I_{mp,t}=I_{mp,t-1}+a_{mpt}-b_{mpt},
\qquad
b_{mpt}+\beta_{mpt}=d_{mpt}+\beta_{mp,t-1},
\qquad
I_{mp,T}\ge I_{mp,0},
\end{equation}
with initial stock $I_{mp,0}>0$ and no initial backlog,
$\beta_{mp,0}=0$. Unmet demand is carried as backlog at a linear penalty
rather than dropped, and cycle stock is restored by the end of the
horizon. The last condition is the rolling-horizon convention. Without
it, a demand surge would be absorbed by silently draining inventory and
would never stress receiving or storage capacity. Capacity consumption
is linear in the operational flows:
\begin{equation}
\label{eq:app-wh-usage}
u^{\mathrm{stor}}_{mt}=\sum_{p\in\Pset_m} v_{mp}\,I_{mpt},
\qquad
u^{\mathrm{hand}}_{mt}=\sum_{p\in\Pset_m}
\big(h^{\mathrm{in}}_{mp}a_{mpt}+h^{\mathrm{out}}_{mp}b_{mpt}\big),
\end{equation}
where $v_{mp}$ is the storage space per unit of product $p$ and
$h^{\mathrm{in}}_{mp},h^{\mathrm{out}}_{mp}$ are handling hours per unit
received and shipped. Both resources meet the common interface
\eqref{eq:app-interface}, every commodity that draws on the same
physical resource entering the same constraint:
\begin{equation}
\label{eq:app-wh-link}
u^{\mathrm{stor}}_{mt}+\sum_{\kk\in\Kset_m^{\mathrm{stor}}(t)}\inj_{m,\kk}
\le\own^{\mathrm{stor}}_{mt},
\qquad
u^{\mathrm{hand}}_{mt}+\sum_{\kk\in\Kset_m^{\mathrm{hand}}(t)}\inj_{m,\kk}
\le\own^{\mathrm{hand}}_{mt}.
\end{equation}
A representative convex objective is
\begin{equation}
\label{eq:app-wh-cost}
C_m=\sum_{t}\Big[
\sum_p \big(c^{\mathrm{hold}}_{mp}I_{mpt}
+c^{\mathrm{in}}_{mp}a_{mpt}+c^{\mathrm{out}}_{mp}b_{mpt}
+c^{\beta}_{mp}\beta_{mpt}\big)
+\gamma_m\big(u^{\mathrm{hand}}_{mt}\big)^2
\Big]
+c^{\mathrm{tr}}_{m}\sum_{\kk\in\Kset_m}\delta_{m,\kk}\,
\lvert\inj_{m,\kk}\rvert,
\end{equation}
combining holding and backlog costs per unit and period, handling costs
per unit received or shipped, a convex congestion term on handling load,
and a shuttle charge on transactions at other cells. Here
$\delta_{m,\kk}$ is the distance from the firm's home cell to the
commodity location $\ell_\kk$ and $c^{\mathrm{tr}}_m$ the cost per unit
and kilometer. A home-cell transaction is free. A remote one moves goods
in one direction or the other whether the firm lends or borrows, so the
charge applies to the absolute injection. In the regional testbeds every
warehouse participates only at its home cell, so the charge is zero
there. In the abstract notation: $y_m=(I,a,b,\beta)$, the mapping $u_m(y_m)$ is
\eqref{eq:app-wh-usage}, the value function $\Fval_m(\inj_m;d_m)$ is
the partial minimum in \eqref{eq:valuefn} over these variables, and
the vector $\inj_m=(\inj_{m,\kk})_{\kk\in\Kset_m}$ is all the
coordination layer sees.

\subsection{Cross-dock or lane-carrier model}
\label{app:crossdock}

Let $\Pset_m$ be shipment classes, $\Lset_m$ the lanes that firm $m$
serves, and $\Tset$ planning periods. For transportation commodities the location
attribute is explicitly an origin--destination pair,
$\ell_\kk=(o_\kk,d_\kk)$. Private variables are the shipment flow
$x_{mplt}\ge 0$ of class $p$ on lane $l$ in period $t$, the
vehicle-equivalent capacity allocation $n_{mlt}\ge 0$, and the freight
$\xi_{mplt}\ge 0$ of class $p$ waiting at the cross-dock for lane $l$.
Fleet capacity is measured in continuous fleet hours rather than an
integer vehicle count, to preserve convexity. The freight waiting for
each lane balances inbound arrivals $g_{mplt}$ against dispatch,
\begin{equation}
\label{eq:app-cd-flow}
\xi_{mpl,t}=\xi_{mpl,t-1}+g_{mplt}-x_{mplt},
\qquad
\xi_{mpl,0}=0,
\end{equation}
so freight bound for one lane cannot leave on another. In the regional
testbeds each carrier serves a single lane.
Lane capacity limits flow by allocated fleet hours,
\begin{equation}
\label{eq:app-cd-lane}
\sum_{p\in\Pset_m} v_{mp}\,x_{mplt}\;\le\;Q_{ml}\,n_{mlt},
\end{equation}
with $v_{mp}$ the vehicle space per unit and $Q_{ml}$ the capacity per
fleet hour on lane $l$. Dock handling consumption is
\begin{equation}
\label{eq:app-cd-usage}
u^{\mathrm{dock}}_{mt}=\sum_{p,l} h_{mp}\,x_{mplt}.
\end{equation}
Writing $\Kset_m^{\mathrm{dock}}(t)$ for the dock commodities of
period $t$ and $\Kset_m^{\mathrm{lane}}(l,t)$ for the commodities that
use lane $l$, the resource links are
\begin{equation}
\label{eq:app-cd-link}
u^{\mathrm{dock}}_{mt}
+\sum_{\kk\in\Kset_m^{\mathrm{dock}}(t)}\inj_{m,\kk}
\le\own^{\mathrm{dock}}_{mt},
\qquad
n_{mlt}
+\sum_{\kk\in\Kset_m^{\mathrm{lane}}(l,t)}\inj_{m,\kk}
\le\own^{\mathrm{lane}}_{mlt}.
\end{equation}
A representative convex objective is
\begin{equation}
\label{eq:app-cd-cost}
C_m=\sum_{t}\Big[
\sum_{l} c^{\mathrm{veh}}_{ml}n_{mlt}
+\sum_{p,l} c^{\mathrm{dock}}_{mp}x_{mplt}
+\sum_{p,l} c^{\mathrm{wait}}_{mp}\xi_{mplt}
+\gamma_m\big(u^{\mathrm{dock}}_{mt}\big)^2
\Big]
+\sum_{p,l} c^{\mathrm{unsh}}_{mp}\,\xi_{mpl,T},
\end{equation}
combining fleet-hour, dock-handling, and waiting costs with a convex
congestion term and an end-of-horizon expediting charge on freight
still waiting at $T$. The charge stands in for a hard clearance
condition $\xi_{mpl,T}=0$. When a lane cut removes capacity in a period
that has no pooled lane commodity, the hard form leaves the private
problem infeasible, whereas the penalized form keeps the carrier model
feasible without any pooled capacity.
Spare dock hours and spare fleet hours on a lane are thus offered to
the pool as distinct commodities, one indexed by $(\mathrm{dock},t,\chi)$
and the other by the lane pair $(\mathrm{lane},l,t,\chi)$, while routing
and dock scheduling stay private.

\subsection{Convexity scope and operational extensions}
\label{app:scope}

Both examples are linear or convex quadratic programs. The space and
handling coefficients $v_{mp}$, $h_{mp}$, $h^{\mathrm{in}}_{mp}$,
$h^{\mathrm{out}}_{mp}$ and the lane capacities $Q_{ml}$ are positive,
and endowments, demands, arrivals, and cost coefficients are
nonnegative. Operational flows are continuous, congestion costs are
convex quadratics, distance charges are convex piecewise-linear, and
backlog, waiting, and terminal charges are linear, as are all
capacity-consumption mappings. The constraints define closed convex polyhedra.
For each feasible injection, the positive coefficients bound every flow
and stock, so the operational slice is compact and the private minimum
is attained. Together with the closed convex costs, these conditions
yield closed, proper, convex induced functions $\Fval_m(\cdot;d_m)$ as
required by \cref{ass:regularity}. They do not imply differentiability
of the system value function or uniqueness of the shadow price
(\cref{subsec:dual}).

The feasibility of \eqref{eq:pooling} required by
\cref{ass:regularity} follows from a no-export option. Each firm has a
finite-cost feasible injection $\inj_m^0\le0$, where the inequality is
componentwise: the firm may receive capacity but is not required to
lend any. In both examples such an injection exists, because a
sufficiently negative injection relaxes each link constraint in
\eqref{eq:app-wh-link} and \eqref{eq:app-cd-link} enough to serve all
demand and dispatch all arrivals. Choosing one such injection for every
firm and setting $\gap=-\sum_m\inj_m^0\ge0$ gives a finite-cost feasible
clearing point. The same argument applies after offer blocks are added,
since each block can stay at zero uptake. All costs are nonnegative, so
the optimal value is finite. After the absolute values are lifted with
auxiliary variables, the problem is a convex quadratic program with
linear constraints. Its optimum is therefore attained and has Lagrange
multipliers, which supply the saddle point of \cref{ass:regularity}.

\section{Coordination Implementation Details}
\label{app:coordination}
\paragraph{Penalty scaling.}
Under the metric $D_\rho=\rho S^{-2}$ of \eqref{eq:admm}, commodity $\kk$
receives quadratic weight $\rho_\kk=\rho/\sigma_\kk^2$ in original units.
The weights express each imbalance relative to its declared size, so that
a numerically large coordinate does not dominate the proximal term or the
residual test. They do not make commodities substitutable or alter the
economic objective and clearing constraints. The quadratic proximal term
in \eqref{eq:admm-firm} is part of the update and cannot be replaced by
a price-only response.

\paragraph{Local updates and participation.}
A firm implements \eqref{eq:admm-x} on its private operational model:
\begin{equation}
\label{eq:admm-firm}
(y_m^{t+1},\inj_m^{t+1})\;\in\;
\argmin_{(y_m,\inj_m)\in\Xset_m(\Theta_m,d_m)}
\Big\{C_m(y_m,\inj_m;\Theta_m)
+\tfrac12\big\lVert\inj_m-\inj_m^t+\bar q^{\,t}+\scaled^t
\big\rVert_{D_\rho}^2\Big\},
\end{equation}
while the recourse agent solves
\begin{equation}
\label{eq:admm-short}
\gap^{t+1}\;\in\;
\argmin_{\gap\ge0}\;\Big\{\Phi(\gap)
+\tfrac12\big\lVert\gap-\gap^t+\bar q^{\,t}+\scaled^t
\big\rVert_{D_\rho}^2\Big\},
\end{equation}
and each admitted offer block solves
\begin{equation}
\label{eq:admm-block}
x_\edge^{t+1}\;\in\;
\argmin_{x_\edge\in\mathcal X_\edge}\;\Big\{c_\edge^\top x_\edge
+\tfrac12\big\lVert A_\edge x_\edge-q_\edge^t+\bar q^{\,t}+\scaled^t
\big\rVert_{D_\rho}^2\Big\},
\qquad q_\edge^{t+1}=A_\edge x_\edge^{t+1}.
\end{equation}
Each column of $A_\edge$ is a positive multiple of one key indicator, so
this step separates by time window and reduces, in each window, to a
one-dimensional search for the multiplier of the window limit. The step
uses only the verified quote. The platform can therefore compute it, and
the provider takes no part in the iterations.
The distributed interpretation of \eqref{eq:admm} rests on the following
behavioral assumption, which is not a premise of the convex convergence
result.
\begin{assumption}[Compliant price-taking participation]
\label{ass:compliance}
Each autonomous firm implements \eqref{eq:admm-firm} on its true
private model $(\Theta_m,d_m)$, reports the resulting injection, and
treats the coordination signal as given. Firms neither misreport their
injections nor strategically distort their local models. Recourse and
offer-block agent updates are prescribed optimization steps rather than
strategic choices. The resulting multiplier signal is a coordinating
shadow value. We model no transfer scheme, individual-rationality
constraint, or budget-balance property.
\end{assumption}

\paragraph{Stopping residuals.}
The residuals are evaluated in the declared scaled coordinates,
\begin{align}
\label{eq:res-pri}
r_{\mathrm{pri}}^{\,t}
&=\sqrt{N_r}\,\big\lVert S^{-1}\bar q^{\,t}\big\rVert_2,\\
\label{eq:res-dual}
r_{\mathrm{dual}}^{\,t}
&=\rho\!\left(\sum_{i\in\Iset^r}\left\lVert
S^{-1}\!\left[(q_i^t-\bar q^{\,t})
-(q_i^{t-1}-\bar q^{\,t-1})\right]
\right\rVert_2^2\right)^{1/2},
\end{align}
so the stopping test uses the same coordinate scaling as the proximal
metric in \eqref{eq:admm}. The iteration stops at the first $t$ at
which
\begin{equation}
\label{eq:res-stop}
r_{\mathrm{pri}}^{\,t}\le\epsilon_{\mathrm{pri}}
=\varepsilon_{\mathrm p}\,\frac{n_{\mathrm{ref}}}{\max\{N_r,n_{\mathrm{ref}}\}}
\sqrt{N_r|\Kset|},
\qquad
r_{\mathrm{dual}}^{\,t}\le\epsilon_{\mathrm{dual}}
=\varepsilon_{\mathrm d}\rho\sqrt{N_r|\Kset|},
\end{equation}
for dimensionless settings
$\varepsilon_{\mathrm p},\varepsilon_{\mathrm d}>0$ and the agent count
$n_{\mathrm{ref}}=1+|\Mset^0|$ of the fixed network. The penalty $\rho$
follows a deterministic default and stays fixed within each call, as
\cref{prop:convergence} requires. The penalty default is a numerical tuning
rule, not a claim that finite-iteration paths are invariant to the
choice of units.

Because $\sum_iq_i^t=N_r\bar q^{\,t}$, the scaled clearing residual
\begin{equation}
\label{eq:clearing-residual}
R^{\,t}=\big\lVert S^{-1}\textstyle\sum_{i\in\Iset^r}q_i^t\big\rVert_2\big/\sqrt{|\Kset|}
=\sqrt{N_r}\,r_{\mathrm{pri}}^{\,t}\big/\sqrt{|\Kset|}
\end{equation}
measures how far the injections are from clearing, relative to the
declared sizes. An unscaled primal test,
$r_{\mathrm{pri}}^{\,t}\le\varepsilon_{\mathrm p}\sqrt{N_r|\Kset|}$,
would bound it only by $N_r\varepsilon_{\mathrm p}$, which loosens with
every admitted offer block. At a stop where both tests hold, the factor
in \eqref{eq:res-stop} keeps the bound at
$n_{\mathrm{ref}}\varepsilon_{\mathrm p}$, the accuracy of the fixed
network, for any number of blocks. A call that reaches the iteration cap
has no such bound. The dual test is unchanged.

\paragraph{Warm starts.}
The first coordination call starts with $q_i^{0,0}=0$ and
$\scaled^{0,0}=0$. Let $T_r$ be the terminal inner iteration in round
$r$. When warm starting is enabled, the next call starts from the
terminal state
\begin{equation}
\label{eq:warm-state}
\mathcal H_r
=\left(\scaled^{r,T_r},\;\rho^r,\;
\{q_i^{r,T_r}\}_{i\in\Iset^r}\right).
\end{equation}
A newly admitted block starts at zero. A block that a further RFQ
replaces (\cref{subsec:rfq}) keeps its terminal values on the keys it
had already quoted and starts at zero on the new ones. This
initialization can change finite-iteration paths, effort, and the
returned price, but not the round optimum.

\paragraph{Convergence.}
\begin{proposition}[Fixed-agent-set coordination]
\label{prop:convergence}
Under \cref{ass:regularity}, for fixed $\Iset^r$, $S$, and $\rho>0$,
and exact solution of each local proximal subproblem, the iteration
\eqref{eq:admm} satisfies
\[
\sum_{i\in\Iset^r}\tilde{\Fval}_i(q_i^t)\longrightarrow
\Qval(R^r),
\qquad
r_{\mathrm{pri}}^{\,t}\longrightarrow0,
\qquad
r_{\mathrm{dual}}^{\,t}\longrightarrow0.
\]
The multiplier sequence for the original clearing constraint,
$\mu^t=D_\rho\scaled^t$, converges to an optimal multiplier
$\mu^{r,\star}$ of \eqref{eq:exchange}.
\end{proposition}
\begin{proof}[Proof sketch of \cref{prop:convergence}]
Set $\tilde q_i=S^{-1}q_i$ and
$\tilde w=S^{-1}\scaled$. Because $S$ is fixed and nonsingular, the
transformed functions
$\tilde{\Fval}_i(S\tilde q_i)$ retain the properties in
\cref{ass:regularity}, and \eqref{eq:admm} becomes standard
scalar-penalty exchange ADMM
\citep[\S7.3.2]{boyd2011admm}. Its objective and residual convergence
follow from the standard two-block result under a saddle point. The
transformed multiplier is $\rho\tilde w^t$, and mapping it to the
original constraint gives
$S^{-1}(\rho\tilde w^t)=\rho S^{-2}\scaled^t
=D_\rho\scaled^t$, which proves the multiplier statement.
\end{proof}
The price sign in \cref{subsec:dual} follows from our orientation:
$q_i$ is net supply, the negative of the net-receipt coordinate in
\citet[\S7.3]{boyd2011admm}, so their receipt price is $\dual$ and
$\mu=-\dual$.

\paragraph{Economic-price Lagrangian.}
For the complete round agent set, the economic-price Lagrangian is
\begin{equation}
\label{eq:lagrangian}
L_r(\{q_i\},\dual)
=\sum_{i\in\Iset^r}\tilde{\Fval}_i(q_i)
-\dual^{\!\top}\sum_{i\in\Iset^r}q_i .
\end{equation}
It separates across recourse, incumbents, and all admitted offer blocks. At an
optimum, each agent minimizes
$\tilde{\Fval}_i(q_i)-\dual^{r,\star\top}q_i$, which for an incumbent is
$\Fval_m(\inj_m;d_m)-\dual^{r,\star\top}\inj_m$. This is a KKT
characterization of the round optimum, not the proximal update
\eqref{eq:admm-firm} or a claim about a settlement scheme.

\paragraph{Perturbation value and price sensitivity.}
To define the marginal value of a specified injection, perturb the
full round problem by $b\in\mathbb R^{|\Kset|}$:
\begin{equation}
\label{eq:perturbation}
V_r(b)
=\inf_{\{q_i\}}
\left\{\sum_{i\in\Iset^r}\tilde{\Fval}_i(q_i):
\sum_{i\in\Iset^r}q_i+b=0\right\}.
\end{equation}
Hence $V_r(0)=\Qval(R^r)$, including recourse and every offer block
already admitted in round $r$. The standard perturbation-value
interpretation of optimal multipliers gives the following result
\citep[\S5.6]{boyd2004convex}.
\begin{proposition}[Shadow-price sensitivity]
\label{cor:sensitivity}
Under \cref{ass:regularity},
\begin{equation}
\label{eq:shadow}
\partial V_r(0)
=\{-\dual:\dual\in\Lambda_r^\star\},
\qquad
V_r(b)\ge V_r(0)-\dual^\top b
\quad
(\dual\in\Lambda_r^\star).
\end{equation}
If $V_r$ is differentiable at zero, then
$\Lambda_r^\star=\{-\nabla V_r(0)\}$.
\end{proposition}
\begin{proof}[Proof sketch of \cref{cor:sensitivity}]
An optimal price gives the affine lower bound in \eqref{eq:shadow}
through the saddle-point inequality for \eqref{eq:lagrangian}.
Conversely, for any $g\in\partial V_r(0)$, applying the subgradient
inequality at $b=-\sum_iq_i$ shows that the round optimum minimizes the
Lagrangian with economic price $-g$. Thus every subgradient corresponds
to an optimal price and the two sets are equal.
\end{proof}

\paragraph{Recourse optimality conditions.}
The recourse agent links prices to capacity gaps through its KKT
conditions \citep[\S5.5]{boyd2004convex}. For the two-tier cost
\eqref{eq:twotier}, at any round optimum $(\gap^{r,\star},\dual^{r,\star})$,
these conditions give
\begin{equation}
\label{eq:twotier-kkt}
\begin{cases}
\dual^{r,\star}_\kk\le P^1_\kk,
& \gap^{r,\star}_\kk=0,\\
\dual^{r,\star}_\kk=P^1_\kk,
& 0<\gap^{r,\star}_\kk<\bar\gap_\kk,\\
\dual^{r,\star}_\kk\in[P^1_\kk,P^2_\kk],
& \gap^{r,\star}_\kk=\bar\gap_\kk,\\
\dual^{r,\star}_\kk=P^2_\kk,
& \gap^{r,\star}_\kk>\bar\gap_\kk.
\end{cases}
\end{equation}
Linear recourse is the special case without a kink.

\paragraph{Reported terminal statistics.}
At the solver boundary, a finite exchange-ADMM call produces the terminal
record
\begin{equation}
\label{eq:outputs-hat}
\begin{aligned}
\widehat{\mathcal T}_r
=\big(&\hat\gap^r,\;\hat\dual^r,\;
 r_{\mathrm{pri}}^{\,r,T_r},\;r_{\mathrm{dual}}^{\,r,T_r},\;
 T_r,\;\zeta_r\big),\\
&\hat\dual^r=-D_\rho\scaled^{r,T_r},
\qquad
\zeta_r\in\{\texttt{tolerance\_met},\texttt{iteration\_cap}\}.
\end{aligned}
\end{equation}
Let $\eta^r=\sum_{i\in\Iset^r}q_i^{r,T_r}$ be the terminal clearing
imbalance. The reported gaps add to the recourse row any deficit that
the terminal injections leave uncleared and ignore any surplus,
\begin{equation}
\label{eq:reported-gap}
\hat\gap^r_\kk=\big[\,q^{r,T_r}_{0\kk}+\max\{-\eta^r_\kk,0\}\,\big]_+,
\end{equation}
whereas the stopping residuals in \eqref{eq:res-stop} use the raw
terminal rows.

\section{Discovery and Admission Details}
\label{app:discovery}

This section collects the propositions and proofs of the discovery and
admission layer in \cref{sec:method}, whose complete loop is
\cref{alg:loop}.

\paragraph{Exact effect of one admission.}
The following result formalizes the admission effect described in
\cref{subsec:pricing}. It uses the directional value of an injection.
For any injection direction $a$ with a finite right directional
derivative, \eqref{eq:shadow} gives
\begin{equation}
\label{eq:directional-value}
-\inf_{\dual\in\Lambda_r^\star}\dual^\top a
=\sup_{g\in\partial V_r(0)}g^\top a
\;\le\;V_r'(0;a).
\end{equation}
We say that the support identity holds in direction $a$ when
\eqref{eq:directional-value} holds with equality. It holds in every
direction when $V_r$ is polyhedral or when zero lies in the interior of
its domain \citep[\S23]{rockafellar1970convex}. Under the identity, the
marginal cost reduction is $\inf_{\dual\in\Lambda_r^\star}\dual^\top a$.
Without it, an individual price still supplies the supporting bound
\eqref{eq:shadow} but not necessarily the directional derivative.

\begin{proposition}[Exact operating-value effect of an optional block]
\label{prop:monotone}
Let $\edge\notin R^r$ satisfy \cref{ass:verify,ass:fixed-offer}, and
suppose no fixed admission charge enters $\Qval$. Then
\begin{equation}
\label{eq:optional-bid-value}
\Qval(R^r\cup\{\edge\})
=\inf_{x_\edge\in\mathcal X_\edge}
\{c_\edge^\top x_\edge+V_r(A_\edge x_\edge)\}
\le V_r(0)=\Qval(R^r).
\end{equation}
Let
$\Delta_\edge=\{h\in\mathbb R_+^{|\mathcal J_\edge|}:
\mathbf 1^\top h=1\}$. If every window and key limit of $\edge$ in
\eqref{eq:uptake-set} is positive and, for some
$h\in\Delta_\edge$, the finite right directional derivative exists and
\begin{equation}
\label{eq:strict-directional}
c_\edge^\top h+V_r'(0;A_\edge h)<0,
\end{equation}
then $\Qval(R^r\cup\{\edge\})<\Qval(R^r)$. Conversely, a strict
improvement implies \eqref{eq:strict-directional} for some
$h\in\Delta_\edge$, so \eqref{eq:strict-directional} characterizes strict
improvement. If the support identity in
\eqref{eq:directional-value} holds in every direction
$A_\edge h$, this strict-improvement condition is equivalent to
\begin{equation}
\label{eq:directional-rc}
\rc^{r,\mathrm{dir}}_\edge
:=\min_{h\in\Delta_\edge}
\left\{c_\edge^\top h-
\inf_{\dual\in\Lambda_r^\star}\dual^\top A_\edge h\right\}
=\sup_{\dual\in\Lambda_r^\star}\rc^r_\edge(\dual)<0.
\end{equation}
The block reduced cost must therefore be negative uniformly over
$\Lambda_r^\star$, with a negative supremum that need not be attained. When $\Lambda_r^\star$ is a singleton, this
reduces to negative block reduced cost \eqref{eq:erc} at its unique exact
price. Negativity at one arbitrary member of a non-singleton price set is
not sufficient.
\end{proposition}

\begin{proof}[Proof sketch of \cref{prop:monotone}]
Equation \eqref{eq:optional-bid-value} follows from the offer-block agent,
and $x_\edge=0$ proves non-worsening. For any $h\in\Delta_\edge$, the
uptake $x_\edge=\tau h$ lies in $\mathcal X_\edge$ for every $\tau$
between zero and the smallest limit of $\edge$, because each window sum
and each component of $h$ is at most one. The right derivative at
zero of
$c_\edge^\top(\tau h)+V_r(\tau A_\edge h)$ is the left side of
\eqref{eq:strict-directional}. If it is finite and negative, a
sufficiently small positive $\tau$ gives strict improvement. Conversely, the function
$\varphi(x)=c_\edge^\top x+V_r(A_\edge x)$ is convex. If
$\Qval(R^r\cup\{\edge\})<\varphi(0)$, some $\bar x\in\mathcal X_\edge$ has
$\varphi(\bar x)<\varphi(0)$, and convexity gives
$\varphi'(0;\bar x)\le\varphi(\bar x)-\varphi(0)<0$. This derivative is
finite because $\partial V_r(0)$ is nonempty (\cref{cor:sensitivity}),
and $h=\bar x/\mathbf 1^\top\bar x$ satisfies
\eqref{eq:strict-directional}. Under the support identity,
\eqref{eq:directional-value} gives
$V_r'(0;A_\edge h)=-\inf_{\dual\in\Lambda_r^\star}
\dual^\top A_\edge h$, which yields the first expression in
\eqref{eq:directional-rc}. The term in braces equals
$\sup_{\dual\in\Lambda_r^\star}(c_\edge-A_\edge^\top\dual)^\top h$. The
function $(c_\edge-A_\edge^\top\dual)^\top h$ is affine in $h$ and in
$\dual$ separately, the simplex $\Delta_\edge$ is compact and convex, and
$\Lambda_r^\star$ is convex, so the minimax
theorem \citep{sion1958minimax} exchanges the minimum and the supremum.
For fixed $\dual$, the minimum over the simplex is the smallest component
of $c_\edge-A_\edge^\top\dual$, which is $\rc^r_\edge(\dual)$ in
\eqref{eq:erc}.
\end{proof}

\begin{proof}[Proof sketch of \cref{prop:termination}]
Each contact of the reference protocol removes one record from finite
$\Eset$, and each provider returns at most one persistent block.
Because no provider is asked twice, no admitted block is ever replaced,
so in \cref{alg:loop} every round that does not terminate ends with
exactly one admission. Both loops therefore terminate after finitely
many actions and at most $|\mathcal C_{\mathrm{elig}}|$ admissions. At
termination let $\dual_R^\star$ be the exact shadow price of the final
round used to reprice all excluded eligible blocks. With no admission
tolerance, their block reduced costs are nonnegative, so every component
of $c_\edge-A_\edge^\top\dual_R^\star$ is nonnegative and the pricing
value \eqref{eq:block-pricing} of each excluded block is
$v_\edge(\dual_R^\star)=0$. Hence the same price gives the exchange
problem with every block of $\mathcal C_{\mathrm{elig}}$ a dual value
equal to the final round optimum. The optimum
$\Qval(\mathcal C_{\mathrm{elig}})$ is no smaller by weak duality and no
larger because every added block is optional, so the two values
coincide.
\end{proof}

\section{Data, Linkage, and Labels}\label{app:data}

\subsection{Provider records and registry linkage}\label{app:data-linkage}

Provider records come from the Small Business Search (SBS) of the US
Small Business Administration \citep{sba2026sbs}. On 14 October 2026 we
retrieved every firm registered under five logistics industry codes:
general freight trucking, local (NAICS 484110) and long-distance
truckload (484121), freight transportation arrangement (488510), and
general (493110) and refrigerated (493120) warehousing and storage. The
pull returned 31,921 firms. Each record carries the firm's declared
industry codes, year of establishment, legal structure, certifications,
quality standards, and the date of its last update, and for some firms a
free-text self-description with keywords and special equipment. We keep
the state, city, ZIP code, county, and metropolitan area, which place a
firm in a jurisdiction and a metro. Names, telephone numbers, and street
addresses were used only for record linkage and masking. They were held
in memory and not stored.

Capability labels come from the FMCSA Company Census
\citep{fmcsa2026census}, which lists every entity with a USDOT number,
with its operating status, carrier type, for-hire authority, power units,
and the number of drivers who operate beyond 100 air miles. The
retrieved file had 4.42 million records with registration updates
through 12 October 2026. A directory firm is linked to a census record
only within the same state and only under one of three strict rules: the
same normalized name and five-digit ZIP code, the same normalized name
and city, or the same telephone number with an agreeing name. Firms with
no match, an ambiguous match, or only a weaker match stay unlinked.
Seventeen USDOT numbers are shared by 43 linked training firms. The
firms that share a number are always in the same state, so duplicated
entities never cross the state-blocked folds.

The training population consists of the strictly linked firms in 41
jurisdictions: 39 states, Puerto Rico, and Guam. The jurisdictions of
the three test regions are excluded, and the reported models are fitted
and tuned on the training firms only. Of the 14,480 directory firms in the training jurisdictions,
7,914 (54.7\%) are strictly linked. About a quarter of the 6,566
unlinked firms list trucking as their primary industry, against about
two thirds of the linked firms. A second retrieval of the directory and
the census, made to recover the USDOT numbers needed for the SAFER
lookups below, repeated the linkage for the training firms. It
reproduced the stored link of 7,896 of the 7,914 firms. For 17 firms the
census fields or the link tier had changed, most often the date of the
last registration update, and one firm had left the directory. The
labels use the fields of the first retrieval.

\subsection{Capability labels}\label{app:data-labels}

The census cargo fields are unreliable. The census records the cargo classes a
carrier declares, but in the public file the meat column duplicates the
refrigerated-food column and the building-materials column duplicates
the construction column. The census also marks more cargo classes than
the carrier's current filing shows. In a stratified sample of 144 active
carriers from eight training states, the census marked 5.4 cargo boxes
per carrier against 2.8 in the Safety and Fitness Electronic Records
(SAFER) snapshot of the same carriers. Only two carriers had a SAFER box
that the census lacked, in both cases household goods. We therefore take
the refrigerated-cargo label from SAFER \citep{fmcsa2026safer}, which we
retrieved for 4,907 of the 4,908 capable training firms, all but the
firm that had left the directory, and found for 4,901. Among these firms
the census marks refrigerated food for 881 and SAFER for 387, all of
which are among the 881, so a census-based label would more than double
the refrigerated share.

Five binary labels serve as proxies for the capability $z_{p\kk}$ of
\eqref{eq:verify-factor} (\cref{tab:labels}). A firm is capable if its
census record is active, lists it as a carrier, carries for-hire
authority, authorized or exempt, and reports at least one power unit.
The distance labels are defined for capable firms from the drivers who
operate beyond 100 air miles: long haul if any do, and interstate if any
do so across state lines. Interstate given long haul restricts the
interstate label to long-haul carriers. Interstate given capable is kept
as a reported label because it is the question posed to the LLM. The
refrigerated label is defined for capable firms with a SAFER record and
equals one when Refrigerated Food is among the declared cargo.

These labels combine into the capability probability of a lane's need
type. Every lane uses the probability of being capable. A lane beyond
100 air miles adds long haul given capable, and if it also crosses a
state line, interstate given long haul. A lane within 100 air miles uses
neither distance factor, even when it crosses a state line, because both
labels describe drivers who operate beyond that radius. Chilled lane
keys add refrigerated cargo given capable. This factor is conditioned on
carriage only, so the product treats refrigeration and distance as
independent given carriage. A refrigerated carrier also qualifies for
dry lanes, which therefore use no refrigeration factor.

\begin{table}[!htbp]
\centering
\caption{Capability labels for the 7,914 strictly linked training firms.
Positive rates are shares of the labeled firms. A card is the masked
self-description of \cref{subsec:inputs}.}
\label{tab:labels}
\small
\begin{tabular}{@{}llrrr@{}}
\toprule
Label & Defined for & Firms & With card & Positive \\
\midrule
Capable & all linked firms & 7,914 & 2,035 & 62\% \\
Long haul & capable firms & 4,908 & 1,176 & 77\% \\
Interstate & capable firms & 4,908 & 1,176 & 68\% \\
Interstate given long haul & long-haul firms & 3,788 & 876 & 88\% \\
Refrigerated cargo & capable firms in SAFER & 4,901 & 1,175 & 8\% \\
\bottomrule
\end{tabular}
\end{table}

The labels have known limits. They describe registry status, not current
availability, which no pre-contact record reveals. The refrigerated label
is a cargo declaration rather than certified equipment, whereas the LLM
is asked about refrigerated trucks. We keep this mismatch visible rather
than redefine either side. A capable value of zero means that a linked
firm is not an active for-hire carrier. Such firms are typically brokers,
warehouses, or inactive registrants. Because only strictly linked firms
carry labels, the estimates describe a population with more carriers
than the directory as a whole, while deployment scores every directory
record.

\paragraph{Storage records.}
No provider-level registry supplies storage labels. Of the 7,697
directory firms nationwide that list a warehousing code, only 58 link to
a warehouse establishment in the OSHA injury and illness reports for
2023 to 2025 \citep{osha2026ita}. A missing link does not show that a
firm has no facility, so these records give positive confirmations but
no reliable negatives. The cards say little more. Of the 336 cards that
the LLM read for warehousing-coded firms in the three pools, 25 say that
the firm operates a facility and 12 state its size. These firms are in
the pools because they link to the motor-carrier registry, so even these
few cards are not a representative sample of warehouse providers.

\subsection{An example provider record}\label{app:data-example}
\Cref{tab:example-record} follows one strictly linked training firm
through the inputs and outputs of the matching layer. We suppress its
direct identifiers, exact location, and dates, abridge its card, and
replace its stated years in business by [n]. The quoted phrases in
panel C are evidence spans that the parser found in the card. Not every span
is reproduced.

\begin{table}[!htbp]
\centering
\caption{One provider record through the matching layer. Direct
identifiers, exact location, and dates are withheld, and the card is
abridged. LightGBM outputs are out-of-fold predictions from models whose
outer fold held out the firm's state. Base rates are training shares
among the firms for which each label is defined.}
\label{tab:example-record}
\footnotesize
\begin{tabularx}{\textwidth}{@{}p{0.30\textwidth}X@{}}
\toprule
\multicolumn{2}{@{}p{\textwidth}@{}}{\textit{A. Directory record (Small Business Search)}} \\
\addlinespace[2pt]
Primary industry code & 485991 Special needs transportation \\
Other listed codes & Freight transportation arrangement and two other
  transportation-support codes \\
Firm age & More than 10 years \\
Certifications, quality standards & None \\
Website, capabilities statement & No, no \\
Card & Yes, 599 characters \\
\addlinespace
\multicolumn{2}{@{}p{\textwidth}@{}}{\textit{B. Card, masked and abridged}} \\
\addlinespace[2pt]
\multicolumn{2}{@{}p{\textwidth}@{}}{``[\ldots] available throughout the
  Continental United States [\ldots] We have Refrigerated trailers
  [\ldots] Hopper, Side Dumps, Belly Dumps, Flat Beds, Step Decks, and
  RGN trailers [\ldots]''} \\
\addlinespace
\multicolumn{2}{@{}p{\textwidth}@{}}{\textit{C. Selected entries from the LLM capability table}} \\
\addlinespace[2pt]
Main activity & Trucking carrier; evidence not reproduced \\
Service range & Long haul or nationwide; ``available throughout the
  Continental United States'' \\
Interstate & Yes; the same span \\
Temperature-controlled trucks & Yes; ``We have Refrigerated trailers'' \\
Equipment & Reefer, flatbed, step deck or lowboy, dump or bulk, other;
  representative evidence appears in panel B \\
Not stated & Operating model, hazardous materials, storage, fleet size,
  facility size \\
Parser check & 9 of 9 claimed values found in the full card \\
\addlinespace
\multicolumn{2}{@{}p{\textwidth}@{}}{\textit{D. LLM probabilities: carrier, long haul, interstate, refrigerated, storage operator, cold storage}} \\
\addlinespace[2pt]
LLM feature call, card only & 0.95, 0.93, 0.90, 0.93, 0.05, 0.02 \\
Zero-shot call, directory only & 0.50, 0.25, 0.15, 0.04, 0.07, 0.01 \\
Zero-shot call, directory and card & 0.88, 0.93, 0.90, 0.85, 0.04, 0.02 \\
\addlinespace
\multicolumn{2}{@{}p{\textwidth}@{}}{\textit{E. Registry (FMCSA Company Census and SAFER)}} \\
\addlinespace[2pt]
Status and authority & Active carrier, authorized for hire, interstate \\
Power units, drivers & 9 and 9; all drivers interstate beyond 100 air
  miles \\
Cargo classes & 12, including Refrigerated Food; livestock not listed \\
Labels & Capable, long haul, interstate, and refrigerated all 1;
  fleet size 9 \\
\addlinespace
\multicolumn{2}{@{}p{\textwidth}@{}}{\textit{F. Out-of-fold LightGBM outputs}} \\
\addlinespace[2pt]
Probability order & Directory; directory + LLM; label base rate \\
Capable & 0.42; 0.51; 0.62 \\
Long haul & 0.60; 0.83; 0.77 \\
Interstate given long haul & 0.72; 0.74; 0.88 \\
Refrigerated & 0.08; 0.44; 0.08 \\
Fleet-size proxy (power units) & Directory 3.3; directory + LLM 5.4;
  registry 9; capable-firm median 2 \\
\bottomrule
\end{tabularx}
\end{table}
\FloatBarrier

The directory record gives little indication of a freight carrier. Its
primary code is special-needs transportation. Among the five industry
codes used to retrieve providers, it lists only freight transportation
arrangement and no trucking code. The card nevertheless states
continental service and names refrigerated, flatbed, and dump-trailer
equipment, and the parser finds a quoted span in the full card for all
nine claimed values. Adding the LLM reading raises the out-of-fold refrigerated score
from 0.08 to 0.44, placing the firm among the highest-scored 1\% of
labeled training firms. The carrier score rises from 0.42 to 0.51,
against a base rate of 0.62, consistent with the directory's
nontrucking codes. The fleet-size proxy rises from 3.3 to 5.4 power
units but remains below the registry value of 9. The two declarations
do not fully coincide: the card lists livestock transport, while
livestock is absent from the firm's reported SAFER cargo classes.
Neither source verifies current operations (\cref{app:data-labels}).

\section{Regional Testbed Protocol and Supplementary Results}
\label{app:testbed-protocol}

\subsection{Networks and disruptions}
\label{app:testbed-networks}

\begin{table}[!htbp]
\centering
\caption{Regional testbeds and disruption scenarios. Pool firms are the
strictly linked directory firms of \cref{app:data} located in each metro,
and a card is the masked self-description of \cref{subsec:inputs}. Severity
is the base mean share of electricity customers without power in the 72
hours after the event \citep{brelsford2024eaglei,ornl2025eaglei}. A dash
marks a share below 1\%. The cross marks the metros of the stand-alone
demand surge. The Southeast also faces a fuel shortage on the lanes leaving
Atlanta.}
\label{tab:testbeds}
\small
\begin{tabular*}{\textwidth}{@{\extracolsep{\fill}}lrrrc@{}}
\toprule
 & & \multicolumn{2}{c}{Severity} & \\
\cmidrule(lr){3-4}
 & Pool firms & Concentrated storm & Widespread storm & Surge \\
\midrule
\multicolumn{5}{@{}l}{\textit{Southeast: Milton 2024 and Irma 2017 (1,723 firms, 390 with a card)}} \\
\addlinespace[2pt]
Atlanta & 1,046 & -- & 0.08 & \\
Tampa & 147 & 0.53 & 0.32 & $\times$ \\
Orlando & 190 & 0.14 & 0.31 & $\times$ \\
Miami & 340 & -- & 0.38 & \\
\addlinespace
\multicolumn{5}{@{}l}{\textit{Texas: Beryl 2024 and Uri 2021 (1,501 firms, 355 with a card)}} \\
\addlinespace[2pt]
Houston & 651 & 0.60 & 0.25 & $\times$ \\
Dallas--Fort Worth & 633 & -- & 0.20 & \\
San Antonio & 137 & -- & 0.19 & \\
Austin & 80 & -- & 0.26 & \\
\addlinespace
\multicolumn{5}{@{}l}{\textit{Mid-Atlantic: Isaias 2020 and Riley 2018 (1,285 firms, 392 with a card)}} \\
\addlinespace[2pt]
New York & 338 & 0.14 & 0.03 & $\times$ \\
Philadelphia & 256 & 0.06 & 0.07 & \\
Baltimore & 219 & -- & 0.04 & \\
Washington & 472 & -- & 0.07 & \\
\bottomrule
\end{tabular*}
\end{table}

Each testbed joins four metros by all twelve directed lanes over three
daily periods. Every metro has an ambient and a chilled warehouse, which
supply storage and handling, and every lane has a dry and a chilled
carrier, which supply lane and dock capacity. This gives 32 incumbents and
120 capacity keys, and no incumbent supplies both warehouse and transport
capacity. The regions differ in which capability labels matter. The six
Southeast lanes to and from Atlanta cross a state line, no Texas lane does,
and every Mid-Atlantic lane does. One metro pair in the Southeast and one in
Texas lie within 100 air miles, as do three pairs in the Mid-Atlantic. A
lane within that radius requires only basic carrier capability, and a
longer lane adds the long-haul label and, across a state line, the
interstate label (\cref{app:data-labels}). The interstate label therefore
never matters in Texas. On the six Mid-Atlantic lanes within 100 air
miles, the testbed does not check interstate operating authority, which
a real carrier would need.

Warehouse capacity converts County Business Patterns employment in general
and refrigerated warehousing \citep{census2025cbp} into pallet positions at
2,222 square feet per worker \citep{eia2022cbecs} and 0.1 positions per
square foot. Suppressed refrigerated employment in Baltimore is replaced by
injury-tracking employment \citep{osha2026ita}. Lane capacity converts
Freight Analysis Framework truck tonnage \citep{bts2025faf} into daily
loaded truck-hours with a 16-ton payload, a road distance of 1.2 times the
air distance at 80 km/h, and one handling hour per trip. The chilled share
of a lane is the food and pharmaceutical share of freight at its origin.

Storm severity is computed from county EAGLE-I records with 2024 customer
counts. The 72-hour window starts at the recorded event time or, for the
Mid-Atlantic events, which made no regional landfall, when the regional
outage share first exceeds 1\%. Each storm and seed draws one multiplier
$e\sim U(0.8,1.2)$ for all four metros, and realized severity is capped at
0.95. A metro below 1\% is not hit. At severity $s_m$, handling and chilled
storage keep $1-s_m$ of their capacity, a lane carrier's docks keep
$1-s_o$, and its lane capacity keeps $1-\max\{s_o,s_d\}$. Ambient storage is
not reduced. In the first two days, relief freight on a lane into a hit
metro $d$ rises by the factor $1+(g-1)s_d/\max_m s_m$, where
$g\sim U(2.0,2.6)$ is drawn per seed.

The stand-alone surge damages nothing. In the same two days it multiplies
demand at both warehouses of each marked metro and on every lane into it by
$1+0.5(g-1)$. The fuel shortage removes a share drawn from $U(0.60,0.73)$ of
the capacity on lanes leaving Atlanta, the range of station outages
reported in metro Atlanta during the 2021 Colonial Pipeline disruption
\citep{ajc2021colonial,patch2021fuel}. Counting a station outage as lost
truck-hours is an upper bound, since fleets with their own tanks lose less.
Uri uses the same outage rule as the hurricanes. Road icing is not modeled
and EAGLE-I misses some controlled outages, so the Uri shock is, if
anything, understated.

\subsection{Outside offers and RFQs}
\label{app:testbed-offers}

Availability is drawn per firm, seed, and scenario. The rates 0.5 and 0.3 of
\cref{app:sensitivity} reuse the same uniform numbers, so a firm available
at 0.3 is also available at 0.5 and 0.7. Let $U_p$ be the registered power
units of a capable firm, capped at 2,000, and let $a_p=0.5$ if $U_p>100$ and
$a_p=1$ otherwise. Its daily offer before shocks is
$10\times0.10\times U_pa_p\kappa_r\epsilon_p$ truck-hours, that is, ten hours
per truck and a 10\% commit share. The regional constant $\kappa_r$ keeps the
expected total offer of capable firms at its value without the large-fleet
rule, so the rule moves capacity from large to small fleets without changing
the total. The spare-capacity factor $\epsilon_p=\exp(0.5z_p-0.125)$, with
$z_p\sim N(0,1)$ drawn per firm and seed, has mean one. The threshold and
the halving are design choices based on trade reports that larger carriers
rely less on spot freight \citep{freightwaves2019spot,freightwaves2026knightswift},
neither of which is peer reviewed. A policy that ranks by size applies the
same public rule to its own size estimate but never observes $\epsilon_p$.

A carrier can serve only lanes that start or end at its home metro, and its
offer is cut only by the shock at that metro. Before contact, a size-based
policy multiplies predicted capacity by the public keep factor, computed
from the outage share without the seed multiplier. Quotes are
$P^1_\kk(1+s)\exp(0.15z_p^{c})/1.1$, where the spot premium
$s\sim U(0.23,0.35)$ is shared by all carriers in a seed
\citep{acocella2020elephants} and $z_p^{c}\sim N(0,1)$ is drawn per carrier
and seed. Both draws are the same in every scenario. Ties between equal
scores follow a pseudo-random order fixed by seed and RFQ.

One RFQ asks one pool firm about one directed lane and service class over
all three days. An available and capable firm that cannot serve the
requested class may receive a later RFQ for another class. A firm found
unavailable or not capable is not asked again. Successful replies from the
same firm are combined into one offer block with a shared daily limit, and
a new reply replaces the firm's earlier block in the repository or in the
admitted set. Before the next RFQ, the repository is screened at the
current prices and the block with the most negative pricing value is
admitted, one per coordination call, until no stored block passes the
threshold. In the end-to-end ADMM run, the admitted set after each RFQ is
re-solved centrally for scoring only and never fed back into the search.

Open storage enters only the full-information benchmark of
\cref{subsec:open-value}. In each metro and temperature class, outside
supply equals the incumbent stock times the industrial vacancy rate
\citep{cw2026industrial}, a proxy for idle operating space rather than an
observed offer. It is split into one block per County Business Patterns
establishment in proportion to the midpoint employment of its size class.
Each block offers storage and the matching handling capacity, is available
with probability 0.7, inherits the shock of the incumbent warehouse of its
metro and class, and is quoted like a carrier, with the seed's spot premium
and its own dispersion draw.

\subsection{Scaling and settings}
\label{app:testbed-settings}

\Cref{tab:numerical-settings} lists the baseline settings, which are common
to all regions. One model unit of lane or storage capacity equals 44.5, 76.8, and 34.0
truck-hours or pallet positions in the Southeast, Texas, and the
Mid-Atlantic, a figure that already includes the platform share
$\phi=0.05$. The transport anchor sets the dry cheap tier to 1.1 times the
2025 industry-average truck operating cost of US\$2.336 per mile
\citep{atri2026costs}. At the testbed speed and handling time, a loaded
truck-hour covers 41, 40, and 32 miles, so the anchor is US\$106, 104, and
84 per truck-hour. Matching it to the mean dry cheap-tier price of the
reported instances makes one transport cost unit worth US\$35.5, 58.9, and
21.1. The storage anchor sets the ambient cheap tier to 1.1 times the 2025
average third-party fee of US\$20.17 per pallet-month
\citep{tfa2025warehousing}, or US\$0.73 per pallet-day, so one warehouse cost
unit is worth US\$10.3, 18.0, and 7.9. Shares that combine the two sides
convert each side separately.

The recourse tiers are drawn per capacity key and seed and are the same
in every event of a seed. For key $\kk$, $P^1_\kk$ and $P^2_\kk$ are
independent multiples, $U(1.8,2.5)$ and $U(5,8)$, of the internal
marginal cost of the key's resource type, multiplied by 1.5 for chilled
storage and refrigerated lanes. The internal marginal cost is the mean
linear cost coefficient of that type over the region's incumbents: the
holding cost per pallet-day for storage, the inbound handling cost per
labor-hour for handling, the dock cost per dock-hour for docks, and the
vehicle cost per fleet-hour for lanes. The cheap-tier capacity
$\bar\gap_\kk$ comes from a pilot run before any outside firm is added.
For each disrupted event of the seed, including the Southeast I-4
closure of \cref{app:testbed-checks}, the incumbent network is solved
with full firm-to-firm pooling and a provisional cheap-tier capacity.
$\bar\gap_\kk$ is then $U(0.3,0.6)$ times the median of the key's
positive recourse quantities across these events. A key that never needs
recourse in the pilot keeps its provisional value, $U(0.3,0.6)$ times
30\% of the largest per-day capacity of an incumbent that holds the key.
The pilot uses only the incumbent network of the same seed, with no
outside offer or search policy.

The translation has two limits. The recourse multipliers put internal
marginal cost at about half the anchored rate on both sides, which is a
design choice. The chilled multiplier of 1.5 also exceeds the premium of
refrigerated over dry-van contract rates, about 10\% in 2026
\citep{dat2026contract}, so chilled transport is valued high. The ADMM
penalty is set from model-unit prices, so a model priced in dollars from
the start would use a different penalty, while the exact solutions would
not change. No network was reselected after the translation.

\begin{table}[!htbp]
\centering
\caption{Baseline settings. Basis: the source of a data-anchored value, a
design choice without a data anchor, or a numerical setting that controls
accuracy or ties only.}
\label{tab:numerical-settings}
\footnotesize
\setlength{\tabcolsep}{5pt}
\renewcommand{\arraystretch}{1.12}
\begin{tabularx}{\textwidth}{@{}>{\raggedright\arraybackslash}p{0.23\textwidth}>{\raggedright\arraybackslash}X>{\raggedright\arraybackslash}p{0.22\textwidth}@{}}
\toprule
Setting & Value & Basis \\
\midrule
\multicolumn{3}{@{}l}{\textit{Coordination}} \\
Characteristic size $\sigma_\kk$ & cheap-tier capacity $\bar\gap_\kk$, floor $10^{-2}$ & design \\
Residual tolerances $\varepsilon_{\mathrm p},\varepsilon_{\mathrm d}$ & $10^{-4}$ & numerical \\
Iteration cap & 6{,}000 per call & numerical \\
Penalty $\rho$ & median of $P^1_\kk/\sigma_\kk$ over keys, fixed within an episode & numerical \\
Reference agent count $n_{\mathrm{ref}}$ & 33 (32 incumbents and recourse) & numerical, \eqref{eq:res-stop} \\
Warm start & previous terminal state, new keys at zero & design \\
\addlinespace[3pt]
\multicolumn{3}{@{}l}{\textit{Search}} \\
Thresholds $\varepsilon_{\mathrm{brief}}$, $d_{\min}$ & $10^{-6}$ & numerical \\
Admission & pricing value below $-10^{-6}$ & numerical \\
RFQ & one directed lane and service class, all three days & design \\
Ties & pseudo-random order per seed and RFQ & numerical \\
RFQ budget $K$ & 80 & design \\
\addlinespace[3pt]
\multicolumn{3}{@{}l}{\textit{External carriers}} \\
Availability & 0.7 per firm, seed, and scenario & \citet{acocella2022setit} \\
Daily offer & 10 hours per power unit $\times$ 10\% commit share $\times$ fleet, regional, and spare-capacity factors & design \\
Spare-capacity factor $\epsilon_p$ & lognormal, mean 1, log SD 0.5, per firm and seed & design \\
Large fleets & half the share above 100 power units, regional total kept & industry analyses \citep{freightwaves2019spot,freightwaves2026knightswift} \\
Registry fleet size & FMCSA power units of capable firms, capped at 2{,}000 & numerical \\
Quote & $P^1_\kk(1+s)e^{0.15z^c_p}/1.1$, $s\sim U(0.23,0.35)$ per seed, $z^c_p\sim N(0,1)$ per firm and seed & premia from \citet{acocella2020elephants}, dispersion by design \\
Shock & home-metro severity only & design \\
Reach & lanes that start or end at the firm's metro & design \\
Platform share $\phi$ & 0.05 & design \\
\addlinespace[3pt]
\multicolumn{3}{@{}l}{\textit{Recourse}} \\
Cheap tier $P^1_\kk$ & $U(1.8,2.5)\times$ the internal marginal cost & design \\
Expensive tier $P^2_\kk$ & $U(5,8)\times$ the internal marginal cost & design \\
Chilled multiplier & 1.5 & design \\
Cheap-tier capacity $\bar\gap_\kk$ & $U(0.3,0.6)\times$ the median positive pilot gap, or a capacity proxy for keys that never gap & design \\
\addlinespace[3pt]
\multicolumn{3}{@{}l}{\textit{Prize and dollars}} \\
Volume part of disruption cost & re-solve with incumbent capacity $\times\,100$ ($\times\,1{,}000$ checked) & design \\
Transport anchor & dry cheap tier $=1.1\times$ US\$2.336 per mile, at the testbed's miles per truck-hour & \citet{atri2026costs} \\
Storage anchor & ambient cheap tier $=1.1\times$ US\$20.17 per pallet-month & \citet{tfa2025warehousing} \\
\addlinespace[3pt]
\multicolumn{3}{@{}l}{\textit{Reporting}} \\
Seeds & 1--20 development, 21--40 reported & design \\
Intervals & 4{,}000 seed-bootstrap resamples, one generator for all arms & design \\
Computation & one process per episode, 12 episodes in parallel on a 12-core laptop; Clarabel interior-point solver, OSQP fallback & numerical \\
\bottomrule
\end{tabularx}
\end{table}

\subsection{Pre-run checks and run history}
\label{app:testbed-checks}

Before any policy is run, each network is checked on seeds 1 to 20 without
the outside pool. The normal state should be calm, with recourse below 2\% of
cost. Each event should create a shortage that coordination reduces but does
not close: recourse of at least 15\% of the affected incumbents' cost, a
positive coordination gain with at least two remaining gaps, at least three
scarce commodities, and at least two prices strictly between the recourse
tiers. No seed is dropped when a check fails. Every concentrated storm
passes, and the normal states and widespread storms fail in at most four
seeds. The surges fail most
often, in 5, 11, and 13 of the 20 seeds, because added freight alone often
leaves recourse below the 15\% threshold. The fuel shortage fails in 7 seeds
on the price check. The Southeast generator also contains an I-4 closure as
an alternative transport-only event. Its disruption cost was small and it
passed in only 5 seeds, so the fuel shortage was chosen under a rule fixed
before any policy run.

The availability rate, the offered-capacity and quote rules, carrier reach,
and the RFQ protocol were set during development runs on seeds 1 to 20
and kept fixed for the reported runs. The 80-RFQ comparison is the
pre-declared headline, and the 160-RFQ runs of the learned and registry
policies belong to the same run set. The 480-RFQ runs use unchanged
settings, and every longer run reproduces the shorter paths
exactly over their shared budget.
The capability-only control of the recommended model in
\cref{tab:contact-inputs} uses unchanged settings, and its re-run
recommended and zero-shot policies reproduce the earlier runs exactly.
The comparison of gap guidance across provider scores in
\cref{tab:gap-by-info} uses the original runs through 80 RFQs, and no
policy was run for it.

\subsection{Supplementary policy comparisons}
\label{app:testbed-search-results}

\begin{table}[!htbp]
\centering
\caption{Recovered share of the transport open value, defined in
\eqref{eq:recovered-share}, by RFQ budget under exact coordination, in
percent, with 95\% seed-bootstrap intervals in brackets. Interaction: the effect of gap
guidance with learned scores minus its effect with random carriers, in
percentage points, a paired difference. Budgets above 80 RFQs come from
separate runs that reproduce the 80-RFQ paths exactly over their first 80
RFQs.}
\label{tab:search-budgets}
\footnotesize
\setlength{\tabcolsep}{4pt}
\begin{tabular*}{\textwidth}{@{\extracolsep{\fill}}ll r@{\hspace{3pt}}l r@{\hspace{3pt}}l r@{\hspace{3pt}}l r@{\hspace{3pt}}l@{}}
\toprule
What to ask & Whom to ask & \multicolumn{2}{c}{80 RFQs} & \multicolumn{2}{c}{160 RFQs} & \multicolumn{2}{c}{320 RFQs} & \multicolumn{2}{c}{480 RFQs} \\
\midrule
\multicolumn{10}{@{}l}{\textit{Southeast}} \\
\addlinespace[2pt]
Gap-guided & Learned score & 35.5 & {\scriptsize[33.0, 38.1]} & 51.2 & {\scriptsize[48.5, 53.9]} & 67.0 & {\scriptsize[65.1, 69.0]} & 74.1 & {\scriptsize[72.4, 75.9]} \\
Random & Learned score & 27.6 & {\scriptsize[25.1, 30.3]} & 39.8 & {\scriptsize[37.4, 42.4]} & 56.1 & {\scriptsize[53.9, 58.3]} & 66.7 & {\scriptsize[64.3, 69.0]} \\
Gap-guided & Random carrier & 5.5 & {\scriptsize[4.7, 6.6]} & 11.2 & {\scriptsize[10.0, 12.5]} & 23.4 & {\scriptsize[21.4, 25.9]} & 32.1 & {\scriptsize[30.0, 34.5]} \\
Random & Random carrier & 3.7 & {\scriptsize[2.9, 4.6]} & 7.7 & {\scriptsize[6.2, 9.4]} & 15.2 & {\scriptsize[13.3, 17.1]} & 20.5 & {\scriptsize[18.7, 22.5]} \\
\multicolumn{2}{@{}l}{Perfect provider data (reference)} & 74.6 & {\scriptsize[72.0, 77.0]} & 82.3 & {\scriptsize[80.5, 84.1]} & 88.1 & {\scriptsize[86.9, 89.1]} & 91.0 & {\scriptsize[90.2, 91.8]} \\
\multicolumn{2}{@{}l}{Interaction (percentage points)} & +6.0 & {\scriptsize[3.9, 8.4]} & +8.0 & {\scriptsize[5.3, 10.6]} & +2.6 & {\scriptsize[$-$0.5, 5.6]} & $-$4.1 & {\scriptsize[$-$6.6, $-$1.8]} \\
\addlinespace
\multicolumn{10}{@{}l}{\textit{Texas}} \\
\addlinespace[2pt]
Gap-guided & Learned score & 19.4 & {\scriptsize[17.4, 21.3]} & 25.9 & {\scriptsize[23.8, 28.0]} & 37.5 & {\scriptsize[35.4, 39.6]} & 44.8 & {\scriptsize[42.8, 46.8]} \\
Random & Learned score & 15.1 & {\scriptsize[13.7, 16.7]} & 22.0 & {\scriptsize[20.0, 24.0]} & 32.7 & {\scriptsize[30.3, 35.1]} & 40.0 & {\scriptsize[37.9, 42.3]} \\
Gap-guided & Random carrier & 2.1 & {\scriptsize[1.6, 2.7]} & 4.1 & {\scriptsize[3.4, 4.8]} & 8.0 & {\scriptsize[6.9, 9.0]} & 12.3 & {\scriptsize[11.0, 13.7]} \\
Random & Random carrier & 1.6 & {\scriptsize[1.2, 1.9]} & 3.1 & {\scriptsize[2.6, 3.8]} & 6.5 & {\scriptsize[5.6, 7.6]} & 9.8 & {\scriptsize[8.6, 11.1]} \\
\multicolumn{2}{@{}l}{Perfect provider data (reference)} & 41.1 & {\scriptsize[38.8, 43.5]} & 53.5 & {\scriptsize[51.4, 55.7]} & 65.1 & {\scriptsize[63.4, 66.9]} & 70.9 & {\scriptsize[69.3, 72.5]} \\
\multicolumn{2}{@{}l}{Interaction (percentage points)} & +3.8 & {\scriptsize[2.1, 5.6]} & +3.0 & {\scriptsize[2.0, 4.1]} & +3.4 & {\scriptsize[1.9, 4.9]} & +2.3 & {\scriptsize[0.5, 4.0]} \\
\addlinespace
\multicolumn{10}{@{}l}{\textit{Mid-Atlantic}} \\
\addlinespace[2pt]
Gap-guided & Learned score & 42.9 & {\scriptsize[40.2, 46.1]} & 60.2 & {\scriptsize[56.9, 63.2]} & 77.0 & {\scriptsize[73.4, 80.5]} & 83.0 & {\scriptsize[80.1, 85.7]} \\
Random & Learned score & 25.5 & {\scriptsize[22.9, 28.8]} & 44.3 & {\scriptsize[40.8, 48.1]} & 61.7 & {\scriptsize[58.3, 65.0]} & 70.9 & {\scriptsize[67.7, 74.4]} \\
Gap-guided & Random carrier & 11.1 & {\scriptsize[7.9, 15.8]} & 19.5 & {\scriptsize[15.8, 24.5]} & 34.7 & {\scriptsize[29.8, 40.3]} & 46.3 & {\scriptsize[42.2, 50.8]} \\
Random & Random carrier & 6.0 & {\scriptsize[4.4, 7.8]} & 10.9 & {\scriptsize[9.1, 13.0]} & 20.3 & {\scriptsize[17.3, 24.2]} & 27.8 & {\scriptsize[24.7, 31.5]} \\
\multicolumn{2}{@{}l}{Perfect provider data (reference)} & 86.6 & {\scriptsize[82.3, 90.2]} & 92.0 & {\scriptsize[89.2, 94.3]} & 95.6 & {\scriptsize[93.9, 96.9]} & 96.8 & {\scriptsize[95.7, 97.8]} \\
\multicolumn{2}{@{}l}{Interaction (percentage points)} & +12.3 & {\scriptsize[6.7, 17.2]} & +7.2 & {\scriptsize[0.7, 13.1]} & +1.1 & {\scriptsize[$-$6.1, 8.8]} & $-$6.3 & {\scriptsize[$-$12.0, $-$0.2]} \\
\bottomrule
\end{tabular*}
\end{table}

\begin{table}[!htbp]
\centering
\caption{Operating value of provider information after 80 gap-guided RFQs
under exact coordination. Recovered share of the transport open value,
defined in \eqref{eq:recovered-share}, in percent (A) and paired effects
in percentage points (B), with 95\% seed-bootstrap intervals in brackets.
\Cref{tab:contact-inputs} compares the other inputs of the contact score. Registry values
are the testbed truth for capability and fleet size. Availability,
quantity, and price are still revealed by RFQs. In dollars per event, the
effect of registry fleet size in panel B is US\$66,000, 74,000, and 64,000,
and that of registry capability US\$25,000, 37,000, and 32,000.
\dag\ Recommended policy.}
\label{tab:info-value-full}
\footnotesize
\setlength{\tabcolsep}{4pt}
\begin{tabular*}{\textwidth}{@{\extracolsep{\fill}}ll r@{\hspace{3pt}}l r@{\hspace{3pt}}l r@{\hspace{3pt}}l@{}}
\toprule
Capability from & Fleet size from & \multicolumn{2}{c}{Southeast} & \multicolumn{2}{c}{Texas} & \multicolumn{2}{c}{Mid-Atlantic} \\
\midrule
\multicolumn{8}{@{}l}{\textit{A. Recovered share (\%), predicted or registry values}} \\
\addlinespace[2pt]
Predicted\,\dag & Predicted & 35.5 & {\scriptsize[33.0, 38.1]} & 19.4 & {\scriptsize[17.4, 21.3]} & 42.9 & {\scriptsize[40.2, 46.1]} \\
Registry & Predicted & 47.2 & {\scriptsize[44.6, 49.7]} & 26.3 & {\scriptsize[24.3, 28.4]} & 61.7 & {\scriptsize[57.9, 65.7]} \\
Predicted & Registry & 66.3 & {\scriptsize[63.5, 69.2]} & 33.1 & {\scriptsize[30.6, 35.5]} & 80.6 & {\scriptsize[75.7, 85.0]} \\
Registry & Registry & 74.6 & {\scriptsize[72.0, 77.0]} & 41.1 & {\scriptsize[38.8, 43.5]} & 86.6 & {\scriptsize[82.3, 90.2]} \\
\addlinespace
\multicolumn{8}{@{}l}{\textit{B. Effect of registry values (percentage points)}} \\
\addlinespace[2pt]
\multicolumn{2}{@{}l}{Registry capability} & +11.7 & {\scriptsize[10.4, 12.9]} & +6.9 & {\scriptsize[6.1, 7.8]} & +18.8 & {\scriptsize[16.0, 21.8]} \\
\multicolumn{2}{@{}l}{Registry fleet size} & +30.8 & {\scriptsize[27.6, 34.1]} & +13.7 & {\scriptsize[12.0, 15.4]} & +37.7 & {\scriptsize[32.1, 42.5]} \\
\multicolumn{2}{@{}l}{Fleet size minus capability} & +19.1 & {\scriptsize[15.6, 22.8]} & +6.8 & {\scriptsize[5.2, 8.4]} & +18.9 & {\scriptsize[14.4, 22.9]} \\
\bottomrule
\end{tabular*}
\end{table}

\begin{table}[!htbp]
\centering
\caption{Effect of gap guidance by provider score: recovered share of the
transport open value, defined in \eqref{eq:recovered-share}, with
gap-guided RFQs minus that with random RFQs, in percentage points, under
exact coordination, with paired 95\% seed-bootstrap intervals in
brackets. Each difference holds the provider score fixed. Registry values
are the perfect provider data of \cref{subsec:decision-value}. The last
row of each panel is the paired difference between the registry and
learned-score effects. A displayed 0.0 may round a small value of either
sign.}
\label{tab:gap-by-info}
\footnotesize
\setlength{\tabcolsep}{4pt}
\begin{tabular*}{\textwidth}{@{\extracolsep{\fill}}l r@{\hspace{3pt}}l r@{\hspace{3pt}}l r@{\hspace{3pt}}l@{}}
\toprule
Provider score & \multicolumn{2}{c}{Southeast} & \multicolumn{2}{c}{Texas} & \multicolumn{2}{c}{Mid-Atlantic} \\
\midrule
\multicolumn{7}{@{}l}{\textit{20 RFQs}} \\
\addlinespace[2pt]
Random carrier & +0.3 & {\scriptsize[0.0, 0.6]} & +0.2 & {\scriptsize[$-$0.3, 0.6]} & +0.8 & {\scriptsize[$-$1.2, 2.6]} \\
Learned score & +6.4 & {\scriptsize[4.5, 8.4]} & +2.3 & {\scriptsize[1.0, 4.0]} & +6.1 & {\scriptsize[4.8, 7.5]} \\
Registry values & +18.5 & {\scriptsize[15.3, 21.7]} & +4.4 & {\scriptsize[3.1, 5.8]} & +25.6 & {\scriptsize[20.3, 31.2]} \\
Registry minus learned & +12.1 & {\scriptsize[9.1, 15.2]} & +2.1 & {\scriptsize[0.2, 3.7]} & +19.5 & {\scriptsize[14.9, 24.6]} \\
\addlinespace
\multicolumn{7}{@{}l}{\textit{40 RFQs}} \\
\addlinespace[2pt]
Random carrier & +1.3 & {\scriptsize[0.7, 2.0]} & +0.2 & {\scriptsize[$-$0.3, 0.7]} & +2.5 & {\scriptsize[0.3, 4.5]} \\
Learned score & +7.1 & {\scriptsize[5.9, 8.3]} & +2.9 & {\scriptsize[1.5, 4.4]} & +12.1 & {\scriptsize[10.3, 14.0]} \\
Registry values & +12.9 & {\scriptsize[10.0, 16.0]} & +6.4 & {\scriptsize[4.6, 8.2]} & +21.3 & {\scriptsize[17.0, 26.1]} \\
Registry minus learned & +5.8 & {\scriptsize[3.4, 8.6]} & +3.5 & {\scriptsize[1.3, 5.5]} & +9.2 & {\scriptsize[4.5, 14.0]} \\
\addlinespace
\multicolumn{7}{@{}l}{\textit{80 RFQs}} \\
\addlinespace[2pt]
Random carrier & +1.8 & {\scriptsize[0.8, 3.0]} & +0.5 & {\scriptsize[$-$0.1, 1.1]} & +5.1 & {\scriptsize[0.9, 10.6]} \\
Learned score & +7.9 & {\scriptsize[6.3, 10.0]} & +4.3 & {\scriptsize[2.7, 6.0]} & +17.4 & {\scriptsize[14.1, 20.5]} \\
Registry values & +6.5 & {\scriptsize[4.6, 8.5]} & +8.0 & {\scriptsize[6.1, 9.9]} & +11.3 & {\scriptsize[8.3, 14.5]} \\
Registry minus learned & $-$1.4 & {\scriptsize[$-$4.0, 0.9]} & +3.7 & {\scriptsize[2.1, 5.2]} & $-$6.2 & {\scriptsize[$-$10.5, $-$1.9]} \\
\bottomrule
\end{tabular*}
\end{table}

The supplementary policies keep the RFQ and admission protocol and change
only what the contact score knows. The industry-code rule scores a firm by
the capability rate and the geometric-mean fleet size of capable training
firms in its industry-code cell, with the overall mean for unseen cells.
The zero-shot LLM supplies capability probabilities and either gives every
firm the same size, so that nothing is fitted to this study's data, or uses
the trained fleet-size proxy. The directory-only policy is LightGBM on the
directory record with its own fleet proxy, and the Titan policy is
LightGBM on the directory record and the Titan embedding with its own
fleet proxy (\cref{tab:contact-inputs}). Paired differences from LightGBM on the directory at 80
RFQs, computed before rounding, are $-0.2$ [$-0.8$, 0.5], $+4.2$ [3.2, 5.4],
and $+0.4$ [$-0.7$, 1.5] percentage points for the LLM reading in the
Southeast, Texas, and the Mid-Atlantic. For the Titan embedding they are
$-1.8$ [$-2.8$, $-1.0$], $+0.2$ [$-0.8$, 1.0], and $+1.1$ [$-0.3$, 2.5]. Three references replace
predicted capability, predicted fleet size, or both by registry values. The
last is the perfect-provider-data reference, which still leaves
availability, quantity, and price to the RFQs.

Each pool firm passes the location screen for 12 of the 24 lane-class
inquiries in its region. Asking every firm about each of them in separate RFQs, and
dropping firms found unavailable or not capable, takes a median of about
6,600 to 8,600 RFQs per episode. The reference protocol of
\cref{subsec:stop} instead asks each firm once about all of its inquiries. Within its 80 RFQs, the recommended policy reaches
41 to 44 distinct firms and admits 10.5 to 16 of them on average. Between 39
and 47\% of its RFQs return an offer, and most of the rest reach an available
carrier that cannot serve the requested class. With perfect provider data,
78 to 84\% return an offer. Requests 41 to 80 of the recommended policy save an
average of US\$560 to 980 each, compared with US\$80 to 300 for requests
241 to 480. These values exclude search costs and therefore give
break-even benchmarks for an RFQ and its verification, not an optimal
budget (\cref{subsec:stop}).

\subsection{Sensitivity analyses}
\label{app:sensitivity}

The full-information variants change one assumption at a time. In the main
setting, outside transport and storage remove 19 to 31\% of the scarcity
cost. Setting the expensive tier to twice the cheap tier lowers this to 11
to 22\%, and setting it to 1.5 times lowers it to 6 to 15\%, because outside
offers mostly replace the expensive tier. A platform share of 2.5\% instead
of 5\%, which doubles outside capacity relative to the network, raises it to
30 to 34\%, and a share of 10\% lowers it to 11 to 28\%. The spare-capacity
dispersion, repositioning of carriers anywhere in their region, and the
storage dollar anchor each change it by at most 1.2 points. Two
assumptions, the spare-capacity dispersion and the availability rate,
were also rerun with the limited-RFQ policies.

The pre-specified rerun with a spare-capacity log standard deviation of
1.0 makes the offer of a given fleet more variable but keeps its mean
proportional to fleet size within each size class. At 80 RFQs, registry
fleet size still adds 8.6 to 20.0 percentage points more than registry
capability, with every paired 95\% interval excluding zero, and the
recommended policy moves by at most 3.7 points. A capped or concave
relation between fleet size and offered capacity was not run. Lowering availability to
0.5 shrinks the transport open value by 5 to 26\%, and lowering it to 0.3
shrinks it by 18 to 53\%. Texas loses most. Its full-information plan
already uses nearly every carrier that would offer
(\cref{subsec:decision-value}), so it has few substitutes for a carrier
that becomes unavailable. At 80 RFQs, the recommended policy recovers a share of the
remaining value that is up to 8.0 points smaller in the Southeast and the
Mid-Atlantic and similar or larger in Texas. The interaction between gap
guidance and learned ranking stays positive, at 2.9 to 12.6 points at
availability 0.5 and 2.9 to 9.3 points at 0.3, with every paired interval
excluding zero.

\section{Numerical Validation}\label{app:validation}

\paragraph{Full-information benchmark.}
The benchmark $\Qval^{\mathrm{full}}$ of \cref{subsec:open-value} admits
every eligible block at once, whereas the mechanism admits one block per
coordination call. The reference protocol of \cref{prop:termination}
checks that the two routes reach the same cost. In each of the 200
reported episodes, it asks every firm in the provider pool once about all
its location-compatible keys, in a fixed pseudo-random order for each
seed and event, and admits stored blocks one at a time at exact prices.
The screen uses a tolerance of $10^{-9}$ times the larger of one and the
block's scale, its capacity in each window valued at the largest tier
spread $P^2_\kk-P^1_\kk$ among its keys and summed over the windows. Three pass criteria were fixed before
the reported run: the terminal cost lies within $10^{-5}$ of
$\Qval^{\mathrm{full}}$ relative to its level, no admission raises the
exact cost by more than $10^{-5}$ of its level, and
$\Qval^{\mathrm{full}}$, $\Qval^{\mathrm{fixed}}$, and the number of
eligible blocks equal those of the decision runs. The cost and
monotonicity tolerances were first set at $10^{-6}$ and raised before any
reported episode was run, because the largest deviation in a development
run on seed 1, $9\times10^{-7}$, came within 10\% of that bound.
Deviations of this size reflect the precision of the interior-point
solver Clarabel. All three criteria hold in every episode. The terminal
cost lies above $\Qval^{\mathrm{full}}$ in 108 episodes and below it in
92, never by more than $3.3\times10^{-6}$ relative to its level. The
original tolerance would have failed 5 episodes on cost and 8 on
monotonicity. Measured against the open value, the largest difference is
0.21\%, in the Texas surge episode where that value is smallest. At the
median, the protocol excludes about half of the eligible blocks in the
Southeast and the Mid-Atlantic but only about 3\% in Texas. Every
excluded block has a positive reduced cost at the terminal price, which
provides the optimality certificate of \cref{prop:termination}.

\paragraph{Finite-iteration coordination.}
The ADMM residuals do not certify how close the reported gaps and prices
are to exact ones (\cref{subsec:outputs}), so both were compared with the
exact central solution on seeds 21 to 25. On the 65 fixed networks,
covering the normal state and every event of each region, 61 calls meet
the stopping rule, while the other four reach the iteration cap in
Texas. In every case at least 96.7\% of the keys have an ADMM price within
1\% of a tier spread, $P^2_\kk-P^1_\kk$, of the price set allowed by
\eqref{eq:twotier-kkt} at the exact recourse. The lane and service class of
the first RFQ, chosen from the ADMM gaps, match the exact choice in 58
cases, and four of the seven
differences arise in normal states, where recourse is small. The same
96.7\% lower bound holds after cold and warm starts on the networks
reached by the recommended policy after 20, 40, and 80 RFQs. At these
checkpoints, a warm start from the terminal fixed-network state meets the
stopping rule in 48 or 49 of 50 cases, against 43 to 47 for a cold start,
and reduces the paired median iteration count by 44 to 54\%. The scale
stress test admits all of the roughly 490 to 520 full-information blocks
in the concentrated storm of seed 21. Cold starts then reach the cap of
6,000 iterations in every region, whereas warm starts meet the stopping
rule within 3,259 to 4,526 iterations. By contrast, an end-to-end call
starts from the terminal state of a network that differs by one added or
extended block and needs a median of 120 to 219 iterations. Only 26 of
its 6,159 calls reach the cap.

\paragraph{Consequences for admission and search.}
At a kink of a recourse schedule, \eqref{eq:twotier-kkt} only confines
the price to $[P^1_\kk,P^2_\kk]$, so an ADMM price can pass the accuracy
check above and still screen quotes differently from the exact price.
The predeclared screening criterion required that no block judged
differently by the ADMM price and by the exact price returned by the
central solver have a pricing value at that exact price whose magnitude
exceeds 1\% of the block's scale. Both judgments use the scaled
tolerance of the benchmark check above. Of the 365
checks, 65 on the fixed networks and 300 after the cold and warm starts
at the three checkpoints, 109 misclassify at least one block, 7 exceed
0.1\%, and 3 exceed 1\%, so the criterion is not met. All three failures are in the Southeast, in the
normal state of seed 23 and after the warm start of seed 25 in the surge
scenario at 20 and 40 RFQs. In each case, the blocks above the threshold
share one key at the kink of its recourse schedule, where the ADMM price
lies 0.095 to 0.19 tier spreads below the exact price. These blocks have
a pricing value of zero at the ADMM price and below $-0.07$ at the exact
price, so the admission threshold of $-10^{-6}$ used by the policies
(\cref{tab:numerical-settings}) misclassifies them in the same way.
Every one of these
blocks lowers the exact cost when admitted, so the ADMM price, although
inside $[P^1_\kk,P^2_\kk]$, is not an exact shadow price of that network.
The fixed-price pricing value is an upper bound on the saving from
admission, and the bound can be loose at a kink because admitting the
block can lower the price. Re-solving the exchange problem with each
rejected block admitted gives a maximum forgone saving of 1.39\% of the
block's scale in the first case and 0.58\% in the other
two, compared with fixed-price severities of 1.8 and 4.0\%.

Warm starts misclassify about four times as many blocks as cold starts in
the checkpoint tests, and two of the three criterion failures follow a
warm start. In the 200 end-to-end episodes, where every call is
warm-started, the ADMM and exact versions of the same policy take an
identical sequence of RFQs and admissions in only 12 episodes, but they
end with the same admitted providers in 162. Finite-iteration gaps and prices
thus change which RFQs are sent and, in about one episode in five, which
providers are admitted. The recovered shares, which evaluate each admitted set at its exact
cost, still differ by no more
than 0.2 percentage points at 20, 40, and 80 RFQs, with all paired 95\%
intervals contained between $-0.3$ and $+0.7$ points. In two end-to-end
episodes and three validation cases, a firm's proximal subproblem was
re-solved with a longer Clarabel run or with OSQP after the default solve
failed. The proximal objective is strongly convex in the firm's key
injections, so its minimizing injections are unique and the fallback
changes only the numerical method used to reach them. Across the 7,000
per-class search paths in \cref{sec:experiments} and \cref{app:testbed-protocol},
evaluated at exact cost, the cost never rises from one RFQ to the next by
more than $3.3\times10^{-6}$ of its level, as the non-worsening of optional blocks in
\cref{subsec:pricing} implies, up to solver precision.

\section{Supplementary Learning Methods and Results}
\label{app:learning}

\subsection{Inputs and representations}
\label{subsec:inputs}
\label{app:learning-inputs}

The learning sample contains the 7,914 strictly linked firms of
\cref{app:data}. The directory input has 44 numeric or binary variables and
three categorical ones, the six-digit and four-digit primary industry codes
and the legal structure. The numeric variables describe the listed industry
codes, firm age and profile recency, the presence of a website, a
capabilities statement, quality standards, keywords, and special equipment,
and seven certification flags. Categories with fewer than ten firms are
pooled within each outer training fold. FMCSA and SAFER fields serve only as
outcomes, and structured location fields enter only the fold assignment
and the pre-contact screen. Place names that a firm writes in its card
are not masked.

A card combines the firm's self-description, keywords, and special-equipment
entry. Boilerplate and HTML are removed, and contact details, carrier
identifiers, the firm's legal and trade names, other incorporated firm
names, and long identifiers are masked before any card is sent to an
external model. Of the training firms, 2,035 (25.7\%) have a card, with a
median length of 416 characters. Firms without one receive a card indicator
and zeros in the text block.

Latent semantic analysis (LSA) applies a 64-dimensional truncated singular
value decomposition to fold-specific unigram and bigram TF--IDF features
\citep{deerwester1990indexing}. The embedding input reduces normalized
1,024-dimensional vectors from Amazon Titan Text Embeddings v2
(\texttt{amazon.\allowbreak titan-embed-\allowbreak text-v2:0}) to 64 dimensions in the same way.
The LLM input combines the structured capability table and the six
probabilities of \cref{subsec:match}, produced in two separate calls to
Claude Sonnet 5.5 (\texttt{us.\allowbreak anthropic.\allowbreak claude-sonnet-5-5}). The probabilities enter as logits. In the
training run, the parser finds the quoted span in the card for 99.83\% of
the table entries that claim a value. This checks that each claim is
grounded in the text, not that it is correct. Repeating both calls with
GPT-6.1 Sol (\texttt{us.\allowbreak openai.\allowbreak gpt-6.1-sol}) changes the mean
training-state AUC of the retained learners by at most 0.003.

The zero-shot references use separate probability calls that read the
directory record and, when available, the card. Invalid responses are set to
0.5. Temperature was left at the provider default, and every response is
cached, so reruns reproduce the inputs exactly. The replication package
contains the prompts, response schemas, parsing rules, and failure counts.

\subsection{Estimation and model freeze}
\label{app:learning-estimation}
\begin{table}[!htbp]
\centering
\caption{Qualified firms per 100 firms in the top tenth of each ranking
among firms with a card for which the panel's label is defined. Train
uses state-blocked five-fold cross-validation, averaged over three
seeds. In Test, each held-out region ranks its own firms, and the
results are pooled. Each top tenth contains 72 to 203 firms, so
differences of a few firms per 100 should not be overinterpreted. In
panels B to D, firms are ranked only within the group named in the panel
title, so the counts are not yields over all firms. The
last row gives each learner's mean training-state log-loss rank among the
five learners shown, over the five labels and nine combinations of input
and firm group evaluated in the training states: four inputs each for
firms with a card and for all firms, and directory fields for firms
without a card.}
\label{tab:inputs-learners}
\footnotesize
\setlength{\tabcolsep}{4.5pt}
\begin{tabular*}{\textwidth}{@{\extracolsep{\fill}}llrrrrrrrr@{}}
\toprule
 & & \multicolumn{3}{c}{Reference} & \multicolumn{5}{c}{Trained learner} \\
\cmidrule(lr){3-5}\cmidrule(l){6-10}
 & Input & Random & \makecell[r]{Industry\\code} & \makecell[r]{Zero-shot\\LLM} & LR & RF & MLP & XGBoost & LightGBM \\
\midrule
\multicolumn{10}{@{}l}{\textit{A. Active for-hire carrier}} \\
\addlinespace[1pt]
\multirow{3}{*}{Train} & Directory & \multirow{3}{*}{58} & \multirow{3}{*}{65} & \multirow{3}{*}{85} & 90 & 92 & 85 & 89 & 91 \\
 & + Titan embedding &  &  &  & 89 & 89 & 84 & 92 & 91 \\
 & + LLM reading &  &  &  & 89 & 89 & 84 & 91 & 91 \\
\cmidrule(l){2-10}
\multirow{3}{*}{Test} & Directory & \multirow{3}{*}{52} & \multirow{3}{*}{62} & \multirow{3}{*}{75} & 86 & 87 & 83 & 88 & 88 \\
 & + Titan embedding &  &  &  & 83 & 87 & 86 & 88 & 87 \\
 & + LLM reading &  &  &  & 82 & 90 & 86 & 91 & 89 \\
\addlinespace
\multicolumn{10}{@{}l}{\textit{B. Long haul among carriers}} \\
\addlinespace[1pt]
\multirow{3}{*}{Train} & Directory & \multirow{3}{*}{74} & \multirow{3}{*}{74} & \multirow{3}{*}{97} & 95 & 92 & 87 & 93 & 93 \\
 & + Titan embedding &  &  &  & 95 & 94 & 92 & 94 & 95 \\
 & + LLM reading &  &  &  & 95 & 94 & 90 & 94 & 95 \\
\cmidrule(l){2-10}
\multirow{3}{*}{Test} & Directory & \multirow{3}{*}{73} & \multirow{3}{*}{78} & \multirow{3}{*}{90} & 89 & 92 & 90 & 94 & 92 \\
 & + Titan embedding &  &  &  & 87 & 93 & 86 & 90 & 90 \\
 & + LLM reading &  &  &  & 94 & 95 & 86 & 96 & 92 \\
\addlinespace
\multicolumn{10}{@{}l}{\textit{C. Interstate among long-haul carriers}} \\
\addlinespace[1pt]
\multirow{3}{*}{Train} & Directory & \multirow{3}{*}{86} & \multirow{3}{*}{83} & \multirow{3}{*}{99} & 95 & 96 & 89 & 96 & 98 \\
 & + Titan embedding &  &  &  & 95 & 95 & 83 & 95 & 93 \\
 & + LLM reading &  &  &  & 96 & 96 & 87 & 97 & 97 \\
\cmidrule(l){2-10}
\multirow{3}{*}{Test} & Directory & \multirow{3}{*}{84} & \multirow{3}{*}{90} & \multirow{3}{*}{97} & 93 & 94 & 88 & 94 & 97 \\
 & + Titan embedding &  &  &  & 92 & 94 & 94 & 92 & 93 \\
 & + LLM reading &  &  &  & 93 & 92 & 93 & 97 & 96 \\
\addlinespace
\multicolumn{10}{@{}l}{\textit{D. Refrigerated cargo among carriers}} \\
\addlinespace[1pt]
\multirow{3}{*}{Train} & Directory & \multirow{3}{*}{8} & \multirow{3}{*}{14} & \multirow{3}{*}{33} & 17 & 19 & 13 & 17 & 18 \\
 & + Titan embedding &  &  &  & 26 & 20 & 16 & 21 & 22 \\
 & + LLM reading &  &  &  & 30 & 32 & 15 & 32 & 31 \\
\cmidrule(l){2-10}
\multirow{3}{*}{Test} & Directory & \multirow{3}{*}{8} & \multirow{3}{*}{10} & \multirow{3}{*}{30} & 11 & 19 & 12 & 20 & 20 \\
 & + Titan embedding &  &  &  & 22 & 21 & 20 & 25 & 24 \\
 & + LLM reading &  &  &  & 24 & 31 & 22 & 30 & 29 \\
\midrule
\multicolumn{5}{@{}l}{Mean log-loss rank, Train (1 = best)} & 3.47 & 2.42 & 5.00 & 2.60 & 1.51 \\
\bottomrule
\end{tabular*}
\end{table}

Jurisdictions are assigned whole to the five outer folds by a fixed greedy
rule that balances fold sizes. Within each outer training fold, every
input--learner--label combination of \cref{subsec:protocol} is tuned by a
seeded 50-trial search with the tree-structured Parzen estimator of Optuna
\citep{bergstra2011algorithms,akiba2019optuna}, minimizing mean log loss
over a stratified inner three-fold split. The first trial is a reference
configuration close to the library defaults, and the search spaces are
adapted from published tabular benchmarks
\citep{gorishniy2021revisiting,grinsztajn2022tree}. Boosted trees stop early
on the inner validation folds, and the MLP on its own validation split. The
selected hyperparameters are reused for three model seeds.

Calibration uses a separate stratified three-fold split of each outer
training set. In each fold, a nearly unregularized logistic map
\citep{platt1999probabilistic} is fitted to the base learner's held-fold
probabilities, and the three calibrated predictions are averaged.
Neither inner split is grouped by jurisdiction, so only the outer folds
hold out whole jurisdictions.
Probability quality is reported as the normalized log-loss improvement
$1-\mathrm{LL}(h)/\mathrm{LL}(b)$, where $b$ is the outer-fold training rate
in cross-validation and the full training-sample rate in a test region. The
industry-code reference assigns the smoothed training-fold positive rate of
a cell defined by a primary truckload code, any listed truckload code, and a
listed refrigerated-warehousing code, with one pseudo-count of each class.
The zero-shot probability of interstate given long haul is the clipped ratio
of the interstate and long-haul probabilities.

The fleet-size proxy is fitted among capable carriers to predict
$\log(1+U_p)$, where $U_p$ is registered power units. It uses the same
inputs, folds, search, and seeds, with regression versions of the
learners (ridge regression in place of logistic regression), inner mean
squared error, and no calibration. It is evaluated by $R^2$ against the
training-mean prediction, by Spearman correlation, and by the share of
registered trucks near the top of the capability-and-size ranking.

The contact score uses $\max\{\exp(\hat y_p),1\}$. Because $\hat y_p$
estimates the conditional mean of $\log(1+U_p)$, this is a proxy for the
conditional geometric mean of one plus the fleet, not for the
conditional mean fleet. The proxy then passes through the same public
conversion as the offers, the large-fleet share, the regional constant,
and the outage keep factor (\cref{app:testbed-offers}). No
retransformation correction is applied. A constant correction would
rescale every proxy alike, and since no proxy in the pools exceeds 57,
none would cross the 100-unit threshold of the large-fleet rule, so the
ranking would not change. The $+1$ does change it. It matches the
training target but roughly doubles the score of a one-truck firm
relative to a large fleet, and 51 to 59\% of the capable pool carriers
have one truck. Using $\max\{\exp(\hat y_p)-1,0\}$ instead changes the
recovered share at 80 RFQs by $-0.6$, $-0.1$, and $+0.6$ percentage
points, with paired 95\% intervals of $[-1.1,-0.1]$, $[-0.3,0.1]$, and
$[-0.1,1.6]$.

A preliminary training-state rule selected the simplest input--learner
pair whose bootstrap interval did not exclude parity with the best pair.
Its choice changed with the bootstrap unit, label set, and inclusion of
the fleet-size target because the leading boosted-tree pairs performed
similarly. We therefore dropped single-pair selection before any
held-out outcome was used.

The frozen evaluation retained all 16 pairs of the four inputs with
logistic regression, random forest, XGBoost, and LightGBM, each with the
fleet proxy from the same input and learner. MLP had the worst mean
training-state log-loss rank (\cref{tab:inputs-learners}). Across the
all-firm training metrics, CatBoost was within 0.004 of the best other
boosted tree but required several times its run time. The files of all
retained models and the regional prediction arrays are hashed, and the evaluation code checks a region's hash before
it reads that region's outcomes.

\subsection{Predictive results}
\label{app:learning-results}
The refrigerated gain from text is not specific to LightGBM. Across the
four retained learners, the LLM reading places 9 to 15 more refrigerated
carriers per 100 in the top tenth than directory fields alone, in both
the training states and the test regions (\cref{tab:inputs-learners}).
\Cref{tab:oor-predeclared} reports the geographic evaluation of all 16
retained pairs. The test regions contain 4,073, 1,996, and 2,266 linked
firms, of which 867, 463, and 614 have a card.

\begin{table}[!htbp]
\centering
\caption{Geographic evaluation for all firms. Entries are unweighted means of
the five per-label metrics for training-jurisdiction cross-validation (CV)
and each test region. Metrics are computed on seed-averaged predictions.
All 16 trained combinations were fixed before held-out outcomes were used.
Log-loss improvement is measured against the outer-fold training rate in CV
and the full training-sample rate in a test region. Bold marks the largest
CV value among the trained combinations.}
\label{tab:oor-predeclared}
\small
\setlength{\tabcolsep}{4.5pt}
\begin{tabular*}{\textwidth}{@{\extracolsep{\fill}}lrrrrrrrr@{}}
\toprule
 & \multicolumn{4}{c}{Mean AUC} & \multicolumn{4}{c}{Normalized log-loss improvement} \\
\cmidrule(lr){2-5}\cmidrule(l){6-9}
 & CV & Southeast & Texas & \makecell[r]{Mid-\\Atlantic} & CV & Southeast & Texas & \makecell[r]{Mid-\\Atlantic} \\
\midrule
\multicolumn{9}{@{}l}{\textit{References}} \\
\addlinespace[2pt]
Industry-code rule & 0.552 & 0.572 & 0.564 & 0.563 & 0.017 & 0.013 & 0.009 & 0.009 \\
Zero-shot LLM, directory & 0.696 & 0.679 & 0.644 & 0.654 & 0.035 & 0.010 & $-$0.018 & $-$0.011 \\
Zero-shot LLM, directory + card & 0.705 & 0.685 & 0.655 & 0.659 & 0.032 & 0.008 & $-$0.016 & $-$0.018 \\
\addlinespace
\multicolumn{9}{@{}l}{\textit{Directory}} \\
\addlinespace[2pt]
LR & 0.689 & 0.698 & 0.666 & 0.655 & 0.079 & 0.075 & 0.060 & 0.046 \\
RF & 0.700 & 0.700 & 0.679 & 0.657 & 0.090 & 0.082 & 0.067 & 0.047 \\
XGBoost & 0.701 & 0.707 & 0.682 & 0.674 & 0.090 & 0.086 & 0.074 & 0.060 \\
LightGBM & 0.705 & 0.708 & 0.684 & 0.677 & 0.093 & 0.087 & 0.074 & 0.059 \\
\addlinespace
\multicolumn{9}{@{}l}{\textit{+ LSA}} \\
\addlinespace[2pt]
LR & 0.686 & 0.691 & 0.666 & 0.656 & 0.076 & 0.069 & 0.059 & 0.051 \\
RF & 0.698 & 0.700 & 0.679 & 0.665 & 0.088 & 0.082 & 0.070 & 0.049 \\
XGBoost & 0.704 & 0.705 & 0.685 & 0.671 & 0.092 & 0.086 & 0.075 & 0.060 \\
LightGBM & 0.706 & 0.703 & 0.687 & 0.677 & 0.093 & 0.085 & 0.078 & 0.061 \\
\addlinespace
\multicolumn{9}{@{}l}{\textit{+ Titan embedding}} \\
\addlinespace[2pt]
LR & 0.689 & 0.692 & 0.674 & 0.658 & 0.078 & 0.073 & 0.062 & 0.051 \\
RF & 0.702 & 0.699 & 0.678 & 0.655 & 0.091 & 0.081 & 0.066 & 0.049 \\
XGBoost & 0.704 & 0.705 & 0.685 & 0.670 & 0.092 & 0.086 & 0.075 & 0.062 \\
LightGBM & 0.706 & 0.707 & 0.689 & 0.672 & 0.095 & 0.087 & 0.077 & 0.062 \\
\addlinespace
\multicolumn{9}{@{}l}{\textit{+ LLM reading}} \\
\addlinespace[2pt]
LR & 0.694 & 0.701 & 0.668 & 0.664 & 0.081 & 0.081 & 0.061 & 0.052 \\
RF & 0.707 & 0.706 & 0.680 & 0.669 & 0.096 & 0.089 & 0.068 & 0.060 \\
XGBoost & 0.709 & 0.710 & 0.686 & 0.681 & 0.096 & 0.092 & 0.079 & 0.065 \\
LightGBM & \textbf{0.715} & 0.710 & 0.685 & 0.681 & \textbf{0.100} & 0.092 & 0.076 & 0.063 \\
\bottomrule
\end{tabular*}
\end{table}

\subsection{Secondary diagnostics}
\label{app:learning-diagnostics}

For the recommended pair, the fleet-size proxy reaches an $R^2$ of 0.28 to
0.37 and a Spearman correlation of 0.32 to 0.54 in the test regions. Ranking all firms of a region by predicted carrier
probability alone places 21 to 26\% of the registered trucks of capable
carriers in the top tenth. Multiplying that probability by the
fleet-size proxy raises the share to 38 to 41\%. Ranking by registry
capability and fleet size would place 72 to 76\% there.

The zero-shot LLM ranks long-haul, interstate, and refrigerated carriers
among firms with a card about as well as the recommended model, with 90,
97, and 30 qualified providers per 100 in the top decile against 92, 96,
and 29. It is weaker for active for-hire status, 75 against 89
(\cref{tab:prediction-summary}). Its raw probabilities transfer less well
than its rankings. In Texas and the Mid-Atlantic, their log loss is worse
than that of a constant base-rate predictor (\cref{tab:oor-predeclared}).

A secondary diagnostic, specified before its results were read and run in
the training states, asks whether an LLM can rank refrigerated capability
from cards alone, without any refrigerated label. Among the 1,175
card-holding carriers with a SAFER record, 89 of them refrigerated, the raw
probabilities of GPT-6.1 Sol and Claude Sonnet 5.5 both reach an AUC of
0.785, against 0.782 for the supervised LLM--LightGBM model trained with
every label. The GPT probabilities are poorly calibrated, with a log loss of
0.334 against 0.268 for a constant base rate, while the Sonnet log loss,
0.266, is at the base rate.

In this diagnostic, about 200 labels were enough to calibrate the GPT
score. A logistic map fitted to 200 labels drawn, stratified by label,
from the 4,901 capable carriers with a SAFER record
lowers its log loss on the held-out card holders to 0.243, and the fold-mean
AUC stays at 0.778, its value with all labels. Samples of 25 to 100 labels
contain only 2 to 8 refrigerated carriers, the fitted map is unreliable, and
the fold-mean AUC falls to 0.64 to 0.69. With all labels, the log loss is
0.229. Zero-shot scores can therefore order initial contacts for a new
need type, but a small labeled sample is needed to calibrate them, and
verification still determines eligibility. The zero-shot policies of
\cref{subsec:decision-value} use the raw Sonnet probabilities, and the
recommended policy uses calibrated LightGBM probabilities with the
Sonnet outputs among its inputs.

\end{document}